\documentclass[a4paper,10pt]{amsart}
\usepackage{
   amssymb
  ,amsmath
  ,stmaryrd
  ,mathrsfs
  ,enumitem
  ,leftidx
  ,mathtools
  ,tikz
  ,url
  ,newclude
  ,mathtools
  ,mathbbol
  ,mathabx
  ,color
  ,mathdots
  ,newtxtext
  }
  
\usepackage{vmargin}
\setmarginsrb%
{30mm}
{25mm}
{30mm}
{18mm}
            {0mm}
            {10mm}
            {0mm}
            {10mm}

  \usepackage[color,all,cmtip]{xy}

\newcommand\SmallMatrix[1]{{\text{\scriptsize$
\arraycolsep=0.3\arraycolsep\ensuremath{\begin{pmatrix}#1\end{pmatrix}}$}}}

\def\O{\mathcal O}
\def\DG{{\sc dg}}
\def\dg{\mathrm{dg}}
\def\CE{{\sc ce}}
\def\M{\mathcal{M}}
\def\Ext{\mathrm{Ext}}

\def\F{\mathcal{F}}
\def\Y{\mathcal{Y}}
\def\I{\mathcal{I}}
\def\Q{\mathrm{Q}}

\def\Mod{\mathrm{Mod}}

\def\Inj{\mathrm{Inj}}

\def\codim#1{\text{-}\mathrm{codim}(#1)}
\def\dim#1{\text{-}\mathrm{gl.dim}(#1)}
\def\Q{\mathbf Q}
\def\S{\mathbf S}
\def\y{\mathbf{y}}

\def\two{\mathbf{2}}
\def\Inj{\mathrm{Inj}}
\def\D{\mathcal D}
\def\I{\mathcal I}
\def\E{\mathcal E}

\def\T{\mathcal T}
\def\C{\mathcal C}

\def\hom{\mathrm{Hom}}

\def\Z{\mathbb Z}

\def\A{\mathcal A}

\def\Ker{\mathrm{Ker}}
\def\N{\mathbb N}
\def\id{\mathrm{id}}
\def\Ch{\mathrm{Ch}}
\def\Tot{\mathrm{Tot}}

\def\Ob{\mathrm{Ob}}

\def\Add{\mathrm{Add}}
\def\Ab{\mathrm{Ab}}
\def\G{\mathcal{G}}
\def\W{\mathcal W}

\def\op{{\mathrm{op}}}

\def\S{\mathbf S}

\def\coker{\mathrm{Coker}}
\def\X{\mathcal X}
\def\cone{\mathrm{cone}}

\def\Prod{\mathrm{Prod}}
\def\Im{\mathrm{Im}}

\def\Qcoh{\mathrm{Qcoh}}
\def\gdim{\mathrm{gl.dim}}
\def\pdim{\mathrm{p.dim}}

\numberwithin{equation}{section}
\newtheorem{thm}{Theorem}[section]
\newtheorem{cor}[thm]{Corollary}
\newtheorem{prop}[thm]{Proposition}
\newtheorem{lem}[thm]{Lemma}
\newtheorem{lemdef}[thm]{Lemma-Definition}

\newtheorem{quest}[thm]{Question}

\theoremstyle{plain}
\newtheorem{defn}[thm]{Definition}

\theoremstyle{plain}
\newtheorem{rmk}[thm]{Remark}
\newtheorem{eg}[thm]{Example}

\newtheorem{con}[thm]{Construction}

\title[A poisonous example to unbounded resolutions]{A poisonous example to explicit resolutions\\ of unbounded complexes}

\author{Dolors Herbera}
\address{Dept.\ de Matem\`atiques,
	Universitat Aut\`onoma de Barcelona
	(08193) Bellaterra, Barcelona, Spain\\
	\indent
	Centre de Recerca Matem\`atica
	(08193) Bellaterra, Barcelona, Spain.}
\email{Dolors.Herbera@uab.cat}

\author{Wolfgang Pitsch}
\address{Dept.\  de Matem\`atiques,
	Universitat Aut\`onoma de Barcelona (08193) Bellaterra, Barcelona, Spain.}
\email{Wolfgang.Pitsch@uab.cat}

\author{Manuel Saor\'in}
\address{Facultad de Matem\'aticas, Universidad de Murcia,
	Avda.\ Teniente Flomesta 5 (30003) Murcia, Spain.}
\email{MSaorinc@um.es}

\author{Simone Virili}
\address{Dept.\  de Matem\`atiques,
	Universitat Aut\`onoma  de Barcelona
	(08193) Bellaterra, Barcelona, Spain.}
\email{Simone.Virili@uab.cat  or Virili.Simone@gmail.com}
\thanks{The first author was supported by the Spanish State Research Agency, through the Severo Ochoa and María de Maeztu Program for Centers and Units of Excellence in R\& D (CEX2020-001084-M). The first and the fourth authors were partially supported by the projects MIMECO  PID2020-113047GB-I00 and MIMECO  PID2023-147110NB-I00 financed by the Spanish Government. The first, second, and fourth authors are  members of the  LIGAT (Interactions Lab between Geometry, Algebra, and Topology), Research Group of the Generalitat de Catalunya number 2021 SGR 01015. The second author was supported  by the FEDER/MEC grant ``Homotopy theory of combinatorial and algebraic structures'' PID2020-116481GB-I00. The third author was supported by the grant PID2020-113206GB-I00, funded by MCIN/AEI/10.13039/501100011033, and the project 22004/PI/22 of the ``Fundaci\'on S\'eneca'' (Murcia).}

\date{\today}

\begin{document}

\begin{abstract}
	We show that various methods for explicitly building resolutions of unbounded complexes in fact fail when applied to a rather simple and explicit complex. We show that one way to rescue these methods is to assume Roos (Ab.4$^*$)-$k$ axiom, which we adapt to encompass also resolutions in the framework of relative homological algebra. In the end we discuss the existence of model structures for relative homological algebra for unbounded complexes under the relative (Ab.4$^*$)-$k$ condition, and present a variety of examples where our results apply.
\end{abstract}

\subjclass[2020]{Primary 18G25, Secondary 18E10, 18G10.}

\maketitle
\setcounter{tocdepth}{1}
\tableofcontents

\section*{Introduction}

After the publication of Verdier's seminal thesis in 1967 (see \cite{tesis_Verdier} for the later complete publication), the study of the derived categories of Abelian categories has received considerable attention and  found widespread applications across many areas of abstract mathematics. It is important to underline that the initial focus was mostly restricted to bounded or half-bounded derived categories, where it was already known that, when starting with an Abelian category with enough injectives (or projectives),  each lower (upper) bounded complex is quasi-isomorphic to a lower (upper) bounded complex of injectives (projectives). Additionally, the morphisms in the derived category between two such bounded complexes of injectives (projectives) correspond precisely to the homotopy classes of chain maps between them. This suffices to ensure that the corresponding half-bounded derived categories are locally small, i.e., they have small $\hom$-sets. In contrast, the extension of this result to unbounded derived categories presented a considerable challenge.

One of the main difficulties that arises in the unbounded setting is the necessity to find, for each complex, a quasi-isomorphism to a complex of injectives that, furthermore, is $K$-injective (or, dually, from a $K$-projective complex of projectives). In other words, one needs to build \DG-injective (\DG-projective) resolutions of unbounded complexes. Even so, such constructions were not available at the time, even for commonly used Grothendieck categories, like categories of modules or of quasi-coherent sheaves. The first breakthrough in this direction was pioneered by Spaltenstein \cite{Spal} in 1988, who was able to build a \DG-projective resolution for each $X^\bullet\in \Ch(\Mod (R))$, with $R$ a ring, and a \DG-injective resolution for each $X^\bullet\in \Ch(\A)$, where $\mathcal{A}=\mathrm{Sheaf}_\Theta(X)$ is the category of sheaves of modules, over a sheaf of rings  $\Theta$ on a topological space $X$.  These  constructions have the advantage of not being merely existence results, as they rely on very explicit constructions of resolutions. 

The history of the proof of the existence of  \DG-injective resolutions of unbounded complexes for general Grothendieck categories and of the techniques used to obtain the result is complicated. In 2000, Alonso, Jeremias and Souto \cite[Theorem~5.4]{tarrio2000localization} proved it as a consequence of the existence of \DG-injective resolutions for module categories combined with the Gabriel and Popescu embedding Theorem via Bousfield localization. In 2003, C.~Serp\'e, in \cite[Theorem~3.13]{serpe}, gives an alternative proof by extending Spaltenstein's approach to general Grothendieck categories. Serp\'e's proof has a variation of the small object argument as a central point.

The first published proof of the existence of a model structure over a Grothendieck category is due to Hovey \cite{7fd5ffe3-5ff1-3e28-b6fb-e47f75d8437a} and it was published in 2001.  A very similar argument had already been outlined by Joyal in a private correspondence with Grothendieck dated back to 1984 (nowadays, the letter is available at \cite{letter_joyal_groth}). Both authors use variations of the small object argument to establish the existence of \DG-injective resolutions in complete generality. The resolutions produced by these methods are non-explicit in nature, rendering them impractical for the computation of \DG-resolutions in concrete situations.

At this point, the following questions  arise naturally: 

\smallskip\noindent {\bf Questions.}\ {\em
Let $\A$ be a complete Abelian category with enough injectives, and take $X^\bullet\in \Ch(\A)$.
\begin{enumerate}
\item[\rm (Q.1)] Is it always possible to find a \DG-injective resolution for $X^\bullet$?
\item[\rm (Q.2)] If $X^\bullet$ has a \DG-injective resolution, can one build such a resolution explicitly (i.e., with a step-by-step procedure that allows for a complete description of the final result, thus excluding, in principle, all constructions based on the small object argument)?
\end{enumerate} }

In the literature, one can find currently two separate papers -- \cite{saneblidze2007derived} and \cite{yang2015question} -- each claiming to give positive answers to both questions. Unfortunately, both constructions fail when applied to the ``poisonous example'' mentioned in the title, which is an unbounded complex $X^\bullet\in \Ch(\G)$, where $\G$ is a localization of $\Mod(R)$, with $R$ being Nagata's classical example of a commutative Noetherian ring with infinite Krull dimension (see \cite[Appendix~A.1]{Nagata}). Let us clarify that we are not claiming that the above questions have a negative answer, and so we do not discard the possibility that the derived category $\mathcal{D}(\mathcal{A})$ is locally small whenever $\A$ is a complete Abelian category with enough injectives. For example, for our ``poisonous example'' $X^\bullet$, both (Q.1) and (Q.2) can be answered in the positive. Indeed, take a \DG-injective resolution $X^\bullet\to E^\bullet$ in $\Ch(\Mod\text{-}R)$ (e.g., using Spaltenstein's construction); the inclusion of the  torsion part $T^\bullet\leq E^\bullet$ is degree-wise split (and, arguably, easy to describe), so $E^\bullet/T^\bullet$ is a complex of torsion-free injectives. In particular, this quotient lives in $\Ch(\G)$, and $X^\bullet\to E^\bullet/T^\bullet$ is a \DG-injective resolution (see Section~\ref{sec_three} for details).

On the positive side, the second named author together with Chachólski, Neeman and Scherer~\cite{zbMATH06915995} gave affirmative answers to both questions, adapting Spaltenstein's construction, under the extra hypothesis that the category $\mathcal{A}$ is (Ab$4^*$)-$k$ in the sense of Roos \cite[Definition~1.1]{Roos}, for some $k\in \N$, i.e., the $n$-th derived functor of the product vanishes for all $n>k$ (e.g., for $\A$ to be (Ab$4^*$)-$0$ it means precisely that products are exact in $\A$). More generally, given an injective class (i.e., a preenveloping class  closed under summands) $\I\subseteq\A$, they proved  that any complex $X^\bullet\in\Ch(\A)$ has an $\I$-fibrant replacement, provided $\A$ satisfies the (Ab.4$^*$)-$\I$-$k$ condition -- a version of Roos' (Ab$4^*$)-$k$ relative to $\I$. To recover the results about \DG-injective resolutions, it is then enough to take $\I=\Inj(\A)$, the class of all injective objects in $\A$. The existence of $\I$-fibrant replacements, when combined with results of \cite{CH}, allows us to build and study a suitable model structure in $\Ch(\mathcal{A})$ whose homotopy category is the $\mathcal{I}$-derived category $\mathcal{D}(\mathcal{A};\, \mathcal{I})$ from \cite{zbMATH06915995,CH}; the usual derived category $\D(\A)$ is equivalent to $\D(\A;\, \Inj(\A))$. 

\medskip
The paper is organized as follows: In Section~\ref{sec_one}, after fixing the needed conventions and notations about cochain complexes, we review some basic definitions and  results regarding weak factorization systems, model categories, cotorsion pairs and Hovey's correspondence. In Section~\ref{sec_two} we show that, whenever $\A$ is a complete (Ab.$4^*$)-$k$ Abelian category with enough injectives, any $X^\bullet\in\Ch(\A)$ has a \DG-injective resolution. Although this may be deduced from more general results in \cite{zbMATH06915995}, we include a short direct proof. The ``poisonous example'' mentioned in the title is introduced and studied in Section~\ref{sec_three}, where it is  used to show that the (Ab.$4^*$)-$k$ condition is really necessary for the constructions of Section~\ref{sec_two}. Section~\ref{sec_four} is devoted to injective Cartan-Eilenberg (\CE-)resolutions: after reviewing the definitions and some classical results, we prove that, to build a \DG-injective resolution of an $X^\bullet\in\Ch(\A)$ (with $\A$ as in Section~\ref{sec_two}) one can just totalize a \CE-resolution of $X^\bullet$. Furthermore, this strategy fails when applied to our ``poisonous example'', showing that (Ab.$4^*$)-$k$ is necessary also in this approach. In Section~\ref{sec_five} we briefly review the construction of \cite{saneblidze2007derived} and we show that it fails to produce a \DG-injective resolution for the ``poisonous example''. Similarly,  in Section~\ref{sec_six}, we review the construction of \cite{yang2015question} and we show that it does not work when applied to our ``poisonous example'', also observing that the gap in the argument disappears if one assumes that the ambient category is (Ab.$4^*$). In Section~\ref{sec_seven}, we concentrate on the construction of the so-called $\I$-injective model structure on $\Ch(\A)$, relative to a given injective class $\I\subseteq \A$. After recalling some general results from \cite{CH}, we describe the construction of Spaltenstein  towers of partial $\I$-injective resolutions given in \cite{zbMATH06915995}. In the last part of the section, we prove a number of results related to the (Ab.4$^*$)-$\I$-$k$ condition, showing that it implies the existence of the $\I$-injective model structure on $\Ch(\A)$. In Section~\ref{ex_and_app_sec}, we include a number of examples and applications of the formalism introduced in Section~\ref{sec_seven}. In particular, we first consider  those injective classes with the additional property of being cogenerating and, as an application, we give a partial positive answer to Gillespie's question; we then study the injective classes contained in $\Inj(\A)$, and the injective class of the pure-injective objects in a locally finitely presented Grothendieck category. Moreover, we introduce a notion of global dimension for $\A$ relative to an injective class $\I$; if this invariant is finite, then $\A$ is (Ab.4$^*$)-$\I$-$k$. Using the relative global dimension, we give a new characterization of the $n$-tilting cotorsion pairs. Finally, we give an application to categories of quasi-coherent sheaves over a scheme.

\medskip\noindent
{\bf Acknowledgments.} The initial results that persuaded us to write this paper date back to 2019, during an ``Intensive Research Period in Small Groups'' at the CRM (Barcelona), supported by the Barcelona Graduate School of Mathematics (BGSMath) as a Mar\'\i a de Maeztu Excellence Centre 2015--2019.\\
We would like to thank Sergio Estrada, for several fruitful conversations on the topics of this paper and, in particular, for sharing with us his doubts regarding \cite{yang2015question}, Leo Alonso and Ana Jeremías, for their remarks on the history of the existence of \DG-injective resolutions over general Grothendieck categories, and Leonid Positselski, for bringing  \cite{positselski2024roos} to our attention, and for his  remarks and insights on quasi-coherent sheaves; this interaction motivated us to add Section~\ref{geometry_subs} to the paper.
\\
Finally, we are grateful to two anonymous referees for their insightful comments and constructive criticism, that helped us to improve the exposition of our results.

\section{Preliminaries and notations}\label{sec_one}
\noindent
{\bf Cochain complexes.} 
Given an Abelian category $\A$, we denote by $\Ch(\A)$ the category of unbounded cochain complexes and maps of complexes. In particular, given
\[
X^\bullet:=(\xymatrix@C=18pt{\cdots \ar[r]& X^{n-1}\ar[r]^-{d^{n-1}}&X^n\ar[r]^-{d^n}&X^{n+1}\ar[r]&\cdots})\in\Ch(\A),
\]
we denote the $n$-th {\bf coboundaries}, {\bf cocycles}, and {\bf cohomology} (with $n\in \Z$) of $X^\bullet$, respectively, by $B^n(X^\bullet):=\Im(d^{n-1})$, $Z^n(X^\bullet):=\Ker(d^{n})$, and $H^n(X^\bullet):=Z^n(X^\bullet)/B^n(X^\bullet)$. Moreover, we denote by $\Sigma\colon \Ch(\A)\to \Ch(\A)$ the ``shift to the left'' auto-equivalence of $\Ch(\A)$, that is,  the $n$-th component of $\Sigma X^\bullet$ is $(\Sigma X^\bullet)^n:=X^{n+1}$, while its $n$-th differential is $-d^{n+1}\colon X^{n+1}\to X^{n+2}$ (i.e., the $(n+1)$-th differential of $X^\bullet$, with a minus sign).

A given morphism of complexes $\phi^\bullet\colon X^\bullet \to Y^\bullet$ is said to be a {\bf quasi-isomorphism} if $H^n(\phi^\bullet)$ is an isomorphism for all $n\in\Z$. We also say that a complex $X^\bullet$ is {\bf exact} (or {\bf acyclic}) if $H^n(X^\bullet)=0$, for all $n\in\Z$ or, equivalently, $X^\bullet\to 0$ is a quasi-isomorphism.

For each $k\in\Z$, we denote by $\Ch^{\geq k}(\A)$ (resp., $\Ch^{\leq k}(\A)$) the full subcategory of  the $X^\bullet\in\Ch(\A)$ such that $X^n=0$, for all $n<k$ (resp., $n>k$). Moreover, $\Ch^+(\A):=\bigcup_{n\in\Z}\Ch^{\geq n}(\A)$ denotes the full subcategory of bounded below complexes. Given $k\in\Z$, the inclusion  $\Ch^{\geq k}(\A)\to \Ch(\A)$ has a left adjoint, called the {$k$-th \bf left truncation functor} $\tau^{\geq k}\colon \Ch(\A)\to \Ch^{\geq k}(\A)$ such that:
\[
\tau^{\geq k}(X^\bullet):=(\xymatrix@C=18pt{\cdots \ar[r]& 0\ar[r]&X^k/B^k(X^\bullet)\ar[r]^-{\bar d^k}&X^{k+1}\ar[r]^-{d^{k+1}}&\cdots})\in\Ch^{\geq k}(\A).
\]
Similarly, the inclusion  $\Ch^{\leq k}(\A)\to \Ch(\A)$ has a right adjoint, called the {$k$-th \bf right truncation functor} $\tau^{\leq k}\colon \Ch(\A)\to \Ch^{\leq k}(\A)$ such that:
\[
\tau^{\leq k}(X^\bullet):=(\xymatrix@C=18pt{\cdots \ar[r]&X^{k-2}\ar[r]^-{d^{k-2}}& X^{k-1}\ar[r]^-{d^{k-1}}&Z^k(X^\bullet)\ar[r]&0\ar[r]&\cdots})\in\Ch^{\leq k}(\A).
\]
 Given an object $A\in\A$ and $k\in\Z$ we let $S^k(A)$ denote the following {\bf stalk complex}:
\[
S^k(A):=(\xymatrix@C=18pt{\cdots \ar[r]& 0\ar[r]&A\ar[r]&0\ar[r]&\cdots})\in\Ch^{\geq k}(\A)\cap\Ch^{\leq k}(\A);
\] 
for $n=0$ we sometimes denote $S^0(A)$ simply by $A$.
Moreover, we denote by $D^k(A)$ the  {\bf disk complex}:
\[
D^k(A):=(\xymatrix@C=18pt{\cdots \ar[r]& 0\ar[r]&A\ar[r]^-{\id_A}&A\ar[r]&0\ar[r]&\cdots})\in\Ch^{\geq k}(\A)\cap\Ch^{\leq k+1}(\A).
\] 
Observe that $D^k(A)$ is an exact complex, that is, $H^i(D^k(A))=0$ for all $i\in \Z$. Moreover, as co/products of isomorphisms are isomorphisms, any co/product of disk complexes is still exact. Finally, observe that, if $Y^\bullet$ is such a co/product of disk complexes, and $X^\bullet\in \Ch(\A)$ is any complex, then $H^i(X^\bullet\oplus Y^\bullet)\cong H^i(X^\bullet)\oplus H^i(Y^\bullet) \cong H^i(X^\bullet)$, for all $i\in \Z$. This simple observation is often useful when we need to modify a given complex $X^\bullet$ without changing its cohomologies.

\smallskip
\noindent
{\bf Weak factorization systems.} Given a category $\A$, denote by $\A^\two$ its category of morphisms. Given $\phi$ and $\psi\in \A^\two$, one says that $\phi$ is {\bf left weakly orthogonal} to $\psi$ (and, equivalently, that $\psi$ is {\bf right weakly orthogonal} to $\phi$), in symbols $\phi\boxslash\psi$, if for any solid commutative square 
\[
\xymatrix@R=20pt@C=40pt{
X_0\ar[d]_{\phi}\ar[r]^{a}&Y_0\ar[d]^{\psi}\\
X_1\ar[r]_{b}\ar@{.>}[ur]|{\exists \,d}&Y_1
}
\]
there is a (not necessarily unique) diagonal $d\colon X_1\to Y_0$ such that $a=d\circ \phi$ and $b=\psi\circ d$.

Given a class of objects $\mathcal X\subseteq \A^\two$, we denote by
\[
{}^{\boxslash}\mathcal X:=\{l\in \A^\two:l\boxslash x,\ \forall x\in \mathcal X\} \quad \text{and} \quad \mathcal X^{\boxslash}:=\{r\in \A^\two:x\boxslash r,\ \forall x\in \mathcal X\},
\]
the classes of all morphisms in $\A$ that are weakly left or right orthogonal, respectively, to  $\mathcal X$. Using another common terminology, $l\in {}^{\boxslash}\mathcal X$ (resp., $r\in \mathcal X^{\boxslash}$), if it has the {\bf left} ({\bf right}) {\bf lifting property} with respect to $\mathcal X$. The following general lemma is often useful in practice:
\begin{lem}[{\cite[Lemma~11.1.4]{riehl2014categorical}}]\label{riehl_lem}
Any class of arrows of the form ${}^{\boxslash}\mathcal X$ is closed under coproducts, push-outs, transfinite composition, retracts, and it contains the isomorphisms. The class $\mathcal X^{\boxslash}$ has  dual  properties.
\end{lem}
In particular, we will need the following  consequence of the above lemma: Suppose that $(A_n)_{n\in\N}$ is an inverse system in $\A$ such that $\varprojlim_\N A_n$ exists in $\A$. If $(A_{n+1}\to A_n)\in \mathcal X^{\boxslash}$ for all $n\in\N$, then the canonical map $\varprojlim_\N A_n\to A_0$ is also in $\mathcal X^{\boxslash}$.

\smallskip
A pair of subclasses $(\mathcal X,\mathcal Y)$ of $\A^\two$ is a {\bf weak factorization system} if the following hold:
\begin{enumerate}
\item[{\rm ({\sc wfs}.1)}] $\mathcal X^{\boxslash}=\mathcal Y$ and $\mathcal X={}^{\boxslash}\mathcal Y$;
\item[{\rm ({\sc wfs}.2)}] each  $\phi\in\A^\two$ can be written as $\phi=y\circ x$, with $x\in \mathcal X$ and $y\in \mathcal Y$. 
\end{enumerate}

\smallskip
\noindent
{\bf Model categories.}
A {\bf model structure} on a bicomplete category $\M$ is a triple $(\W,\C,\F)$ of classes of maps, called respectively {\bf weak equivalences}, {\bf cofibrations} and {\bf fibrations}, such that:
\begin{enumerate}
\item[{\rm ({\sc ms}.1)}] $\W$ contains all the isomorphisms and has the $2$-out-of-$3$ property;
\item[{\rm ({\sc ms}.2)}] both $(\C\cap \W,\F)$ and $(\C,\F\cap \W)$ are weak factorization systems.
\end{enumerate}
In this case, the quadruple $(\M,\W,\C,\F)$ is called a {\bf model category}. The morphisms in $\C\cap \W$ and $\F\cap \W$ are called {\bf trivial cofibrations} and {\bf trivial fibrations}, respectively. Furthermore, an object $F\in \M$ is said to be {\bf fibrant} if the terminal map $F\to *$ is a fibration while, for $\M$ pointed, we say that $F$ is {\bf trivial} if the trivial map $0\to F$ is a weak equivalence. Finally, a {\bf fibrant replacement} for an object $X\in \M$, is a weak equivalence $X\to F$, with $F$ fibrant. Cofibrant objects and cofibrant replacements are defined dually. 

\begin{rmk}\label{rem_lim_of_fib}
Suppose that we are in the following situation: we have a bicomplete category $\M$ and two classes of maps: $\W$, whose members we call weak equivalences, and $\C$, whose members we call cofibrations, and suppose that we want to use them to construct a model structure on $\M$. In this setting, we are forced to take $\F:=(\C\cap \W)^{\boxslash}$ as a class of fibrations. Even if we do not know whether $(\W,\C,\F)$ is a model structure, let us abuse terminology and call an object $X\in \M$ fibrant if $X\to *$ belongs to $\F$. Consider an inverse system $(F_n)_{\N}$ of fibrant objects and suppose that $(F_{n+1}\to F_n)\in \F$ for all $n\in\N$, then $F:=\varprojlim_\N F_n$ is a fibrant object. 
To see this, consider the inverse system $(F'_n)_{\N}$, where $F'_0=*$ is the terminal object, and $F'_{n+1}:=F_n$ for all $n\in\N$. Then, as $\F=(\C\cap \W)^{\boxslash}$, the unique map $\varprojlim_\N F_n\cong \varprojlim_\N F_n'\to F'_0=*$ still belongs to $\F$, as discussed briefly after Lemma~\ref{riehl_lem}. 
\end{rmk}

\smallskip
\noindent
{\bf Cotorsion pairs.}
Let $\A$ be an Abelian category, and let $\C\subseteq \A$ be a subclass. We say that $\C$ is ({\bf co}){\bf generating} if any object of $\A$ is isomorphic to a quotient (resp., subobject) of an object in $\C$. Given $A,\,B\in \A$ and $n\geq 1$, we denote by $\Ext^n_\A(A,B)$ the usual group of equivalence classes of $n$-extensions; in general, we do not assume that $\A$ has enough injectives nor projectives, so we cannot exclude the possibility that the underlying sets of these groups of extensions are ``big'', see \cite[Ch.XII.5]{MR1344215} and \cite[I.6]{McL} for  details.
 
For any  $I\subseteq \Z$, we use the following notations for the corresponding right and left orthogonal classes of $\C$:
\[
\C^{\perp_I}:=\{A\in \A:\Ext_\A^i(C,A)=0,\text{ for all $i\in I$ and all $C\in \C$}\} \quad \text{and} 
\]
\[ 
{}^{\perp_I}\C:=\{A\in \A:\Ext_\A^i(A,C)=0,\text{ for all $i\in I$ and all $C\in \C$}\}.
\]
Moreover, let  $\C^{\perp_n}:=\C^{\perp_{\{n\}}}$ (for any $n\in\Z$), and $\C^{\perp}:=\C^{\perp_{0}}$, with analogous conventions on the left. 
\begin{rmk}\label{rem_on_hereditary}
Let $\A$ be an Abelian category, and let $\C\subseteq \A$ be a subclass. Then:
\begin{itemize}
    \item ${}^{\perp_{>0}}\C$ is closed under kernels of epimorphisms. In fact, if $0\to K\to X_1\to X_2\to 0$ is a short exact sequence in $\A$ with  $X_1,\,X_2\in{}^{\perp_{>0}}\C$ then $\Ext_\A^n(K,C)=0$, for any $n>0$ and any $C\in \C$, as such an $\Ext$-group fits in the following exact sequence:
    \[
    0=\Ext_\A^n(X_1,C)\to\Ext_\A^n(K,C)\to\Ext_\A^{n+1}(X_2,C)=0.
    \]
    \item If $\C$ is generating and closed under kernels of epimorphisms, then $\C^{\perp_{>0}}=\C^{\perp_1}$. Indeed, it is enough to show that, for all $n\geq 1$, if $X\in \C^{\perp_n}$, then $X\in \C^{\perp_{n+1}}$: let $C\in \C$ and consider an extension $\epsilon:=[0\to X\to A_1\to \dots\to A_n\to A_{n+1}\to C\to 0]\in\Ext_\A^{n+1}(C,X)$. If $C'\twoheadrightarrow A_{n+1}$ is an epimorphism, with $C'\in \C$, we get a new representative for the extension $\epsilon$  by taking a pull-back:
    \[
    \xymatrix@R=15pt{
    \epsilon:=&[0\ar[r]& X\ar[r]\ar@{=}[d]& A_1\ar[r]\ar@{=}[d]& \dots\ar[r]& A_n\ar[r]\ar@{<-}[d]\ar@{}[dr]|{P.B.}& A_{n+1}\ar[r]\ar@{<<-}[d]& C\ar[r]\ar@{=}[d]& 0]\\
    \epsilon=&[0\ar[r]& X\ar[r]& A_1\ar[r]& \dots\ar[r]& A'_n\ar[r]& C'\ar[r]& C\ar[r]& 0].
    }
    \] 
    To conclude, consider $\xi:=[0\to K\to C'\to C\to 0]\in \Ext_\A^1(C,K)$, where $K\in \C$ by hypothesis, and $\epsilon':=[0\to X\to A_1\to \dots\to A'_n\to K\to 0]\in\Ext_\A^{n}(K,X)$: since $X\in \C^{\perp_n}$ by assumption, then $\epsilon'=0$, and so $\epsilon=0$, as it is the Yoneda product of $\xi$ and $\epsilon'$ (see \cite[Ch.VIII.4]{MR1344215}).
\end{itemize}
\end{rmk}
A pair $(\X,\Y)$ of classes in $\A$ is said to be a {\bf cotorsion pair} if $\X={}^{\perp_1}\Y$, and $\Y=\X^{\perp_1}$. A cotorsion pair $(\X,\Y)$ is {\bf left} (resp., {\bf right}) {\bf complete} if for all $A\in\A$ there is an exact sequence
\[
0\longrightarrow Y\longrightarrow X\longrightarrow A\longrightarrow 0\quad (\text{resp., }0\longrightarrow A\longrightarrow Y'\longrightarrow X'\longrightarrow 0),
\]
with $X,X'\in \X$ and $Y,Y'\in \Y$, called a {\bf left} (resp., {\bf right}) $(\X,\Y)${\bf -approximation} of $A$. Moreover, we say that $(\X,\Y)$ is {\bf complete} if it is left and right complete.
\begin{rmk}\label{rem_on_completeness}
Let $\A$ be an Abelian category. Then, a cotorsion pair $(\X,\Y)$ in $\A$ is complete if, and only if, it is left-complete and the class $\Y$ is cogenerating. The non-trivial implication can be verified as follows: given $A\in \A$, use that $\Y$ is cogenerating to find a short exact sequence $0\to A\to Y'\to B\to0$, with $Y'\in \Y$. Consider now a left $(\X,\Y)$-approximation $0\to Y''\to X\to B\to 0$; taking the pull-back of the first sequence along $X\to B$, we get the desired right approximation $0\to A\to Y\to X\to 0$, where $Y\in \Y$ since it is an extension of $Y''$ by $Y'$, and they are both objects in $\Y$.
\end{rmk}
A cotorsion pair $(\X,\Y)$ in $\A$ is called {\bf hereditary} if $\Ext_\A^n(X,Y)=0$, for all $X\in \X$, $Y\in \Y$,  $n\geq 1$.
\begin{lem}\label{her_com_cot_pair_lem}
Let $\A$ be an Abelian category, and let $\X$ and $\Y\subseteq \A$ be two subclasses such that $\X$ is generating and $\Y$ is cogenerating. Then, the following are equivalent:
\begin{enumerate}
    \item $\X^{\perp_{>0}}=\Y$, and $^{\perp_{>0}}\Y=\X$;
    \item $(\X,\Y)$ is a cotorsion pair and $\X$ is closed under kernels of epimorphisms;
    \item $(\X,\Y)$ is a cotorsion pair and $\Y$ is closed under cokernels of monomorphisms.
\end{enumerate}
In this case, $(\X,\Y)$ is a hereditary cotorsion pair, which is left-complete if and only if it is right-complete, if and only if it is complete.
\end{lem}
\begin{proof}
The implications ``(1)$\Rightarrow$(2,3)'' follow by Remark~\ref{rem_on_hereditary} and its dual, so it is enough to check the implication ``(2)$\Rightarrow$(1)'', as ``(3)$\Rightarrow$(1)''  follows  dually. Indeed, if $(\X,\Y)$ satisfies (2), then $\Y=\X^{\perp_1}=\X^{\perp_{>0}}$ by the second part of Remark~\ref{rem_on_hereditary}. Moreover, by the dual of Remark~\ref{rem_on_hereditary}, we also get that $\Y(=\X^{\perp_{>0}})$ is closed under cokernels of monomorphisms and, therefore, $\X={}^{\perp_1}\Y={}^{\perp_{>0}}\Y$.\\
The last part of the statement about completeness follows by Remark~\ref{rem_on_completeness}.
\end{proof}

\smallskip
\noindent
{\bf Hovey correspondence.}
A {\bf model structure} on a bicomplete Abelian category $\A$ is said to be {\bf Abelian} if:
\begin{itemize}
\item the (trivial) cofibrations are  the monomorphisms that have a (trivially) cofibrant cokernel;
\item the (trivial) fibrations are the epimorphisms that have a (trivially) fibrant kernel.
\end{itemize}
By a famous result of Hovey (see \cite[Theorem~2.2]{hovey}), Abelian model structures on a bicomplete Abelian category $\A$ are in bijection with the so-called {\bf Hovey triples}, which are triples of classes of objects $(\O_\W,\O_\C,\O_\F)$ such that $\O_\W$ is thick (i.e., it is closed under summands and, given any short exact sequence $0\to A\to B\to C\to 0$ in $\A$, if two elements of the set $\{A,B,C\}$ belong in $\O_\W$, then so does the third) and both $(\O_\C,\O_\W\cap \O_\F)$ and $(\O_\C\cap \O_\W, \O_\F)$ are complete cotorsion pairs. In this case, the class $\O_\C$ is precisely the class of cofibrant
objects, $\O_\F$ is the class of fibrant objects, and $\O_\W$ is the class of trivial objects in
the corresponding Abelian model structure.

\smallskip
\noindent
{\bf Lifting of cotorsion pairs.}
Given a  cotorsion pair $(\X,\Y)$ in $\A$, Gillespie \cite[Definition~3.3]{gillespie2006flat}  introduced the following classes, with the goal of inducing a Hovey triple (and so  an Abelian model structure) in $\Ch(\A)$:
\begin{itemize}
\item $\widetilde \X$ (resp., $\widetilde \Y$) is the class of all the acyclic complexes $A^\bullet\in \Ch(\A)$ such that $Z^n(A^\bullet)\in \X$ (resp., $Z^n(A^\bullet)\in \Y$), for all $n\in\Z$;
\item $\mathrm{dg}\X$ is the class of those $A^\bullet\in \Ch(\A)$ such that $A^n\in \X$ for all $n\in\Z$ and for which the total hom-complex $\text{\sc hom}(A^\bullet,Y^\bullet)$ is exact for all $Y^\bullet\in\widetilde Y$;
\item $\mathrm{dg}\Y$ is the class of those $A^\bullet\in \Ch(\A)$ such that $A^n\in \Y$ for all $n\in\Z$ and for which the total hom-complex $\text{\sc hom}(X^\bullet,A^\bullet)$ is exact for all $X^\bullet\in\widetilde X$.
\end{itemize}

\begin{prop}[{\cite[Lemma~3.4, Proposition~3.6, Lemma~3.14]{gillespie2006flat}} and {\cite[Lemma~6.2]{hovey}}]\label{prop_known_facts_cotors}
Let $\A$ be a bicomplete Abelian category,  $(\X,\Y)$  a cotorsion pair, and  $\E$  the class of exact complexes. Then, the following  hold true
\begin{enumerate}
\item  a complex $A^\bullet\in \Ch^{+}(\A)$ belongs to $\mathrm{dg}\Y$ if and only if $A^n\in\Y$, for all $n\in\Z$;
\item if $(\X,\Y)$ is complete,  both $(\widetilde \X,\mathrm{dg}\Y)$ and $(\mathrm{dg}\X,\widetilde \Y)$ are cotorsion pairs in $\Ch(\A)$. Hence, if $(A_n)_\N$ is an inverse system such that $A_0\in \mathrm{dg}\Y$ and $A_{n+1}\to A_n$ is an epimorphism with $\ker(A_{n+1}\to A_n)\in \mathrm{dg}\Y$ (for all $n\in\N$), then $\lim_\N A_n\in \mathrm{dg}\Y$;
\item if $(\mathrm{dg}\X,\widetilde \Y)$ is a complete cotorsion pair, and $\mathrm{dg}\X\cap \E=\widetilde \X$, then $\mathrm{dg}\Y\cap \E=\widetilde \Y$. 
\end{enumerate}
\end{prop}

As a consequence of the above proposition, one can deduce the following criterion for a cotorsion pair in $\A$ to induce a model structure on complexes:

\begin{cor}\label{coro_model_struct_via_hovey}
Let $\A$ be a bicomplete Abelian category, let $(\X,\Y)$ be a cotorsion pair in $\A$, let $\E$ be the class of exact complexes and suppose that:
\begin{enumerate}
\item $(\X,\Y)$, $(\mathrm{dg}\X,\widetilde \Y)$ and $(\widetilde \X, \mathrm{dg}\Y)$ are complete;
\item $\mathrm{dg}\X\cap \E=\widetilde \X$.
\end{enumerate}
Then $(\E,\mathrm{dg}\X,\mathrm{dg}\Y)$ is a Hovey triple and, therefore, it induces an Abelian model structure on $\Ch(\A)$ such that the weak equivalences are the quasi-isomorphisms,  the cofibrations are the  monomorphisms with cokernel in $\mathrm{dg}\X$, and the
  fibrations are the epimorphisms with kernel in $\mathrm{dg}\Y$.
\end{cor}

\section{\DG-injective resolutions of complexes via truncations} \label{DG_inj_sec}\label{sec_two}

Let $\A$ be an Abelian category with enough injectives, that is, for all $A\in \A$ there is a monomorphism $A\to E$, with $E\in \Inj(\A)$, and consider the complete cotorsion pair $(\A,\Inj(\A))$ in $\A$. Then,
\begin{itemize}
\item $\widetilde \A=\E$ is the class of exact complexes;
\item $\widetilde {\Inj(\A)}$ is the class of contractible complexes of injectives, that coincides with the class $\Inj(\Ch(\A))$ of injective objects in $\Ch(\A)$ (see \cite[Exercise~2.2.1]{Weibel});
\item $\mathrm{dg}\A=\Ch(\A)$ and $\mathrm{dg}(\Inj(\A))$ is the class of the so-called  {\DG}{-injective complexes}:
\end{itemize}%
\begin{defn}
A complex $X^\bullet\in \Ch(\A)$ is {\DG}{\bf-injective} if it satisfies the following two conditions:
\begin{itemize}
    \item $X^n$ is an injective object in $\A$, for all $n\in \Z$;
    \item the total hom-complex $\text{\sc hom}(E^\bullet,X^\bullet)\in \Ch(\Ab)$ is exact, for any exact complex $E^\bullet\in \E$.
\end{itemize}
\end{defn}
In particular, the cotorsion pairs induced  in $\Ch(\A)$ are $(\Ch(A),\Inj(\Ch(\A)))$, which is complete (see \cite[Exercise~2.2.2]{Weibel}), and $(\E,\mathrm{dg}(\Inj(\A)))$. Furthermore, condition (2) in Corollary~\ref{coro_model_struct_via_hovey} is trivial in this case. Hence, when $\A$ is bicomplete, we get a model structure in $\Ch(\A)$ provided each $X^\bullet\in \Ch(\A)$ has a \DG{\bf -injective resolution}, that is,  a quasi-isomorphism $\lambda^\bullet\colon X^\bullet\to E^\bullet$, with $E^\bullet\in \mathrm{dg}(\Inj(\A))$. In fact, this is equivalent to the fact that, in the short exact sequence 
\begin{equation}\label{eq_tria_left_complete}
0\longrightarrow \Sigma^{-1}E^\bullet\longrightarrow \Sigma^{-1}\cone(\lambda^\bullet)\longrightarrow   X^\bullet \longrightarrow 0,
\end{equation}
 $\cone(\lambda^\bullet)\in \E$ and $E^\bullet\in \mathrm{dg}(\Inj(\A))$, so $(\E,\mathrm{dg}(\Inj(\A)))$ is left complete. Note also that, up to adding a bunch of disc complexes to $E^\bullet$, we may assume that $\lambda^\bullet$ is a monomorphism. Hence, we also get a right approximation sequence $0\to X^\bullet\to E^\bullet\to E^\bullet/X^\bullet\to 0$, showing that $(\E,\mathrm{dg}(\Inj(\A)))$ is complete.

\smallskip
In this section we give a complete proof, based on ideas from \cite{zbMATH06915995}, of the fact that Spalstenstein's construction of \DG-injective resolutions in the category of complexes of modules (over a ring or, more generally, over a sheaf of rings) works in every complete Abelian category $\A$ with enough injectives that satisfies Roos' condition (Ab.$4^*$)-$k$ (for some $k\in \N$).

\subsection{Resolution of bounded-below complexes}\label{bounded_subs}
Consider a complex $X^\bullet\in \Ch^{\geq k}(\A)$ (for some $k\in\Z$), where $\A$ is any Abelian category with enough injectives. In what follows we show how to construct a monomorphic quasi-isomorphism $\lambda^\bullet\colon X^\bullet\to E^\bullet$ with $E^\bullet\in \Ch^{\geq k}(\A)$ and $E^i$ injective for all $i\in\Z$; then this is a \DG-injective resolution, by Proposition~\ref{prop_known_facts_cotors}(1). Suppose for simplicity that $k=0$, that is, we start with $X^\bullet\in \Ch^{\geq0}(\A)$, and we proceed inductively to construct
\[
\lambda^\bullet\colon X^\bullet\longrightarrow E^\bullet:=(\xymatrix@C=18pt{\cdots\ar[r]& 0\ar[r]& E^0\ar[r]^-{e^0}& E^1\ar[r]^-{e^1}&\cdots}).
\] 
Indeed, let $E^{n}=0$, $e^n=0$ and $\lambda^n=0$, for all $n<0$ and, for $n=0$, define $\lambda^0\colon X^0\to E^0$ to be a chosen monomorphism with $E^0\in \Inj(\A)$. Now, given $n\geq 0$, if we have already built $e^{i-1}\colon E^{i-1}\to E^i$ and $\lambda^i\colon X^i\to E^i$, for all $i\leq n$, such that: 
\begin{enumerate}
\item[(1{$_n$})] $\lambda^i\circ d^{i-1}=e^{i-1}\circ \lambda^{i-1}$ for all $i\leq n$;
\item[(2{$_n$})] $e^{i-1}\circ e^{i-2}=0$ for all $i\leq n$;
\item[(3{$_n$})] the map $ H^i(X^\bullet)\to H^i(E^\bullet)$ induced by $\lambda^i$ is an isomorphism for all $i\leq n-1$;
\item[(4{$_n$})] the map $\bar\lambda^n\colon X^n/B^n(X^\bullet)\to E^n/B^n(E^\bullet)$ induced by $\lambda^n$ is a monomorphism;
\end{enumerate}
we consider the following push-out diagram:
\[
\xymatrix{
X^n/B^n(X^\bullet)\ar@{}[dr]|{\text{P.O.}}\ar[d]_-{\bar\lambda^n}\ar[r]^-{\bar d^n}&X^{n+1}\ar[d]^-{\mu^{n+1}}\\
E^n/B^n(E^\bullet)\ar[r]_-{y^n}&P^{n+1}.
}
\]
In particular, the map $\bar\mu^{n+1}\colon (X^{n+1}/B^{n+1}(X^\bullet)\cong)\coker(\bar d^{n})\to \coker(y^n)$ is an isomorphism. 
Moreover, as $\bar\lambda^n$ is a monomorphism by hypothesis, the above square is also a pullback, so that $\bar \lambda^n$ induces an isomorphism between $\Ker(\bar d^n)=H^n(X^\bullet)$ and $\Ker(y^n)$. Choose now a monomorphism $\iota^{n+1}\colon P^{n+1}\to E^{n+1}$, with $E^{n+1}\in\Inj(\A)$. Then, the following square remains a pullback:
\[
\xymatrix@C=40pt{
X^n/B^n(X^\bullet)\ar@{}[dr]|{\text{P.B.}}\ar[d]_-{\bar\lambda^n}\ar[r]^-{\bar d^n}&X^{n+1}\ar[d]^-{\lambda^{n+1}:=\iota^{n+1}\circ\mu^{n+1}}\\
E^n/B^n(E^\bullet)\ar[r]_-{\iota^{n+1}\circ y^n}&E^{n+1}.
}
\]
Finally define $e^n:=\iota^{n+1}\circ y^n\circ\pi^n\colon E^n\to E^{n+1}$, where $\pi^n\colon E^n\to E^n/B^n(E^\bullet)$ is the canonical projection. Conditions (1$_{n+1}$)--(4$_{n+1}$) hold by construction, so we can continue with the induction.

\subsection{Spaltenstein towers of partial resolutions}
Let $\A$ be an Abelian category with enough injectives and fix a complex $X^\bullet\in\Ch(\A)$. We start by considering the following inverse system of successive truncations:
\[
\xymatrix{
\cdots\ar[r]^-{\rho_2^\bullet}&\tau^{\geq-2}(X^\bullet)\ar[r]^-{\rho_1^\bullet}&\tau^{\geq-1}(X^\bullet)\ar[r]^-{\rho_0^\bullet}&\tau^{\geq0}(X^\bullet),
}
\]
where $\rho_n^\bullet$ is the canonical epimorphism. As $\tau^{\geq-n}(X^\bullet)\in \Ch^{\geq-n}(\A)$, the argument of Subsection~\ref{bounded_subs} gives a monomorphic \DG-injective resolution $\lambda_n^\bullet\colon\tau^{\geq-n}(X^\bullet)\to E^\bullet_n$, with $E_n^\bullet\in\Ch^{\geq -n}(\Inj(\mathcal{A}))$,  for all $n\in\mathbb{N}$. We get the following solid diagram:
\begin{equation}\label{inverse_syst_for_res_eq}
\xymatrix@C=30pt@R=18pt{
\cdots\ar[r]^-{\rho_2^\bullet}&\tau^{\geq-2}(X^\bullet)\ar[d]^-{\lambda_2^\bullet}\ar[r]^-{\rho_1^\bullet}&\tau^{\geq-1}(X^\bullet)\ar[d]^-{\lambda_1^\bullet}\ar[r]^-{\rho_0^\bullet}&\tau^{\geq0}(X^\bullet)\ar[d]^-{\lambda_0^\bullet} \\
\cdots\ar@{.>}[r]_-{ t_2^\bullet}&E^\bullet_2\ar@{.>}[r]_-{ t_1^\bullet}&E^\bullet_1\ar@{.>}[r]_-{ t_0^\bullet}&E^\bullet_0;
}
\end{equation}
and we claim that it can be completed in a commutative way by the dotted arrows. Indeed, consider the following short exact sequence for each $n>0$ 
\[
0\rightarrow \tau^{\geq -n}(X^\bullet)\stackrel{\lambda_n^\bullet}{\longrightarrow}E_n^\bullet\longrightarrow\coker(\lambda_n^\bullet)\rightarrow 0,
\]
and apply the functor $\hom_{\Ch(\mathcal{A})}(-,E_{n-1}^\bullet)$ to get the following exact sequence in $\Ab$:  
\[
\hom_{\Ch(\mathcal{A})}(E_n^\bullet,E_{n-1}^\bullet)\stackrel{(\lambda_n^\bullet)^*}{\longrightarrow}\hom_{\Ch(\mathcal{A})}( \tau^{\geq -n}(X^\bullet),E_{n-1}^\bullet)\longrightarrow\Ext_{\Ch(\mathcal{A})}^1(\coker(\lambda_n^\bullet),E_{n-1}^\bullet )=0.
\]
 The last equality to zero is due to the fact $\coker(\lambda_n^\bullet)\in\mathcal{E}$, $E_{n-1}^\bullet\in\mathrm{dg}(\Inj(\mathcal{A}))$ and $(\mathcal{E}, \mathrm{dg}(\Inj(\mathcal{A})))$ is a cotorsion pair in $\Ch(\mathcal{A})$. Then it follows that $\lambda_{n-1}^\bullet\circ\rho_{n-1}^\bullet\in\Im((\lambda_n^\bullet)^*)$, so we have a map $t_{n-1}^\bullet\colon E_n^\bullet\to E_{n-1}^\bullet$, such that $t_{n-1}^\bullet\circ\lambda_n^\bullet =\lambda_{n-1}^\bullet\circ\rho_{n-1}^\bullet$, as desired.

Observe also that, possibly adding a number of disk complexes to our partial resolutions if needed, one can always suppose that $t_n^\bullet$ is a (degree-wise split) epimorphism with $\ker(t_n^\bullet)\in \Ch^+(\Inj(\A))$, whence with \DG-injective kernel, for all $n \in\N$. 
\begin{defn}
A {\bf Spaltenstein  tower} of partial resolutions for $X^\bullet\in\Ch(\A)$ is a commutative diagram like the one in \eqref{inverse_syst_for_res_eq} where $\lambda^\bullet_n$ is a \DG-injective resolution for all $n\in\N$, and $t_n^\bullet$ is a (degree-wise split) epimorphism with \DG-injective kernel, for each $n\in\N$. 
\end{defn}

In particular, in the above discussion we have just shown that:

\begin{prop}\label{prop_existence_towers}
Let $\A$ be an Abelian category with enough injectives. Then any complex in $\Ch(\A)$ has a Spaltenstein  tower of partial resolutions.
\end{prop}

\subsection{Injective resolutions and the (Ab.$4^*$)-$k$ condition}\label{unbounded_subs}
The goal of this subsection is to prove that, under suitable conditions, the inverse limit of a Spaltenstein  tower of partial resolutions actually produces a \DG-injective resolution. We start with  two  easy observations:

\begin{lem}\label{lem_cohomologies_of_lim_vs_prod}
Let $\A$ be a complete Abelian category and consider a sequence of complexes 
\[
\xymatrix{
\cdots\ar[r]^-{t^\bullet_2}&X^\bullet_2\ar[r]^-{t^\bullet_1}&X^\bullet_1\ar[r]^-{t^\bullet_0}&X^\bullet_0
}
\]
such that, for each $n\in\N$, the chain map $t^\bullet_n$ is degree-wise a split-epimorphism. Then, there is the following degree-wise split-exact sequence in $\Ch(\A)$:
\[
\xymatrix{
0\ar[r]&\varprojlim_\N X^\bullet_n\ar[r]&\prod_\N X^\bullet_n\ar[r]^-{1-t}&\prod_\N X^\bullet_n\ar[r]&0,
}
\]
where the map $1-t$ is described by a matrix with identities on the main diagonal, $-t^\bullet_n$ as the $n$-th entry of the first superdiagonal, and $0$'s everywhere else.
\end{lem}
\begin{proof}
The universal property defining $\varprojlim_\N X^\bullet_n$ is clearly equivalent to the universal property defining the kernel of $1-t$. Moreover, to prove that $1-t$ is a degree-wise split-epimorphism, it is enough to find, for each $k\in\Z$, a right-inverse for the following matrix:
\[
\SmallMatrix{
\id_{X^k_0}&-t_0^k&0&0&0&\ldots\\
0&\id_{X^k_1}&-t_1^k&0&0&\ldots\\
\vdots&\ddots&\ddots&\ddots&\ddots&\ddots}
\]
Fix, for each $n\in \N$ and $k\in \Z$, a split-monomorphism $s^k_n\colon X^k_n\to X^{k}_{n+1}$ such that $t^k_n\circ s^k_n=\id_{X^k_n}$. Then, the right-inverse we are looking for is the following row-finite matrix:
\[
\SmallMatrix{
0&0&0&0&0&\ldots\\
-s_0^k&0&0&0&0&\ldots\\
-s^k_1\circ s_0^k&-s_1^k&0&0&0&\ldots\\
-s^k_2\circ s^k_1\circ s_0^k&-s^k_2\circ s_1^k&-s^k_2&0&0&\ldots\\
\vdots&\vdots&\vdots&\ddots&\ddots&\ddots}
\qedhere\]
\end{proof}

\begin{lem}\label{degree-wise_iso_and_prod_lem}
Let $\A$ be a complete Abelian category and fix some integer $k\in \Z$. Consider also two families of complexes $(X^\bullet_n)_{n\in\N}$, $(Y^\bullet_n)_{n\in\N}\subseteq \Ch(\A)$, and maps of complexes $\phi^\bullet_n\colon X^{\bullet}_n\to Y^\bullet_n$ (for each $n\in \N$), such that $\phi^{i}_n$ is an isomorphism for all $i\geq k$. Letting
\[
\xymatrix{\phi^\bullet:=\prod_\N\phi^\bullet_n\colon \prod_\N X^\bullet_n\longrightarrow \prod_\N Y^\bullet_n },
\]
$\phi^i$ is an isomorphism for all $i\geq k$. In particular, $H^i( \prod_\N X^\bullet_n)\cong H^i( \prod_\N Y^\bullet_n)$, for all $i>k$. 
\end{lem}
Let us point out that, in the above lemma, we are just claiming that there is an isomorphism for the $i$-th cohomologies of the products for all $i>k$, but we do not exclude that this may actually fail for $i=k$, even if $H^k(X^\bullet_n)\cong H^k(Y^\bullet_n)$, for all $n\in\N$. In fact,  for products in $\Ch(\A)$ to commute with cohomologies, one needs to assume that products are exact in $\A$ which, of course, is not always the case (see, for example, the category $\G$ that we will describe in Section~\ref{sec_three}).
\begin{proof}
As products are built component-wise, the result follows from the observation that the class of isomorphisms in $\A$ is closed under products. 
\end{proof}

The following definition was introduced by Roos \cite[Definition~1.1]{Roos}:

\begin{defn}\label{def_original_roos}
A complete Abelian category $\A$ with enough injectives is (Ab.$4^*$)-$k$, for some $k\in \N$ if, for any family $(A_\lambda)_{ \Lambda}\subseteq \A$, the $n$-th derived functor $\prod_\Lambda^{(n)}A_\lambda$ vanishes for all $n>k$.
\end{defn}
In particular, $\A$ is (Ab.$4^*$)-$0$ if and only if it satisfies Grothendieck's axiom (Ab.$4^*$), that is, products are exact in $\A$. In the following lemma we show a direct consequence of the \mbox{(Ab.$4^*$)-$k$} condition, which is precisely what is needed for the construction of \DG-injective resolutions:
\begin{lem}\label{useful_form_of_ab4*_k_lem}
Let $\A$ be a complete (Ab.$4^*$)-$k$ Abelian category (for some $k\in \N$) with enough injectives, and let $(X^\bullet_n)_{n\in \N}$ be a family of bounded below complexes that satisfies the following conditions:
\begin{enumerate}
\item $X^i_n$ injective for all $i\in\Z$ and $n\in \N$ (so each $K^\bullet_n$ is \DG-injective);
\item there is an integer $h\in\Z$ such that $H^i(X^\bullet_n)=0$ for all $n\in\N$ and all $i\geq h$.
\end{enumerate}
Then, $H^i\left(\prod_{\N}X_n^\bullet\right)=0$, for all $i>h+k+1$.
\end{lem}
\begin{proof}
Observe that $(\tau^{\geq h}X^\bullet_n)_{\N}$ is a family of injective resolutions for
 $(S^h(X^h_n/B^h(X_n^\bullet)))_{n\in\N}$. Hence, the $h+i$-th cohomology of $\prod_\N \tau^{\geq h}(X^\bullet_n)$ is precisely the $i$-th derived functor $\prod_{n\in\N}^{(i)}(X^h_n/B^h(X_n^\bullet))$. In particular, by Definition~\ref{def_original_roos}, we deduce that $H^i\left(\prod_{\N}\tau^{\geq h}X_n^\bullet\right)=0$, for all $i>h+k$. Moreover, for each $n\in\N$, the canonical map $X^\bullet_n\to \tau^{\geq h}X^\bullet_n$ is an isomorphism in all degrees $>h$. We can now conclude by the first part and Lemma~\ref{degree-wise_iso_and_prod_lem}.
\end{proof}

Combining the above lemma with Lemma~\ref{lem_cohomologies_of_lim_vs_prod}, we deduce the following:

\begin{prop}\label{prop_cohom_vanish}
Let $\A$ be a complete (Ab.$4^*$)-$k$ Abelian category (for some $k\in \N$) with enough injectives, and let $K^\bullet:=\varprojlim_\N K^\bullet_n$ be the limit of the following sequence of bounded below complexes:
\[
\xymatrix{
\cdots \ar[r]&K_2^\bullet\ar[r]^{t_1^\bullet}&K_1^\bullet\ar[r]^{t_0^\bullet}&K^\bullet_0.
}
\]
We suppose that this sequence satisfies the following conditions:
\begin{enumerate}
\item $K^i_n$ is injective for all $i\in\Z$ and $n\in \N$ (so each $K^\bullet_n$ is \DG-injective);
\item $t_n^\bullet$ is a degree-wise split-epimorphism (i.e., an $\Inj(\A)$-fibration, see Definition~\ref{I_mod_st_deff}), for all $n\in \N$;
\item there is an integer $h\in\Z$ such that $H^i(K^\bullet_n)=0$ for all $n\in\N$ and all $i\geq h$.
\end{enumerate}
Then, $H^i(K^\bullet)=0$, for all $i>h+k+2$.
\end{prop}
\begin{proof}
By Lemma~\ref{lem_cohomologies_of_lim_vs_prod} and condition (2), there is a short exact sequence
\begin{equation}\label{milnor_ses_for_K_eq}
\xymatrix{
0\ar[r]&\varprojlim_{\N} K_n^\bullet\ar[r]&\prod_{\N} K_n^\bullet\ar[r]&\prod_{\N} K_n^\bullet\ar[r]&0.
}
\end{equation}
Furthermore, by Lemma~\ref{useful_form_of_ab4*_k_lem} and conditions (1) and (3),  $H^i(\prod_{\N} K_n^\bullet)=0$ for all $i> h+k+1$. To conclude, consider the long exact sequence in cohomologies induced by \eqref{milnor_ses_for_K_eq}, and use the vanishing of cohomologies for the product to deduce that $H^i(K^\bullet)=0$, for all $i> h+k+2$.
\end{proof}

Let us also record the following useful consequence:

\begin{cor}\label{cor_unbound_res}
Let $\A$ be a complete (Ab.$4^*$)-$k$ Abelian category (for some $k\in \N$) with enough injectives, and let $E^\bullet:=\varprojlim_\N E^\bullet_n$ be the limit of the following sequence of bounded below complexes:
\[
\xymatrix{
\cdots \ar[r]&E_2^\bullet\ar[r]^{t_1^\bullet}&E_1^\bullet\ar[r]^{t_0^\bullet}&E^\bullet_0.
}
\]
We suppose that this sequence satisfies the following conditions:
\begin{enumerate}
\item $E^i_n$ is injective for all $i\in\Z$ and $n\in \N$ (so each $E^\bullet_n$ is \DG-injective);
\item $t_n^\bullet$ is a degree-wise split-epimorphism, for all $n\in \N$ (i.e., each $t^\bullet_n$ is an $\Inj(\A)$-fibration);
\item $H^i(t_n^\bullet)\colon H^i(E_{n+1}^\bullet)\to H^i(E_n^\bullet)$ is an isomorphism, for all $n\in \N$ and $i\geq -n$;
\item $H^{-n-1}(E_n^\bullet)=0$, for all $n\in \N$.
\end{enumerate}
Then, $H^i(E^\bullet)\cong H^i(E_h^\bullet)$, for all $h\in \N$ and $i\geq  -h+k+3$.
\end{cor}
\begin{proof}
Fix an $h\in \N$ and, for each $n>h$, consider the composition 
$
c_n^\bullet:=t_{h}^{\bullet}\circ\ldots\circ t_{n-1}^{\bullet}\colon E^\bullet_n\to E_h^\bullet.
$ 
By (2), each $c_n^\bullet$ is a degree-wise split-epimorphism, so $K^\bullet_n:=\Ker(c_n^\bullet)$ is \DG-injective and it fits into a degree-wise split-exact sequence:
$0\to K_n^\bullet\to E^\bullet_n\to E_h^\bullet\to 0$.
Consider the associated long exact sequence of cohomologies and deduce by (3) and (4) that $H^i(K_n^\bullet)=0$, for all $i\geq -h$. Taking the limit for $n>h$, we get the following short exact sequence:
\begin{equation}\label{ses_def_K_eq}
\xymatrix{
0\ar[r]&\varprojlim_{n>h} K_n^\bullet\ar[r]&E^\bullet\ar[r]&E_h^\bullet\ar[r]&0,
}
\end{equation}
where the central term is $E^\bullet$ since $\N_{>h}$ is cofinal in $\N$. Consider now the long exact sequence in cohomologies induced by \eqref{ses_def_K_eq}; by Proposition~\ref{prop_cohom_vanish}, $H^i(\varprojlim_{\N}K_n^\bullet)=0$ for all $i>-h+k+2$, so that $H^i(E^\bullet)\cong H^i(E_h^\bullet)$ for all $i\geq  -h+k+3$, as desired.
\end{proof}

Finally, we are able to prove the main result of this section:

\begin{thm}\label{main_unbounded_dg_thm}
Let $\A$ be a complete (Ab.$4^*$)-$k$ (for some $k\in\N$) Abelian category with enough injectives and let $X^\bullet\in \Ch(\A)$. Then, there is a Spaltenstein tower of partial resolutions of $X^\bullet$:
\[
\xymatrix@C=30pt@R=18pt{
\cdots\ar[r]^-{}&\tau^{\geq-2}(X^\bullet)\ar[d]^-{\lambda_2^\bullet}\ar[r]^-{}&\tau^{\geq-1}(X^\bullet)\ar[d]^-{\lambda_1^\bullet}\ar[r]^-{}&\tau^{\geq0}(X^\bullet)\ar[d]^-{\lambda_0^\bullet} \\
\cdots\ar[r]_-{ t_2^\bullet}&E^\bullet_2\ar[r]_-{ t_1^\bullet}&E^\bullet_1\ar[r]_-{ t_0^\bullet}&E^\bullet_0. 
}
\]
Define  $E^\bullet:=\varprojlim_\N E^\bullet_n$, and let $\lambda^\bullet:=\varprojlim_\N\lambda_n^\bullet\colon  X^\bullet\cong \varprojlim_\N \tau^{\geq-n}(X^\bullet)\to E^\bullet$ be the induced map. Then, $\lambda^\bullet\colon X^\bullet \to E^\bullet$ is a \DG-injective resolution.
\end{thm}
\begin{proof}
The existence of Spaltenstein towers follows by Proposition~\ref{prop_existence_towers}. Furthermore, by Proposition~\ref{prop_known_facts_cotors}(2), $E^\bullet$ is \DG-injective, so it is enough to prove that $\lambda^\bullet$ is a quasi-isomorphism. Observe that, for each $h\in\N$, Corollary~\ref{cor_unbound_res} implies that $H^i(E^\bullet)\cong H^i(E_h^\bullet)\cong H^i(X^\bullet)$, for all $i\geq  -h+k+3$. As this holds for all $h\in \N$, we conclude that $H^i(E^\bullet)\cong H^i(X^\bullet)$ for all $i\in\Z$.
\end{proof}
 
\section{The ``go-to'' counter-example for unbounded resolutions}\label{poison_sec}\label{sec_three}
The goal of this section is to construct a suitable bicomplete Abelian category with enough injectives $\G$ (that, in fact, will be a Grothendieck category) which is not (Ab.$4^*$)-$k$ for any $k\in \N$, and a particular unbounded complex $X^\bullet\in \Ch(\G)$ such that the construction discussed in Section~\ref{DG_inj_sec} fails to produce a \DG-injective resolution for $X^\bullet$,  showing that the (Ab.$4^*$)-$k$ condition is needed for that construction. 
As we will see in Sections~\ref{CE_sec}, \ref{Saneblidze_sec} and \ref{DY_sec}, the  complex $X^\bullet\in \Ch(\G)$ can be used to show that several other constructions of \DG-injective resolutions fail in general.

\subsection{Construction of the category $\G$}
We start recalling Nagata's construction \cite[Example~1, Appendix~A.1]{Nagata} of a commutative Noetherian ring of infinite Krull dimension. Let $\mathbf k$ be a field, consider the polynomial ring on countably many variables 
\[
\mathbf k[\underline x]:=\mathbf k[x_0,x_1, x_2, \ldots],
\] 
and the following sequence of prime ideals in $\mathbf k[\underline x]$:
\[
\mathfrak p_1: = (x_0,x_1),\ \mathfrak p_2: = (x_2, x_3,x_4),\ \mathfrak p_3 := (x_5, x_6, x_7, x_8),\ \ldots
\] 
where the depth of $\mathfrak p_i$ is  $i+1$, for all $i\geq 1$. Let $S$  be the multiplicative set of those elements of $\mathbf k[\underline x]$ which are not in any of the $\mathfrak p_i$'s, and define $R$ as the ring of $S$-fractions of $\mathbf k[\underline x]$, that is,
\[
S:=\mathbf k[\underline x]\setminus\bigcup_{i=1}^{\infty}\mathfrak p_i\quad\text{and}\quad R:= S^{-1}(\mathbf k[\underline x]).
\]
\begin{lem}[{\cite[Example~1, Appendix~A.1]{Nagata}}]
With the above notation, $R$ is a commutative Noetherian ring of infinite Krull dimension. In fact, the maximal ideals of $R$ are of the form $\mathfrak m_i := S^{-1}\mathfrak p_i$, with $i\geq 1$, and this is a sequence of ideals of strictly increasing height.
\end{lem}

Consider  the hereditary torsion class $\mathcal T:=\mathrm{Loc}(R/\mathfrak m_i:i\geq 1)$ generated by the simple $R$-modules (that is, the first layer of the Gabriel filtration of $\mathrm{Mod}(R)$) and define 
\[
\G:=\mathrm{Mod}(R)/\mathcal T
\] 
as the Gabriel quotient of $\mathrm{Mod}(R)$ over $\mathcal T$, identified with the Giraud subcategory of $\mathrm{Mod}(R)$ of those $R$-modules $M$ such that both $M$ and $E(M)/M$ belong to 
\[
\F:=\{R/\mathfrak m_i:i\geq 1\}^\perp=\{M\in \Mod(R):\hom_R(R/\mathfrak m_i,M)=0,\text{ for all $i\geq 1$}\}\subseteq \mathrm{Mod}(R).
\] 
It is well-known that $\G$ is a Grothendieck category. Moreover, the inclusion $\iota\colon \G\to \mathrm{Mod}(R)$ has an exact left adjoint $\mathbf Q\colon \mathrm{Mod}(R)\to \G$ such that 
$
\Ker(\mathbf Q):=\{M\in\mathrm{Mod}(R):\mathbf Q(M)=0\}=\mathcal T$.
The product $\mathrlap{\text{\tiny \hspace{4.3pt}$\G$}}\prod_IX_i$ in $\G$ of a family $(X_i)_{I}\subseteq \G$  can be computed as $\mathbf Q(\mathrlap{\text{\tiny \hspace{3.5pt}$R$}}\prod_{I}X_i)$ (see \cite[Corollary~2.12]{virili2017exactness}), where $\mathrlap{\text{\tiny \hspace{3.5pt}$R$}}\prod$ denotes the product in the category $\mathrm{Mod}(R)$. 

The injective objects in $\G$ are exactly the injective $R$-modules in $\F$. Furthermore, as $R$ is commutative Noetherian, it is well-known (see \cite[Chap.\,VII, Proposition~4.5]{S75}) that  $\mathcal T$ is stable (i.e., closed under taking injective envelopes in $\mathrm{Mod}(R)$). As a consequence,  any injective module $E\in \mathrm{Mod}(R)$ decomposes as 
\[
E\cong t(E)\oplus \mathbf Q(E),
\] 
where $t(E)$ is an injective module in $\mathcal T$ and $\mathbf Q(E)$ is an injective object in $\G$. Note also that, if we denote by $\mathrm{Sp}(R)$ the prime ideal spectrum of $R$, and by $\mathrm{MSp}(R)=\{\mathfrak m_i:i\in \N\}\subseteq\mathrm{Sp}(R)$ the maximal ideal spectrum, by the classification of injectives over commutative Noetherian rings given by Matlis \cite{matlis}, we have that $E\cong \bigoplus_{\mathfrak p\in\mathrm{Sp}(R)}E(R/\mathfrak p)^{(I_{\mathfrak p})}$ (for suitable sets $I_{\mathfrak p}$), so that 
\[
t(E)\cong \bigoplus_{i=1}^\infty E(R/\mathfrak m_i)^{(I_{\mathfrak m_i})}\qquad\text{ and }\qquad \mathbf Q(E)\cong \bigoplus_{\mathfrak p\in\mathrm{Sp}(R)\setminus\mathrm{MSp}(R)}E(R/\mathfrak p)^{(I_{\mathfrak p})}.
\]

\subsection{The complex $X^\bullet\in \Ch(\G)$}
With the usual abuse of notation, let $\mathbf Q\colon \Ch(\Mod(R))\to \Ch(\G)$ be the functor on cochain complexes obtained by applying $\mathbf Q$ component-wise. Observe that, as $\mathbf Q$ is exact, it commutes with cohomologies, that is, given $X^\bullet\in \Ch(\Mod(R))$, we have:
\[
H^i_\G(\mathbf Q(X^\bullet))\cong \mathbf Q(H^i_R(X^\bullet)),\quad \text{for all $i\in \Z$,}
\]
where $H^i_\G(-)$ and $H^i_R(-)$ denote the $i$-th cohomology in $\Ch(\G)$ and in $\Ch(\Mod(R))$, respectively. On the other hand, let us remark that it is still important to distinguish between cohomologies taken in $\Ch(\G)$ from those taken in $\Ch(\Mod(R))$. In fact, since the inclusion of $\G$ in $\Mod(R)$ is not exact, there are complexes $X^\bullet\in \Ch(\G)$ such that $H^i_\G(X^\bullet)\neq H^i_R(X^\bullet)$ for some $i\in \Z$ (e.g., any short exact sequence in $\G$ that is not exact in $\Mod(R)$, viewed as a complex concentrated in degrees $0$, $1$, and $2$).

Let $\mu\colon R\to (\iota\circ \mathbf Q)(R)$ be the component at $R$ of the unit of the adjunction \mbox{$\mathbf Q\dashv\iota\colon \Mod(R)\rightleftarrows \G$} and fix an injective resolution $\lambda\colon R\to E^\bullet$ of $R$ in $\Mod (R)$. In the language of \cite{zbMATH06915995},  the composition $\mathbf Q(\lambda)\circ \mu\colon R\stackrel{}{\to}\mathbf Q(R)\to \mathbf Q(E^\bullet)$ is a relative $(\F\cap \Inj(\Mod(R))$-injective resolution. As a consequence:

\begin{lem}[{\cite[Lemma~8.2]{zbMATH06915995}}]\label{res_of_R_lem}
Let $\lambda\colon R\to E^\bullet:=( E^0\overset{e^0}\longrightarrow E^1\overset{e^1}\longrightarrow \dots \overset{}\longrightarrow E^n\overset{e^n}\longrightarrow \dots)$ be an injective resolution of $R$ as an $R$-module, and consider $\mathbf Q(E^\bullet)\in \Ch^{\geq0}(\Mod (R))$. Then, 
\[
H_R^{i}(\mathbf Q(E^\bullet))\cong \begin{cases}
R&\text{if $i=0$;}\\
E(R/\mathfrak m_i)&\text{otherwise;}
\end{cases} \ \ \text{and}\ \ H_\G^{i}(\mathbf Q(E^\bullet))\cong \begin{cases}
\mathbf Q(R)&\text{if $i=0$;}\\
\mathbf Q(E(R/\mathfrak m_i))=0&\text{otherwise.}
\end{cases}
\]
In particular, $\mathbf Q(\lambda)\colon \Q(R)\to \mathbf Q(E^\bullet)$ is an injective resolution in $\G$.
\end{lem}

We are now ready to introduce the central example for our paper. Consider:
\[\xymatrix{
\mathrlap{\text{\tiny \hspace{3.5pt}$R$}}\prod_{i\in \Z}S^i(R)}:=\ \big(\xymatrix{\cdots\ar[r]^-0&R\ar[r]^-0&R\ar[r]^-0&\cdots\ar[r]^-0&R\ar[r]^-0&\cdots}\big)\in  \Ch(\Mod(R)).
\]
Our example is the following complex in $\Ch(\G)$:
\[
\xymatrix{X^\bullet:= \mathbf Q(
\mathrlap{\text{\tiny \hspace{3.5pt}$R$}}\prod_{i\in \Z}S^i(R))\cong  \mathrlap{\text{\tiny \hspace{4.3pt}$\G$}}\prod_{i\in \Z}S^i(\Q(R))=}\big(\xymatrix@C=14pt{\cdots\ar[r]^-0&\Q(R)\ar[r]^-0&\Q(R)\ar[r]^-0&\cdots\ar[r]^-0&\Q(R)\ar[r]^-0&\cdots}\big)
.
\]

\begin{prop}\label{poison_in_action_prop}
In the above setting, let  $\lambda\colon R\to E^\bullet$ be an injective resolution in $\Mod (R)$. Then, 
\begin{enumerate}
\item $\mathrlap{\text{\tiny \hspace{3.5pt}$R$}}\prod_{i\in \Z}\Sigma^i\lambda\colon \mathrlap{\text{\tiny \hspace{3.5pt}$R$}}\prod_{i\in \Z}S^i(R)\to \mathrlap{\text{\tiny \hspace{3.5pt}$R$}}\prod_{i\in \Z}\Sigma^iE^\bullet$ is a quasi-isomorphism in $\Ch(\Mod(R))$;
\item $\mathrlap{\text{\tiny \hspace{4.3pt}$\G$}}\prod_{i\in \Z}\Sigma^i\mathbf{Q}(\lambda)\colon X^\bullet\to \mathrlap{\text{\tiny \hspace{4.3pt}$\G$}}\prod_{i\in \Z}\Sigma^i\mathbf{Q}(E^\bullet)$ is not a quasi-isomorphism in $\Ch(\G)$.
\end{enumerate}
\end{prop}
\begin{proof}
(1) Since $\Mod(R)$ is (Ab.$4^*$), the class of quasi-isomorphisms is closed under products.

\smallskip\noindent
(2) Take the mapping cone $Z^\bullet:=\cone(\mathrlap{\text{\tiny \hspace{4.3pt}$\G$}}\prod_{i\in \Z}\Sigma^i\mathbf{Q}(\lambda))$. By the proof of \cite[Theorem~8.4]{zbMATH06915995}, viewing  $Z^\bullet$ as a complex in $\Ch(\Mod(R))$, one has $H_R^n(Z^\bullet)\cong \prod_{j=1}^\infty E(R/\mathfrak m_j)\notin \mathcal T$, for all $n\in \Z$. Hence, $H_\G^n(Z^\bullet)\cong \mathbf Q(H_R^n(Z^\bullet))\neq 0$, showing that $Z^\bullet$ is not exact as a complex in $\Ch(\G)$.
\end{proof}

Observe that, as $\G$ is a Grothendieck category, our complex $X^\bullet\in\Ch(\G)$ is known to have a \DG-injective resolution (see, e.g., \cite[Theorem~5.4]{tarrio2000localization} or \cite[Theorem~1.7]{7fd5ffe3-5ff1-3e28-b6fb-e47f75d8437a}). On the other hand,  the following corollary shows that the construction of Section~\ref{DG_inj_sec} fails to produce a \DG-injective resolution for $X^\bullet$:

\begin{cor}
In the above setting, the following statements hold for all $n\in\N$:
\begin{enumerate}
\item $\tau^{\geq-n}(X^\bullet)=\mathrlap{\text{\tiny \hspace{4.3pt}$\G$}}\prod_{i\geq -n}S^{i}(\Q(R))$;
\item the diagonal  $\tau^{\geq-n}(X^\bullet)\to \mathrlap{\text{\tiny \hspace{4.3pt}$\G$}}\prod_{i\leq n}\Sigma^{i}(\mathbf Q(E^\bullet))$ is a \DG-injective resolution;
\item the kernel of the projection $\mathrlap{\text{\tiny \hspace{4.3pt}$\G$}}\prod_{i\leq n+1}\Sigma^{i}(\mathbf Q(E^\bullet))\to \mathrlap{\text{\tiny \hspace{4.3pt}$\G$}}\prod_{i\leq n}\Sigma^{i}(\mathbf Q(E^\bullet))$ is \DG-injective;
\item  the diagonal map $X^\bullet\to \mathrlap{\text{\tiny \hspace{4.3pt}$\G$}}\prod_{i\in \Z}\Sigma^{i}(\mathbf Q(E^\bullet))$ is the limit of a Spaltenstein tower of partial resolutions of $X^\bullet$, but it is not a \DG-injective resolution of $X^\bullet$.
\end{enumerate}
\end{cor}
\begin{proof}
(1) is trivial while (2) follows since, even if we are considering infinite products, they are degree-wise finite products, so they can be seen also as coproducts (which are exact in $\G$), showing that the map in the statement is a quasi-isomorphism. Moreover,  $\mathrlap{\text{\tiny \hspace{4.3pt}$\G$}}\prod_{i\leq n}\Sigma^{i}(\mathbf Q(E^\bullet))$ is \DG-injective by Proposition~\ref{prop_known_facts_cotors}(1). Part (3) is also trivial since the kernel of the projection in the statement is  $\Sigma^{n+1}(\mathbf Q(E^\bullet))$.  Finally, part (4) follows by (1--3) and Proposition~\ref{poison_in_action_prop}(2).
\end{proof}

\section{Cartan-Eilenberg resolutions and their totalization}\label{CE_sec}\label{sec_four}

In this section we recall the  construction of the Cartan-Eilenberg resolution of a complex and we show that its totalization, when applied to the example of Section~\ref{poison_sec} fails to produce a \DG-injective resolution. On the other hand, as for Spaltenstein's approach, Cartan-Eilenberg resolutions produce \DG-injective resolutions, provided the base category is (Ab.$4^*$)-$k$, for some $k\in\N$.

\subsection{Bicomplexes and their totalizations} As a Cartan-Eilenberg resolution is, actually, a bicomplex, we start recalling that  a {\bf bicomplex} $C^{\bullet,\bullet}$ over an Abelian category $\A$ is given by
\begin{itemize}
\item a family of objects $C^{\bullet,\bullet}:=(C^{n,m})_{(n,m)\in\Z\times\Z}$; 
\item {\bf differentials} $d_0^{\bullet,\bullet}\colon C^{\bullet,\bullet}\to C^{\bullet,\bullet+1}$ and  $d_1^{\bullet,\bullet}\colon C^{\bullet,\bullet}\to C^{\bullet+1,\bullet}$,  called {\bf horizontal} and {\bf vertical}, respectively, that satisfy the following formulas for all $(n,m)\in \Z\times \Z$:
\[
d_0^{n,m+1}\circ d_0^{n,m}=0,\ d_1^{n+1,m}\circ d_1^{n,m}=0,\ \text{and}\ d_1^{n,m+1}\circ d_0^{n,m}+d_0^{n+1,m}\circ d_1^{n,m}=0.
\]
\end{itemize}
In other words, a bicomplex $C^{\bullet,\bullet}$ is represented by the following diagram
\begin{equation}\label{picture_bicomplex_eq}
\xymatrix@R=3.7pt@C=27.5pt{
&\vdots\ar[d]&&\vdots\ar[d]&&\vdots\ar[d]&\\
\cdots \ar[r]&C^{n,m-1}\ar[dd]^-{d_1^{n,m-1}}\ar[rr]^-{d_0^{n,m-1}}&&C^{n,m}\ar[rr]^-{d_0^{n,m}}\ar[dd]^-{d_1^{n,m}}&&C^{n,m+1}\ar[r]\ar[dd]^-{d_1^{n,m+1}}&\cdots\\
\\
\cdots \ar[r]&C^{n+1,m-1}\ar[rr]^-{d_0^{n+1,m-1}}\ar[dd]&&C^{n+1,m}\ar[rr]^-{d_0^{n+1,m}}\ar[dd]&&C^{n+1,m+1}\ar[r]\ar[dd]&\cdots\\
\\
&\vdots&&\vdots&&\vdots}
\end{equation}
where all rows and columns are cochain complexes, and all the little squares anticommute. With reference to the  representation \eqref{picture_bicomplex_eq}, we say that $C^{\bullet,\bullet}$ is a {\bf lower half-plane} bicomplex provided $C^{n,m}=0$ for all $n<0$, $m\in\Z$. A {\bf morphism of bicomplexes} is just a family of morphisms $\phi^{\bullet,\bullet}\colon X^{\bullet,\bullet}\to Y^{\bullet,\bullet}$ that commutes both with horizontal and vertical differentials. We denote by $\mathrm{bCh}(\A)$ the category of bicomplexes over $\A$.  

\smallskip
In fact, $\mathrm{bCh}(\A)$ is equivalent to the category $\Ch(\Ch(\A))$ obtained as follows: start with  $\A$ and form the category of cochain complexes $\Ch(\A)$, which is itself an Abelian category; iterating the construction once more, one gets the category $\Ch(\Ch(\A))$ of {\bf double complexes}, where a double complex can be represented by a diagram like \eqref{picture_bicomplex_eq}, where all rows and columns are cochain complexes but the small squares actually commute. The canonical functor $\mathrm{bCh}(\A)\to\Ch(\Ch(\A))$ is constructed via the following {\bf sign trick}: given $C^{\bullet,\bullet}\in\mathrm{bCh}(\A)$ take, for each $m\in \N$, the  complex 
\[
C_0^{\bullet,m}=(\xymatrix@C=20pt{\cdots \ar[r]&C^{n-1,m}\ar[rr]^-{(-1)^md_1^{n-1,m}}&&C^{n,m}\ar[rr]^-{(-1)^md_1^{n,m}}&&C^{n+1,m}\ar[r]&\cdots})
\] 
Then, the horizontal differentials of $C^{\bullet,\bullet}$ induce chain maps $C_0^{\bullet,m}\to C_0^{\bullet,m+1}$ for all $m\in\Z$, so that
\[
C^{\bullet,\bullet}_0:=\quad (\xymatrix{
\cdots\ar[r]&C_0^{\bullet,m-1}\ar[r]^-{d_0^{\bullet,m-1}}&C_0^{\bullet,m}\ar[r]^-{d_0^{\bullet,m}}&C_0^{\bullet,m+1}\ar[r]^-{d_0^{\bullet,m+1}}&\cdots}) \in \Ch(\Ch(\A)).
\]
In particular, given $C^{\bullet,\bullet}\in \mathrm{bCh}(\A)$ we can consider, for each $n\in\Z$, the complexes (in $\Ch(\A)$) of $n$-coboundaries, $n$-cocycles, and $n$-cohomologies of the complex $C^{\bullet,\bullet}_0\in \Ch(\Ch(\A))$, that is:
\begin{enumerate}
    \item[--] 
    $B^n(C^{\bullet,\bullet}_0)=\Im(d_0^{\bullet,n-1})=(\xymatrix@R=5pt@C=10pt{\cdots \ar[r]&\Im(d_0^{m,n-1})\ar[r]
    &\Im(d_0^{m+1,n-1})\ar[r]&\cdots})\leq C_0^{\bullet,n}$;
    \item[--] 
    $Z^n(C^{\bullet,\bullet}_0)=\Ker(d_0^{\bullet,n})=(\xymatrix@R=30pt@C=10pt{\cdots \ar[r]&\Ker(d_0^{m,n})\ar[r]
    &\Ker(d_0^{m+1,n})\ar[r]&\cdots})\leq C_0^{\bullet,n}$;
    \item[--] $H^n(C^{\bullet,\bullet}_0)=K^n(C^{\bullet,\bullet}_0)/B^n(C^{\bullet,\bullet}_0)$.
\end{enumerate}

Finally, if $\A$ is complete, we can also consider the {\bf totalization} of the bicomplex $C^{\bullet,\bullet}$. This is the cochain complex $\Tot(C^{\bullet,\bullet})\in \Ch(\A)$, defined by letting $\Tot(C^{\bullet,\bullet})^n:=\prod_{i\in\Z} C^{i,n-i}$, for all $n\in\Z$, and with the $n$-th differential $\partial^n:=(\partial^n_i)_{i\in\Z}\colon \prod_{i\in\Z} C^{i,n-i}\to \prod_{i\in\Z} C^{i,n+1-i}$, where $\partial^n_i$ is just the map $(d_0^{i,n-i},\, d_1^{i,n-i})^t\colon C^{i,n-i}\to C^{i,n+1-i}\times C^{i+1,n-i}$ composed with the obvious projection and inclusion.

\subsection{Classical Cartan-Eilenberg resolutions} 
Let $\A$ be an Abelian category with enough injectives, and consider a (possibly unbounded) complex 
$A^\bullet\in \Ch(\A)$. An {\bf injective Cartan-Eilenberg resolution} or, more briefly, a \CE{\bf -resolution}  of $A^\bullet$ is given by a lower half-plane bicomplex $C^{\bullet,\bullet}$ of injectives, and a homomorphism of complexes $\lambda^\bullet\colon A^\bullet \to C^{0,\bullet}$:
\[
\xymatrix@C=50pt@R=3pt{
\cdots \ar[r]&A^{n-1}\ar[r]^-{a^{n-1}}\ar@/_-5pt/[dd]^(.35){\lambda^{n-1}}&A^n\ar[r]^-{a^n}\ar@/_-5pt/[dd]^(.35){\lambda^{n}}&A^{n+1}\ar[r]^-{a^{n+1}}\ar@/_-5pt/[dd]^(.35){\lambda^{n+1}}&\cdots\\
\ar@{.}[rrrr]&&&&\\
\cdots \ar[r]&C^{0,n-1}\ar[ddd]^-{d_1^{0,n-1}}\ar[r]^-{d_0^{0,n-1}}&C^{0,n}\ar[r]^-{d_0^{0,n}}\ar[ddd]^-{d_1^{0,n}}&C^{0,n+1}\ar[r]^-{d_0^{0,n+1}}\ar[ddd]^-{d_1^{0,n+1}}&\cdots\\
\\
\\
\cdots \ar[r]&C^{1,n-1}\ar[r]^-{d_0^{1,n-1}}\ar[dd]^-{d_1^{1,n-1}}&C^{1,n}\ar[r]^-{d_0^{1,n}}\ar[dd]^-{d_1^{1,n}}&C^{1,n+1}\ar[r]^-{d_0^{1,n+1}}\ar[dd]^-{d_1^{1,n+1}}&\cdots\\
\\
&\vdots&\vdots&\vdots}
\]
that satisfy the following properties:
\begin{enumerate}
\item[({\CE.}$1$)] for each $n\in \Z$, if $A^n=0$, then the column $C^{\bullet,n}$ is constantly $0$;
\item[({\CE.}$2$)] $B^n(\lambda^\bullet)\colon B^n(A^\bullet)\to B^n(C^{\bullet,\bullet}_0)$  is an injective resolution of $B^n(A^\bullet)$, for all $n\in\Z$;
\item[({\CE.}$3$)]  $H^n(\lambda^\bullet)\colon H^n(A^\bullet)\to H^n(C^{\bullet,\bullet}_0)$  is an injective resolution of $H^n(A^\bullet)$, for all $n\in\Z$.
\end{enumerate}
If these three properties hold true, then one also gets the following  two extra properties (which are sometimes  included in the definition) for free (see \cite[Exercise~5.7.1]{Weibel}):
\begin{enumerate}
\item[({\CE.}$4$)] $Z^n(\lambda^\bullet)\colon Z^n(A^\bullet)\to Z^n(C^{\bullet,\bullet}_0)$  is an injective resolution of $Z^n(A^\bullet)$, for all $n\in\Z$;
\item[({\CE.}$5$)]  $\lambda^n\colon A^n\to C_0^{\bullet,n}$  is an injective resolution of $A^n$, for all $n\in\Z$.
\end{enumerate}

It is well-known (see, e.g., \cite[Lemma~5.7.2]{Weibel}) that any complex $A^\bullet$ admits a \CE-resolution. Let us briefly go through the (dual of the) construction given by Weibel. The whole process is based on the Horseshoe Lemma, so let us briefly recall it here:
\begin{lem}[Horseshoe Lemma]
Let $\A$ be an Abelian category and consider the following solid diagram, whose first and third row are injective resolutions of $X$ and $Z$, respectively, the columns are exact, with $\iota^n$ and $\pi^n$ ($n\in\N$) being the obvious inclusions and projections:
\[
\xymatrix@C=30pt@R=5pt{
0\ar[d]&0\ar[d]&0\ar[d]&0\ar[d]\\
X\ar[dd]\ar[r]&E^0_X\ar[r]\ar[dd]^-{\iota^0}&E_X^1\ar[r]\ar[dd]^-{\iota^1}&E_X^2\ar[r]\ar[dd]^-{\iota^2}&\cdots\\
\\
Y\ar@{.>}[r]\ar[dd]&E^0_X\oplus E^0_Z\ar@{.>}[r]\ar[dd]^-{\pi^0}&E_X^1\oplus E^1_Z\ar@{.>}[r]\ar[dd]^-{\pi^1}&E_X^2\oplus E^2_Z\ar@{.>}[r]\ar[dd]^-{\pi^2}&\cdots\\
\\
Z\ar[d]\ar[r]&E^0_Z\ar[r]\ar[d]&E_Z^1\ar[d]\ar[r]\ar[d]&E_Z^2\ar[d]\ar[r]\ar[d]&\cdots\\
0&0&0&0&
}
\]
One can then complete the second row to an injective resolution, so that the diagram commutes. 
\end{lem}

To construct a \CE-resolution of $A^\bullet$, one proceeds as follows: 
\begin{itemize}
\item[(step.1)] for all $n\in\Z$, fix injective resolutions $\lambda_{B^n}\colon B^n(A^\bullet)\to E_{B^n}^\bullet$ and $\lambda_{H^n}\colon H^n(A^\bullet)\to E_{H^n}^\bullet$;
\item[(step.2)] for all $n\in\Z$, use the  Horseshoe Lemma to combine these resolutions along the sequence  $0\to B^n(A^\bullet)\to Z^n(A^\bullet)\to H^n(A^\bullet)\to 0$, to get a new resolution  $\lambda_{Z^n}\colon Z^n(A^\bullet)\to E_{Z^n}^\bullet$; 
\item[(step.3)] for all $n\in\Z$, use the  Horseshoe Lemma to combine  the previous injective resolutions along  $0\to Z^n(A^\bullet)\to A^n\to B^{n+1}(A^\bullet)\to 0$, to get a new injective resolution  $\lambda_{A^n}\colon A^n\to E_{A^n}^\bullet$;
\item[(step.4)] for all $n\in\Z$ and $m\in \N$, let $C^{m,n}:=E_{A^n}^m$, with $d_1^{\bullet,n}$ using the differentials of $E_{A^n}^\bullet$ multiplied by $(-1)^{n}$,  build  $d_0^{\bullet,n}$ as a composition $E_{A^n}^\bullet\twoheadrightarrow E^\bullet_{B^{n+1}}\hookrightarrow E_{Z^{n+1}}^\bullet\hookrightarrow E_{A^{n+1}}^\bullet$ of the projection and embeddings given by the  Horseshoe Lemma, and let  $\lambda^\bullet\colon A^\bullet \to C^{0,\bullet}$ be $\lambda^n:=\lambda_{A^n}$.
\end{itemize}
By \cite[Lemma~5.7.2]{Weibel}, these $C^{\bullet,\bullet}$ and $\lambda^\bullet\colon A^\bullet \to C^{0,\bullet}$ are a \CE-resolution. Furthermore, the following lemma is a consequence of the dual of \cite[Exercise~5.7.1]{Weibel}:

\begin{lem}\label{CE_bounded_lemma}
Let $\A$ be a complete Abelian category with enough injectives, $A^\bullet\in \Ch^{+}(\A)$, and take a \CE-resolution $\lambda^\bullet\colon A^\bullet\to C^{\bullet,\bullet}$. Then, $\Tot(\lambda^\bullet)\colon A^\bullet\to \Tot(C^{\bullet,\bullet})$ is a quasi-isomorphism.
\end{lem}

Observe that, in the setting of the above lemma, $\Tot(C^{\bullet,\bullet})$ is degree-wise injective and, by (\CE.$1$), this complex is bounded below. Hence, by Proposition~\ref{prop_known_facts_cotors}, $\Tot(\lambda^\bullet)\colon A^\bullet\to \Tot(C^{\bullet,\bullet})$ is a \DG-injective resolution. Therefore, the totalization of \CE-resolutions provides a valid alternative to the construction in Subsection~\ref{bounded_subs}.

\subsection{{\rm\CE}-resolutions and the (Ab.$4^*$)-$k$ condition}
In \cite[Exercise~5.7.1]{Weibel}, Weibel suggests that Lemma~\ref{CE_bounded_lemma} works for unbounded complexes, provided $\A$ is (Ab.$4^*$). In fact, this assumption can be weakened to (Ab.$4^*$)-$k$:

\begin{thm}\label{CE_are_OK_thm}
Let $\A$ be a complete Abelian category with enough injectives, let $A^\bullet\in \Ch(\A)$ be a complex, and consider a \CE-resolution $\lambda^\bullet\colon A^\bullet\to C^{\bullet,\bullet}$. If $\A$ is (Ab.$4^*$)-$k$ for some $k\in\N$, the induced map $A^\bullet\to \Tot(C^{\bullet,\bullet})$ is a \DG-injective resolution of $A^\bullet$.\end{thm}
\begin{proof}
By definition of \CE-resolution, we get injective resolutions $\lambda_{B^n}\colon B^n(A^\bullet)\to E_{B^n}^\bullet:=B^n(C^{\bullet,\bullet}_0)$ and $\lambda_{H^n}\colon H^n(A^\bullet)\to E_{H^n}^\bullet:=H^n(C^{\bullet,\bullet}_0)$, for each $n\in\Z$, so that $C^{n,m}\cong E_{B^n}^m\oplus E_{H^n}^m\oplus E_{B^{n+1}}^{m}$, for all $m\in\N$ and $n\in\Z$. Observe that, for each $n\in\N$, there is a \CE-resolution $\lambda_n^\bullet\colon \tau^{\geq -n}(A^\bullet)\to C_{\geq -n}^{\bullet,\bullet}$, where $C_{\geq -n}^{\bullet,i}=C^{\bullet,i}$ for all $i>-n$, $C_{\geq -n}^{\bullet,i}=0$ for all $i<n$, and $C_{\geq -n}^{\bullet,n}=E_{H^n}^m\oplus E_{B^{n+1}}^{m}$ (with differentials induced by those of $C^{\bullet,\bullet}$). Denoting by $\pi_n^\bullet\colon \tau^{\geq -n-1}A^\bullet\to \tau^{\geq -n}(A^\bullet)$ and $\gamma^{\bullet,\bullet}_n\colon C_{\geq-n-1}^{\bullet,\bullet}\to C_{\geq -n}^{\bullet,\bullet}$ the obvious projections and taking the totalizations, we obtain the following Spaltenstein tower of partial resolutions (see  Lemma~\ref{CE_bounded_lemma}):
\[
\xymatrix@C=40pt@R=18pt{
\cdots\ar[r]^-{\pi^\bullet_{-2}}&\tau^{\geq-2}(A^\bullet)\ar[d]^-{\lambda_2^\bullet}\ar[r]^-{\pi^\bullet_{-1}}&\tau^{\geq-1}(A^\bullet)\ar[d]^-{\lambda_1^\bullet}\ar[r]^-{\pi^\bullet_{0}}&\tau^{\geq0}(A^\bullet)\ar[d]^-{\lambda_0^\bullet} \\
\cdots\ar[r]^-{\Tot(\gamma_{-2}^{\bullet,\bullet})}&\Tot(C^{\bullet,\bullet}_{\geq -2})\ar[r]^-{\Tot(\gamma_{-1}^{\bullet,\bullet})}&\Tot(C^{\bullet,\bullet}_{\geq -1})\ar[r]^-{ \Tot(\gamma_{0}^{\bullet,\bullet})}&\Tot(C^{\bullet,\bullet}_{\geq 0}). 
}
\]
By Theorem~\ref{main_unbounded_dg_thm} (that can be applied here as we are assuming that $\A$ is (Ab.$4^*$)-$k$ for some $k\in\N$), the complex $\varprojlim_{i\in \N}\Tot(C_{\geq-i}^{\bullet,\bullet})\cong \Tot(C^{\bullet,\bullet})$ is a \DG-injective resolution of $A^\bullet$. 
\end{proof}

 To understand what can go wrong if we do not assume (Ab.$4^*$)-$k$, let us construct a \CE-resolution for the complex $X^\bullet:= \mathrlap{\text{\tiny \hspace{4.3pt}$\G$}}\prod_{i\in \Z}S^i(\Q(R))\in \Ch(\G)$ from Section~\ref{poison_sec}. For this specific complex: $B^n(X^\bullet)=0$ and $Z^n(X^\bullet)\cong X^n\cong H^n(X^\bullet)\cong \Q(R)$, for all $n\in\Z$. As in Lemma~\ref{res_of_R_lem}, we choose an injective resolution 
$
\lambda\colon R\to E^\bullet
$ 
of $R$ in $\Mod(R)$, so that $\mathbf Q(E^\bullet)\in \Ch^{\geq0}(\G)$ is an injective resolution of $\Q(R)$ in $\G$. Hence, the following is a \CE-resolution of $X^\bullet$:
\begin{equation}\label{CE_of_poison_eq}
\xymatrix@C=50pt@R=3pt{
\cdots \ar[r]&\Q(R)\ar[r]^0\ar@/_-5pt/[dd]^(.35){\mathbf Q(\lambda)}&\Q(R)\ar[r]^0\ar@/_-5pt/[dd]^(.35){\mathbf Q(\lambda)}&\Q(R)\ar[r]^0\ar@/_-5pt/[dd]^(.35){\mathbf Q(\lambda)}&\cdots\\
\ar@{.}[rrrr]&&&&\\
\cdots \ar[r]&\mathbf Q(E^0)\ar[ddd]^-{\mathbf Q(e^0)}\ar[r]^-{0}&\mathbf Q(E^0)\ar[r]^-{0}\ar[ddd]^-{\mathbf Q(e^0)}&\mathbf Q(E^0)\ar[r]^-{0}\ar[ddd]^-{\mathbf Q(e^0)}&\cdots\\
\\
\\
\cdots \ar[r]&\mathbf Q(E^1)\ar[r]^-{0}\ar[dd]^-{\mathbf Q(e^1)}&\mathbf Q(E^1)\ar[r]^-{0}\ar[dd]^-{\mathbf Q(e^1)}&\mathbf Q(E^1)\ar[r]^-{0}\ar[dd]^-{\mathbf Q(e^1)}&\cdots\\
\\
&\vdots&\vdots&\vdots}
\end{equation}
Totalizing this \CE-resolution one gets the diagonal map $\mathrlap{\text{\tiny \hspace{4.3pt}$\G$}}\prod_{i\in \Z}\Sigma^i\mathbf{Q}(\lambda)\colon X^\bullet\to \mathrlap{\text{\tiny \hspace{4.3pt}$\G$}}\prod_{i\in \Z}\Sigma^i\mathbf Q(E^\bullet)$, which is not a quasi-isomorphism by Proposition~\ref{poison_in_action_prop}.

\section{Saneblidze's construction through multicomplexes}\label{Saneblidze_sec}\label{sec_five}

A different construction of general \DG-injective resolutions for (possibly unbounded) complexes is proposed in \cite[Proposition~3]{saneblidze2007derived}. Given an Abelian category $\A$, which is just supposed to have countable products and enough injectives, and $X^\bullet\in \Ch(\A)$, Saneblidze constructs what he calls a homological injective multicomplex $E^{\bullet,\bullet}$ (see below for unexplained terminology) whose totalization provides, supposedly, a \DG-injective resolution for $X^\bullet$. In this section, after recalling the necessary background and some details about the construction proposed in \cite{saneblidze2007derived}, we show that the same example considered in Section~\ref{poison_sec} is also a counterexample to  this construction.

\subsection{Cohomological injective multicomplexes and resolutions} 
Let $\A$ be an Abelian category. A {\bf multicomplex} $C^{\bullet,\bullet}$ over $\A$ is given by the following data:
\begin{itemize}
\item a family of objects $\C^{\bullet,\bullet}=\{C^{n,m}:(n,m)\in\Z\times \Z\}$ in $\A$;
\item for each $r\in\N$ a degree $r$ differential $d_r^{\bullet,\bullet}\colon C^{\bullet,\bullet}\to C^{\bullet+r,\bullet-r+1}$, such that the following condition is satisfied for all $i\in\N$ and all $(n,m)\in\Z\times \Z$
\begin{equation}\label{condition_multi_eq}
\sum_{r+s=i}d_s^{n+r,m-r+1}\circ d_r^{n,m}=0.
\end{equation}
\end{itemize} 
As we did for bicomplexes in Subsection~\ref{CE_sec}, we restrict our attention to {\bf lower half-plane multicomplexes} $C^{\bullet,\bullet}$, which can be visualized as follows (where labels denote the degree of the arrows) 
\[
\scalebox{0.85}{
\xymatrix@R=10pt@C=30pt{
\cdots \ar[r]|-{\boxed{\text{\tiny $0$}}}&C^{0,m-2}\ar[dd]|-{\boxed{\text{\tiny $1$}}}\ar[rr]|(.60){\boxed{\text{\tiny $0$}}}&&C^{0,m-1}\ar@/_12pt/[lldddd]|(.75){\boxed{\text{\tiny $2$}}}\ar[dd]|-{\boxed{\text{\tiny $1$}}}\ar[rr]|(.60){\boxed{\text{\tiny $0$}}}&&C^{0,m}\ar[rr]|(.60){\boxed{\text{\tiny $0$}}}\ar[dd]|-{\boxed{\text{\tiny $1$}}}\ar@/_10pt/[lllldddddd]|(.80){\boxed{\text{\tiny $3$}}}\ar@/_12pt/[lldddd]|(.75){\boxed{\text{\tiny $2$}}}&&C^{0,m+1}\ar@/_10pt/[lllldddddd]|(.80){\boxed{\text{\tiny $3$}}}\ar@/_12pt/[lldddd]|(.75){\boxed{\text{\tiny $2$}}}\ar[r]|-{\boxed{\text{\tiny $0$}}}\ar[dd]|-{\boxed{\text{\tiny $1$}}}&\cdots\\
\\
\cdots \ar[r]|-{\boxed{\text{\tiny $0$}}}&C^{1,m-2}\ar[rr]|(.60){\boxed{\text{\tiny $0$}}}\ar[dd]|(.56){\boxed{\text{\tiny $1$}}}&&C^{1,m-1}\ar@/_12pt/[lldddd]|(.75){\boxed{\text{\tiny $2$}}}\ar[rr]|(.60){\boxed{\text{\tiny $0$}}}\ar[dd]|(.56){\boxed{\text{\tiny $1$}}}&&C^{1,m}\ar@/_12pt/[lldddd]|(.75){\boxed{\text{\tiny $2$}}}\ar[rr]|(.60){\boxed{\text{\tiny $0$}}}\ar[dd]|-{\boxed{\text{\tiny $1$}}}&&C^{1,m+1}\ar@/_12pt/[lldddd]|(.75){\boxed{\text{\tiny $2$}}}\ar[r]|-{\boxed{\text{\tiny $0$}}}\ar[dd]|-{\boxed{\text{\tiny $1$}}}&\cdots\\
\\
\cdots \ar[r]|-{\boxed{\text{\tiny $0$}}}&C^{2,m-2}\ar[rr]|(.68){\boxed{\text{\tiny $0$}}}\ar[dd]|(.56){\boxed{\text{\tiny $1$}}}&&C^{2,m-1}\ar[rr]|(.70){\boxed{\text{\tiny $0$}}}\ar[dd]|(.56){\boxed{\text{\tiny $1$}}}&&C^{2,m}\ar[rr]|(.60){\boxed{\text{\tiny $0$}}}\ar[dd]|(.56){\boxed{\text{\tiny $1$}}}&&C^{2,m+1}\ar[r]|-{\boxed{\text{\tiny $0$}}}\ar[dd]|(.56){\boxed{\text{\tiny $1$}}}&\cdots\\
\\
\cdots \ar[r]|-{\boxed{\text{\tiny $0$}}}&C^{3,m-2}\ar[rr]|(.68){\boxed{\text{\tiny $0$}}}\ar[dd]&&C^{3,m-1}\ar[rr]|(.70){\boxed{\text{\tiny $0$}}}\ar[dd]&&C^{3,m}\ar[rr]|(.60){\boxed{\text{\tiny $0$}}}\ar[dd]&&C^{3,m+1}\ar[r]|-{\boxed{\text{\tiny $0$}}}\ar[dd]&\cdots\\
\\
&\vdots&&\vdots&&\vdots&&\vdots}}
\]
For $i=0$,  \eqref{condition_multi_eq} means that each row is a cochain complex while, for $i=1$, it says that the little squares anticommute. Moreover, if  $d_0^{\bullet,\bullet}=0$ or $d_2^{\bullet,\bullet}=0$, then \eqref{condition_multi_eq} for $i=2$ says that each column is a cochain complex. Hence, a bicomplex is just a multicomplex such that $d^{\bullet,\bullet}_r=0$, for all $r\geq 2$.

A {\bf morphism of multicomplexes} $(\phi_r^{\bullet,\bullet})_{r\in\N}\colon X^{\bullet,\bullet}\to Y^{\bullet,\bullet}$  is specified by a family of morphisms
$
\phi^{(n,m)}_r\colon X^{(n,m)}\to Y^{(n+r,m-r)}
$
which is compatible with differentials in the following sense:
$$
\sum_{r+s=i}\phi^{(n+r,m-r+1)}_s\circ d^{(n,m)}_r=\sum_{r+s=i}d^{(n+r,m-r)}_s\circ \phi^{(i,j)}_r\qquad\text{for all $(n,m)\in \mathbb Z\times \mathbb Z$ and $i\in\mathbb N$.}
$$
\begin{defn}[{\cite[Section~2]{saneblidze2007derived}}]
A lower half-plane multicomplex of injectives $E^{\bullet,\bullet}$ is said to be a {\bf homological injective multicomplex} if:
\begin{itemize}
\item $d_0^{\bullet,\bullet}=0$ (so that each column $E^{\bullet,m}\in \Ch^{\geq0}(\A)$ with the degree $1$ differentials);
\item for each $m\in \Z$, $H^i(E^{\bullet,m})=0$ for all $i\geq 1$, that is, the column $E^{\bullet,m}$ is an injective resolution of the object $H^0(E^{\bullet,m})=\Ker(d_1^{0,m})\in \A$. 
\end{itemize}
\end{defn}
 Finally, if $\A$ is complete, given a lower half-plane multicomplex $C^{\bullet,\bullet}$, one can consider its {\bf totalization}, which is a cochain complex $\Tot(C^{\bullet,\bullet})\in \Ch(\A)$ such that $\Tot(C^{\bullet,\bullet})^i:=\prod_{j=0}^\infty C^{j,i-j}$, for all $i\in \Z$, where the $n$-th differential is
\[
\partial^i:=(\partial^i_j)_{j\in\N}\colon \prod_{j\in\N} C^{j,i-j}\longrightarrow \prod_{j\in\N} C^{j,i+1-j},
\]
where $\partial^i_j$ is determined by its components $(d_0^{j,i-j},\, d_1^{j,i-j},\, d_2^{j,i-j}, \dots)^t\colon C^{j,i-j}\to  \prod_{r\in\N} C^{j+r,i+1-j-r}$.

\begin{defn}[{\cite[Section~2]{saneblidze2007derived}}]
A {\bf homological injective resolution} of a complex $X^\bullet \in\Ch(\A)$ is a morphism of multicomplexes $\lambda=(\lambda^\bullet_r)_{r\in \N}\colon X^\bullet \to E^{\bullet,\bullet}$ such that $E^{\bullet,\bullet}$ is a homological injective multicomplex and $\Tot(\lambda)\colon X^\bullet\to \Tot(E^{\bullet,\bullet})$ is a quasi-isomorphism. 
\end{defn}

\subsection{A counterexample for the construction}
In \cite[Proposition~3]{saneblidze2007derived}, Saneblidze states that, if $\A$ is a cocomplete Abelian category with enough projectives, then every complex has a homological projective resolution. In the proof of \cite[Proposition~3]{saneblidze2007derived} one can find a rather involved construction of a homological resolution of an arbitrary complex. Let us recall the first few steps (in the dual setting of homological injective resolutions). Indeed, let $A^\bullet\in \Ch(\A)$ be a given complex. To build a homological injective resolution $\lambda=(\lambda^\bullet_r)_{r\in \N}\colon A^\bullet \to E^{\bullet,\bullet}$, consider an injective resolution $\lambda^n\colon H^n(A^\bullet)\to E_{H^n}^\bullet=(E_{H^n}^0\to E_{H^n}^1\to \cdots)$, for each $n\in\Z$, and define:
\begin{itemize}
\item $E^{m,n}:=E^m_{H^n}$ for all $m\in\N$ and $n\in\Z$;
\item  define the differentials of degree $0$ of $E^{\bullet,\bullet}$ to be trivial;
\item  use the differentials of $E_{H^n}^\bullet$ to define the differentials of degree $1$ in $E^{\bullet,\bullet}$;
\item let $\lambda^n_0\colon A^n\to E^0_{H^n}$  be any lifting of the composition $Z^n(A^\bullet)\to H^n(A^\bullet)\to E_{H^n}^0$ along the inclusion $Z^n(A^\bullet)\to A^n$ (which exists by the injectivity of $E_{H^n}^0$). 
\end{itemize}
Hence, after this first step of the construction, we have all the object components of $E^{\bullet,\bullet}$, its horizontal and vertical differentials, and the $0$-degree component of $\lambda$. As it turns out, since our goal is to apply the construction to the complex $X^\bullet\in \Ch(\G)$ of Section~\ref{poison_sec}, which is a complex with trivial differentials, we can just stop here with the construction:

\begin{rmk}[{\cite[pag.\,320]{saneblidze2007derived}}]
As observed by Saneblidze  right after  \cite[Proposition~3]{saneblidze2007derived}, and as it can be easily checked in his proof, if all the differentials of $A^\bullet$ are trivial, that is, if $A^\bullet\cong \prod_{n\in\Z}S^{-n}(A^n)$, one can take $d^{\bullet,\bullet}_r=0$ for all $r\geq 2$ and $\lambda^\bullet_s=0$, for all $s\geq 1$. Hence, in this case, $E^{\bullet,\bullet}$ is just a bicomplex with trivial horizontal differentials, and $\lambda=\lambda^\bullet\colon A^\bullet \to E^{0,\bullet}$.
\end{rmk} 

By the above discussion and remark, it is  easy to see that the result of applying Saneblidze's construction to the complex $X^\bullet\in \Ch(\G)$ from Section~\ref{poison_sec} is the  bicomplex  we have constructed in \eqref{CE_of_poison_eq} when studying \CE-resolutions. As we have already seen, the totalization of that bicomplex is not quasi-isomorphic to $X^\bullet$ and, therefore, Saneblidze's construction does not produce a homological injective resolution in this case.

\section{Ding and Yang's construction via repeated killing of coboundaries}\label{DY_sec}\label{sec_six}

Let $\A$ be a bicomplete Abelian category, $(\X,\Y)$ a complete hereditary cotorsion pair in $\A$, and consider the induced cotorsion pairs $(\mathrm{dg}\X,\widetilde\Y)$ and  $(\widetilde\X,\mathrm{dg}\Y)$ in $\Ch(\A)$, introduced by Gillespie \cite[Definition~3.1]{gillespie2008cotorsion}. In many important cases, these cotorsion pairs are complete, hereditary and compatible,  so they  give rise to an Abelian model structure in $\Ch(\A)$. On the other hand, to the best of our knowledge, the following question remains open:
\begin{quest}\label{Gillespie_question}
Let $\A$ be a bicomplete Abelian category and let $(\X,\Y)$ be a complete and hereditary cotorsion pair in $\A$. Are the induced cotorsion pairs $(\mathrm{dg}\X,\widetilde\Y)$ and  $(\widetilde\X,\mathrm{dg}\Y)$ complete in $\Ch(\A)$?
\end{quest}

An attempt to solve this problem in the positive was made by N.\,Ding and X.\,Yang. The argument they used is based on \cite[Lemma~2.1]{yang2015question}, whose proof contains a very concrete construction, for each  $A^\bullet\in\Ch(\A)$,  of a $Y^\bullet\in\mathrm{dg}\Y$ which is quasi-isomorphic to $A^\bullet$. In particular, if we start with the trivial cotorsion pair $(\A,\Inj(\A))$ in $\A$ (assuming that $\A$ has enough injectives), this construction should produce a \DG-injective resolution of $A^\bullet$. 

In this section, we test the proof of \cite[Lemma~2.1]{yang2015question} against the complex $X^\bullet\in \Ch(\G)$ from Section~\ref{poison_sec}, showing that it actually fails to produce the desired resolution. Furthermore, we briefly analyze Ding and Yang's proof, pointing out a concrete gap in the argument. Moreover, we observe that the problem completely disappears if we assume that $\A$ is (Ab.$4^*$) for the first half of the statement and, dually, that $\A$ is (Ab.$4$) for the second half. 

\subsection{The Ding-Yang construction in concrete situations}\label{subs_tech_Ding_Yang}

Let $\A$ be a bicomplete Abelian category with enough injectives, so that the cotorsion pair $(\A,\Inj(\A))$ is complete. Let us start recalling the idea of the proof of  \cite[Lemma~2.1]{yang2015question}  in this very special case. Indeed, the construction of Ding and Yang is based on a very simple operation, of ``killing coboundaries" that we have condensed in the following 

\begin{con}[Killing coboundaries]\label{killing_cob_const}
Let $\A$ be an Abelian category with enough injectives,  fix a complex $X^\bullet\in \Ch(\A)$ and  $n\in \Z$. Denote by $\rho^n\colon X^n\to X^n/B^n(X^\bullet)$ the obvious projection and by $\iota^n\colon X^n/B^n(X^\bullet)\to E$ an inclusion into an injective object $E\in\Inj(\A)$. Define a new complex:
\[
K{(X^\bullet,n)}\ :\qquad \xymatrix@C=20pt{\cdots\ar[r]&X^{n-1}\ar[rr]^-{d^{n-1}}&&X^{n}\ar[rr]^-{\text{\tiny$\left[\begin{matrix}{d^n} \\ \iota_n\rho_n\end{matrix}\right]$}}&&X^{n+1}\oplus E\ar[rr]^-{\text{\tiny$\left[\begin{matrix}d^{n+1} & 0\end{matrix}\right]$}}&&X^{n+2}\ar[r]^-{d^{n+2}}&\cdots}
\]
and a morphism of complexes $\pi^\bullet_{n}\colon K{(X^\bullet,n)}\to X^\bullet$ such that 
\[
\pi_{n}^m:=\begin{cases} 
\id_{X^m}&\text{for all $m\neq n+1$;}\\
\text{the obvious projection $X^{n+1}\oplus E\to X^{n+1}$}&\text{otherwise.}
\end{cases}
\]
Then, $H^n(K{(X^\bullet,n)})=0$, $\pi^\bullet_{n}$ is a degree-wise split-epic, and $\Ker(\pi^\bullet_{n})=S^{n+1}(E)$ is \DG-injective.
\end{con}
Observe that $K{(X^\bullet,n)}$ depends on the choice of the embedding $\iota^n\colon X^n/B^n(X^\bullet)\to E$. In some cases this choice can be made canonical, e.g., whenever $X^n/B^n(X^\bullet)$ is trivial or, more generally, when it is an injective object, it is natural to take $\iota^n:=\id$.  
To illustrate the construction in some basic case, take an injective object $E\in \A$ and consider the stalk complex $S^i(E)$ and the disk complex $D^i(E)$, for some $i\in \Z$. Then,
\[
K(S^i(E),j)=\begin{cases}S^i(E)&\text{if $j\neq i$;}\\
D^i(E)&\text{if $j=i$;}
\end{cases}
\qquad \text{and}\qquad K(D^i(E),j)=\begin{cases}D^i(E)&\text{if $j\neq i$;}\\
D^i(E)\oplus S^{i+1}(E)&\text{if $j=i$.}
\end{cases}
\]
In fact, it is not completely trivial to verify that $K(D^i(E),i)\cong D^i(E)\oplus S^{i+1}(E)$ since the $i$-th differential of $K(D^i(E),i)$ is, in principle, just a triangular (but not necessarily a diagonal) matrix. On the other hand, most of these complications can be avoided by using the following lemma to simplify
\begin{lem}\label{reordering_differentials}
Let $X^\bullet=(X^n,d^n)_{n\in\Z}\in\Ch(\A)$ and suppose that, for a given $n\in \Z$, $X^{n}\cong Y^{(k)}$ is a coproduct of $k$-many copies of a given $Y\in \A$, and $d^{n-1}=
(\varphi,\varphi,\dots,\varphi)^t\colon X^{n-1}\to Y^{(k)}$ is the diagonal map for a suitable $\varphi\colon  X^{n-1}\to Y$. Then, there is an isomorphism $\phi^\bullet\colon X^\bullet\to (X')^{\bullet}=((X')^n,(d')^n)_{n\in\Z}$ in $\Ch(\A)$, where $(X')^i:=X^i$ for all $i\in \Z$ and, if $d^n=(\psi_1,\psi_2, \dots, \psi_k)\colon Y^{(k)}\to X^{n+1}$, 
\[
(d')^i:=\begin{cases}
(\varphi, 0, \dots, 0)^t&\text{if $i=n-1$;}\\
(\psi_1+\ldots+\psi_k,\psi_2, \dots, \psi_k)&\text{if $i=n$;}\\
  (d')^i:=d^i & \text{for all $i\in \Z\setminus\{ n-1,\ n\}$}.
\end{cases}
\]

\end{lem}
\begin{proof}
One can just define $\phi^i:=\id_{X^i}$ for all $i\neq n$, and 
\[
\phi^n:={\text{\tiny$\left[\begin{matrix}+\id_Y& 0 &  \ldots&0&0\\
-\id_Y&\id_Y&\ldots&0&0\\
\vdots&\vdots&\ddots&\vdots&\vdots\\
-\id_Y&0&\ldots&\id_Y&0\\
-\id_Y&0&\ldots&0&\id_Y\end{matrix}\right]$}}\in \mathrm{Mat}_{k\times k}(\mathrm{End}_\A(Y))\cong\mathrm{End}_\A(Y^{(k)}).\qedhere
\]
\end{proof}

Consider now a family of complexes $\{X^\bullet_n\}_{n\in \N}\in \Ch(\A)$ such that 
\[
\left|\{n\in \N:X_n^i\neq 0\}\right|<\infty\ ,\quad \text{ for each $i\in\Z$.}
\] 
Observe that, for such a family,  $\prod_\N X^\bullet_n=\bigoplus_\N X^\bullet_n$. Hence, letting $Y^\bullet:=\prod_\N X^\bullet_n=\bigoplus_\N X^\bullet_n$, we have that $B^i(Y^\bullet)=\bigoplus_\N B^i(X^\bullet_n)=\prod_\N B^i(X^\bullet_n)$ (where both equalities use that finite products coincide with finite coproducts in any Abelian category and, hence, they are exact). Similarly, one can verify that $Z^i(Y^\bullet)=\prod_{\N}Z^i(X^\bullet_n)=\bigoplus_{\N}Z^i(X^\bullet_n)$ and, combining these two, one even gets $H^i(Y^\bullet)=\prod_{\N}H^i(X^\bullet_n)=\bigoplus_{\N}H^i(X^\bullet_n)$.

\begin{lem}\label{killing_sums_and_products}
Let $\{X^\bullet_n\}_{n\in \N}\in \Ch(\A)$ be a family as above and let $Y^\bullet:=\prod_{\N}X^\bullet_n$. Then, for each $i\in \Z$, one can choose suitable embeddings into injectives in Construction~\ref{killing_cob_const} so to obtain $K{(Y^\bullet,i)}=\prod_{\N}K(X_n^\bullet,i)$.
\end{lem}

Going back to Ding and Yang's proof of their Lemma~2.1, they fix an arbitrary surjection $\tau\colon \N\to \Z$ with $n_i:=\tau(i)$, for all $i\in \N$, with the property of having infinite fibers. As we are trying to do some concrete computations, we need to make a  choice. Indeed, we fix the following sequence of integers (in which every integer appears infinitely many times, as desired):
\begin{align*}
n_0&=0,\\
n_1=-1,\ n_2&=0,\ n_3=1,\\
n_4=-2,\ n_5=-1,\ n_6&=0,\ n_7=1,\ n_8=2,\\
n_9=-3,\ n_{10}=-2,\ n_{11}=-1,\ n_{12}&=0,\ n_{13}=1,\ n_{14}=2,\ n_{15}=3,\\
&\ldots
\end{align*}
Now, starting with our complex $X^\bullet\in\Ch(\A)$,  the idea of the Ding-Yang construction is to iterate Construction~\ref{killing_cob_const} in order to build the following sequential inverse system:
\[
\xymatrix{
\cdots\ar[r]&Y^\bullet_i\ar[r]^-{\pi^\bullet_i}&\cdots\ar[r]^-{\pi^\bullet_2}&Y^\bullet_1\ar[r]^-{\pi^\bullet_1}&Y^\bullet_0\ar[r]^-{\pi^\bullet_0}&Y^\bullet_{-1}:=X^\bullet
}
\qquad\text{in $\Ch(\A)$,}
\]
where $Y^\bullet_{i}:=K(Y^\bullet_{i-1},n_{i})$ and $\pi^\bullet_i$ is the obvious degree-wise split-epimorphism with \DG-injective kernel, for all $i\in \N$. Defining $Y^\bullet:=\varprojlim_{\N}Y^\bullet_i$, one gets an exact sequence $0\to E^\bullet\to Y^\bullet\to X^\bullet\to 0$ in $\Ch(\A)$, with $E^\bullet$ a \DG-injective complex. The claim made in \cite{yang2015question} is that the complex $Y^\bullet$ is exact, so there is a quasi-isomorphism $X^\bullet\to \Sigma E^\bullet$, which is precisely the \DG-injective resolution of $X^\bullet$ we are looking for. In the rest of this section we will apply this strategy in some easy cases: first with $X^\bullet$ a stalk complex, and then a product of stalk complexes. As we will see, if the product of stalk complexes $X^\bullet$ is the complex from Section~\ref{poison_sec}, then the resulting complex $Y^\bullet$ is not exact, showing that the proof of \cite[Lemma~2.1]{yang2015question} may fail even for the trivial cotorsion pair $(\G,\Inj(\G))$, with $\G$ a Grothendieck category.

\smallskip
As a first example, let us apply the Ding-Yang construction to a stalk complex $X=S^0(A)$ concentrated in degree $0$, for some $A\in \A$. Fix an injective resolution $\varphi\colon A\to E^\bullet$ in $\A$, where we consider $E^\bullet$ as a complex concentrated in degrees $\geq 1$:
\[
E^\bullet:=(\xymatrix@C=20pt{
\cdots\ar[r]&0\ar[r]& E^1\ar[r]^-{\lambda^1}& E^{2}\ar[r]^-{\lambda^2} & E^{3}\ar[r]^-{\lambda^3} &\cdots ).
}
\] 
Moreover, let us also consider the following exact complex concentrated in non-negative degrees:
\[
E_A^{[0,\infty)}:=(\xymatrix@C=20pt{
\cdots\ar[r]&0\ar[r]&A\ar@{^(->}[r]^-{\varphi} & E^1\ar[r]^-{\lambda^1}& E^{2}\ar[r]^-{\lambda^2} & E^{3}\ar[r]^-{\lambda^3} &\cdots ),
}
\]
and, for each $m\in\N$, let $E_A^{[0,m]}$ be the naive truncation of $E_A^{[0,\infty)}$ so that, for example, 
\[
E_A^{[0,0]}=S^0(A),\quad  E_A^{[0,1]}=(\cdots \longrightarrow 0\longrightarrow A\longrightarrow E^1\longrightarrow 0\longrightarrow\cdots),\quad \text{ and so on.}
\]
We can now start computing the complexes of the form $Y^\bullet_{i}:=K(Y^\bullet_{i-1},{n_{i}})$, for $i\in \N$:
\begin{enumerate}
\item[($i=0$):] $Y^\bullet_{0}=K(S^0(A),0)=E^{[0,1]}_A$;
\item[($i=1$):] $Y^\bullet_{1}=Y^\bullet_0$;
\item[($i=2$):] $Y^\bullet_{2}=E^{[0,1]}_A\oplus S^{1}(E^1)$ (by a suitable application of Lemma~\ref{reordering_differentials});
\item[($i=3$):] $Y^\bullet_{3}=E^{[0,2]}_A\oplus D^{1}(E^1)$ (use Lemma~\ref{killing_sums_and_products});
\item[($i=4,5$):] $Y^\bullet_5=Y^\bullet_{4}=Y^\bullet_{3}$;
\item[($i=6$):]  $Y^\bullet_{6}=E^{[0,2]}_A\oplus S^{1}(E^1)\oplus D^{1}(E^1)$;
\item[($i=7$):] $Y^\bullet_7=E^{[0,2]}_A\oplus  D^{1}((E^1)^{2})\oplus S^{2}(E^1\oplus E^2)$;
\item[($i=8$):] $Y^\bullet_8=E^{[0,3]}_A\oplus D^{1}((E^1)^2)\oplus D^{2}(E^1\oplus E^2)$;
\item[($i=9,10,11$):] $Y^\bullet_{11}=Y^\bullet_{10}=Y^\bullet_{9}=Y^\bullet_8$;
\item[($i=12$):] $Y^\bullet_{12}=E^{[0,3]}_A\oplus S^{1}(E^1) \oplus D^{1}((E^1)^2)\oplus D^{2}(E^1\oplus E^{2})$;
\item[$\dots\quad$] 
\end{enumerate}
After computing these initial steps, it is already evident that the limit $Y^\bullet:=\varprojlim(\dots \to Y^\bullet_n\to \dots\to Y^\bullet_0)$ is isomorphic to the following complex:
\[
\xymatrix{
Y^\bullet\cong E^{[0,\infty)}_A\times \prod_{i=1}^{\infty}\left(\prod_{j=1}^{i}(D^{i}(E^j))^\N\right).
}
\]
In other words, $Y^\bullet$ is a product of our chosen resolution of $A$ (including $A$, so this part is an exact complex) with a certain amount of disk complexes, so $Y^\bullet$ is an exact complex. In particular, this shows that the Ding-Yang construction works just fine when applied to single stalk complexes. 

\smallskip
Suppose now that, instead of a stalk complex alone, we take  $X^\bullet:=\bigoplus_{i\in\Z}S^i(A)=\prod_{i\in\Z}S^i(A)$ for a given object $A\in \A$, that is:
\[
X^\bullet:=\ (\xymatrix{\cdots\ar[r]^-0&A\ar[r]^-0&A\ar[r]^-0&\cdots\ar[r]^-0&A\ar[r]^-0&\cdots)}
\]
By Lemma~\ref{killing_sums_and_products}, if we are careful with the choice of embeddings into injective objects, the complex that we obtain when we go through the Ding-Yang construction for $X^\bullet$ is the following one:
\begin{align*}
\xymatrix{\prod_{i\in \Z}\Sigma^i(Y^\bullet)}&=\xymatrix{\prod_{i\in \Z}\Sigma^i\left(E^{[0,\infty)}_A\times D^{0}\left(\prod_{j=0}^{\infty}(E^j)^\N\right)\right)}\\
&=\xymatrix{\prod_{i\in \Z}\Sigma^iE^{[0,\infty)}_A\times \prod_{i\in \Z}D^{i}\left(\prod_{j=0}^{\infty}(E^j)^\N\right).}
\end{align*}
In particular, $\prod_{i\in \Z}\Sigma^i(Y^\bullet)$ is quasi-isomorphic to $\prod_{i\in \Z}\Sigma^i(E^{[0,\infty)}_A)$. Observe also that, for each $i\in \Z$, there is a degree-wise split-exact sequence $0\to \Sigma^i(E^\bullet)\to \Sigma^i(E^{[0,\infty)}_A)\to S^{-i}(A)\to 0$ and, taking the product over all $i\in \Z$, we get the following degree-wise split-exact sequence:
\begin{equation}\label{ses_final_DY}
\xymatrix{0\longrightarrow \prod_{i\in \Z}\Sigma^i(E^\bullet)\longrightarrow\prod_{i\in \Z}\Sigma^i(E^{[0,\infty)}_A)\longrightarrow X^\bullet\longrightarrow 0.}
\end{equation}
Hence, whenever the complex $\prod_{i\in \Z}\Sigma^i(E^{[0,\infty)}_A)$ (or, equivalently, $\prod_{i\in \Z}\Sigma^i(Y^\bullet)$) is exact, then one deduces from \eqref{ses_final_DY} that $H^{n}(\prod_{i\in \Z}\Sigma^{i}(E^\bullet))=H^{n-1}(X^\bullet)=A$, for all $n\in\Z$.

Finally, take $\A=\G$, $A:=\Q(R)\in \G$, and $X^\bullet:= \mathrlap{\text{\tiny \hspace{4.3pt}$\G$}}\prod_{i\in \Z}S^i(\Q(R))\in \Ch(\G)$ as in Section~\ref{poison_sec}. Then, 
$H_\G^{n}(\mathrlap{\text{\tiny \hspace{4.3pt}$\G$}}\prod_{i\in \Z}\Sigma^{i}(E^\bullet))\cong \Q(H_R^{n}(\mathrlap{\text{\tiny \hspace{4.3pt}$\G$}}\prod_{i\in \Z}\Sigma^{i}(E^\bullet)))\cong \Q(R)\times \Q(\mathrlap{\text{\tiny \hspace{3.7pt}$R$}}\prod_{i\geq 1}E^i)$, where  the last isomorphism follows from the isomorphism $H_R^{n}(\mathrlap{\text{\tiny \hspace{4.3pt}$\G$}}\prod_{i\in \Z}\Sigma^{i}(E^\bullet))\cong R\times \mathrlap{\text{\tiny \hspace{3.7pt}$R$}}\prod_{i\geq 1}E^i$, which is an important step in the proof of \cite[Theorem~8.4]{zbMATH06915995}, where it is also shown that $\mathrlap{\text{\tiny \hspace{3.7pt}$R$}}\prod_{i\geq 1}E^i\notin \T$. In particular, $\Q(\mathrlap{\text{\tiny \hspace{3.7pt}$R$}}\prod_{i\geq 1}E^i)\neq 0$, and so we can conclude that: 
\[
\xymatrix{
H_\G^{n}(\mathrlap{\text{\tiny \hspace{4.3pt}$\G$}}\prod_{i\in \Z}\Sigma^{i}(E^\bullet)) \cong \Q(R)\times \Q(\mathrlap{\text{\tiny \hspace{3.7pt}$R$}}\prod_{i\geq 1}E^i)\ncong \Q(R)=H_\G^n(X^\bullet).}
\]
By the previous discussion, this shows that the complex $\prod_{i\in \Z}\Sigma^i(Y^\bullet)$, obtained as a result of the Ding-Yang construction applied to $X^\bullet$, can not be exact in this particular case.

\subsection{The problem in the proof of \cite[Lemma~2.1]{yang2015question}}\label{subs_problem_in_lemma}

As shown in the previous subsection, the proof of \cite[Lemma~2.1]{yang2015question} fails in general. For this reason, it may be interesting to identify the concrete problem in Ding and Yang's argument and, if possible, to find additional hypotheses under which the proposed construction can be made to work. The unique problem we could identify in \cite{yang2015question} is the following: on page 3210 in [Op.Cit.], in the last part of the proof of Lemma~2.1, there is an inverse system (in the notation of [Op.Cit.]):  
\[
\cdots \overset{\mu^{i+1}}{\longrightarrow} Y^i\overset{\mu^{i}}{\longrightarrow}\cdots\overset{\mu^{2}}{\longrightarrow} Y^1\overset{\mu^{1}}{\longrightarrow} Y^0\qquad \subseteq \Ch(\G)
\] 
and an $l\in \Z$ such that $H_l(Y^i)=0$ for all $i\in \N$. The authors want to prove that the following map
\[
\xymatrix{
1-\nu\colon \prod_{i=0}^\infty Y^i\longrightarrow \prod_{i=0}^\infty Y^i}
\]
induces an isomorphism in homology at degree $l$. They call $\pi^j\colon \prod_{i=0}^\infty Y^i\to Y^j$ the canonical projection, for each $j\geq 0$, they consider $\pi^j\circ (1-\nu)=\pi^j-\mu^{j+1}\circ \pi^{j+1}$ and correctly verify that 
\begin{equation}\label{ref_correct_part_DY}
\mu_l^{j+1}\pi_l^{j+1}\left(Z_l^{\prod_{i=0}^\infty Y^i}\right)\subseteq B_l^{Y_j},\qquad \text{for all $j\in \Z$.} 
\end{equation}
Unfortunately, \eqref{ref_correct_part_DY} does not imply that 
$
Z_l^{\prod_{i=0}^\infty Y^i}\subseteq B_l^{\prod_{i=0}^\infty Y^i}
$
but just the  weaker inclusion:
$
Z_l^{\prod_{i=0}^\infty Y^i}\subseteq \prod_{i=0}^\infty B_l^{ Y^i}.
$ 
In fact, it may happen that $\prod_{i=0}^\infty B_l^{ Y^i}\neq B_l^{\prod_{i=0}^\infty Y^i}$ (e.g., in the setting of Section~\ref{poison_sec}): the property ``the boundaries of the product coincide with the product of  boundaries'' (if required for all possible products) is in fact equivalent to the (Ab.$4^*$) condition on $\G$. In particular, adding the (Ab.$4^*$) condition to the hypotheses, the proof of \cite[Lemma~2.1(1)]{yang2015question} works perfectly. We will also obtain the conclusion of  \cite[Lemma~2.1(1)]{yang2015question} under a  different set of hypotheses in Section~\ref{other_hyp_fixing_lemma21}.

\section{Model structures for relative homological algebra}\label{rel_homo_alg_sec}\label{sec_seven}
In this  section we  combine some of the main results of \cite{zbMATH06915995} about model approximations for relative homological algebra, and a general criterion for the existence of suitable model categories from \cite{CH} to verify that, if $\A$ is a bicomplete Abelian category, and  $\I$ an injective class of objects such that $\A$ is (Ab.$4^*$)-$\I$-$k$ (for some $k\in \N$), then there is an induced $\I$-injective model structure on $\Ch(\A)$; in particular, the $\I$-derived category $\D(\A;\I)$ is locally small.

\subsection{Injective classes}
Given a class $\I\subseteq \A$, a morphism $\phi\colon X\to Y$ in $\A$ is said to be an {\bf $\I$-monomorphism} if 
\[
\hom_\A(\phi,I)\colon \hom_\A(Y,I)\longrightarrow\hom_\A(X,I)
\]
is surjective (in $\Ab$), for all $I\in \I$. With this concept at hand, we can now introduce the following:
\begin{defn}
A class of objects $\I$ in an Abelian category $\A$ is said to be an {\bf injective class} if:
\begin{enumerate}
\item[\rm ({\sc ic}.1)] $\I$ is closed under products and summands;
\item[\rm ({\sc ic}.2)] for each $A\in \A$ there is an $\I$-monomorphism $\phi\colon A\to I$, with $I\in \I$.
\end{enumerate}
\end{defn}

Observe that a class $\I$ that satisfies the property ({\sc ic}.2) is usually called {\bf pre-enveloping}. In the following lemma we show that the above definition is slightly redundant, in fact, there is no need to ask that $\I$ is closed under products in ({\sc ic}.1):

\begin{lem}\label{pre-enveloping+summands->products}
Let $\A$ be a complete Abelian category. Then, a pre-enveloping class $\I\subseteq \A$ which is closed under direct summands is also closed under products.
\end{lem}
\begin{proof}
Consider a set $\{X_\lambda\}_{\lambda\in\Lambda}$ of objects in $\I$, let $X:=\prod_\Lambda X_\lambda$ and for each $\lambda\in \Lambda$ denote by $\pi_\lambda\colon X\to X_\lambda$ the canonical projection. By hypothesis, there is an $\I$-monomorphism $\varphi\colon X\to I$, for some $I\in \I$. Since $X_\lambda\in \I$ for each $\lambda\in \Lambda$, the following map is surjective:
\[
\hom_\A(\varphi,X_\lambda)\colon \hom_\A(I,X_\lambda)\longrightarrow\hom_\A(X,X_\lambda),
\]
 so there is $\psi_\lambda\in \hom_\A(I,X_\lambda)$ such that $\psi_\lambda\circ \varphi=\pi_\lambda$. By the universal property of the product, there is a unique  $\psi\colon I\to X$ such that $\pi_\lambda\circ\psi=\psi_\lambda$, for all $\lambda\in \Lambda$. Moreover, for each $\lambda\in\Lambda$:
\[
\pi_\lambda\circ \psi\circ \varphi = \psi_\lambda \circ \varphi = \pi_\lambda=\pi_\lambda\circ \id_X.
\]
By the uniqueness in the universal property of the product, $\psi\circ \varphi =\id_X$ and, therefore, $X$ is a summand of $I\in \I$. Hence, $X\in \I$, as desired. 
\end{proof}

If $\A$ is an Abelian category with enough injectives (e.g., if $\A$ is Grothendieck) then  $\I:=\mathrm{Inj}(\A)$, the class of all the injective objects in $\A$, is an injective class for which the $\I$-monomorphisms are the usual monomorphisms. We refer to Section~\ref{ex_and_app_sec} for other concrete examples. 

\subsection{Relative $\I$-injective resolutions of objects}
Let $\A$ be an Abelian category. Given $X^\bullet\in\Ch(\mathcal{A})$ and  $A\in\mathcal{A}$, we define the cochain complex $\hom(X^\bullet ,A)\in\Ch(\Ab)$ as follows:
\begin{itemize}
\item for each $n\in\Z$, let $\hom(X^\bullet ,A)^n:=\hom_\mathcal{A}(X^{-n},A)$;
\item for each $n\in\Z$, the $(n-1)$-th differential of $\hom(X^\bullet ,A)$ is the following map: 
\[
(d^{-n})^*\colon \hom_\mathcal{A}(X^{-n+1},A)\longrightarrow\hom_\mathcal{A}(X^{-n},A),
\] 
such that $(d^{-n})^*(f):=f\circ d^{-n}$, for all $f\in \hom_\mathcal{A}(X^{-n+1},A)$. 
\end{itemize}
Moreover, any morphism $\phi^\bullet \colon X^\bullet\to Y^\bullet$ in $\Ch(\mathcal{A})$, induces the following morphism in $\Ch(\Ab)$: 
\[
\hom(\phi^\bullet ,A)\colon \hom(Y^\bullet ,A)\longrightarrow\hom(X^\bullet ,A)\quad\text{such that}\quad (\hom(\phi^\bullet ,A))(g):=g\circ \phi^{\bullet},
\] 
for all $g\in\hom(Y^\bullet ,A)$. In particular, this gives a functor $\hom(-,A)\colon (\Ch(\A))^{\mathrm{op}}\to \Ch(\Ab)$.

\begin{defn} \label{def.relative-I-injective-resolution}
Let $\mathcal{A}$ be a complete Abelian category, let $\mathcal{I}$ be an injective class, and let $A\in\mathcal{A}$. A {\bf relative $\I$-injective resolution} of $A$ is a pair $(I^\bullet,u\colon S^0(A)\to I^\bullet)$ such that:
\begin{itemize}
\item $I^\bullet\in \Ch^{\geq0}(\I)\subseteq \Ch^{\geq0}(\A)$;
\item $\hom(u,I)\colon\hom(I^\bullet,I)\rightarrow\hom(S^0(A),I)$ is a quasi-isomorphism, for all $I\in \I$.
\end{itemize}
\end{defn}

Let us give some equivalent reformulations of the above definition:

\begin{lem} \label{lem_relative_injective_resolution}
Let $\mathcal{A}$ be a complete Abelian category, $\mathcal{I}\subseteq \A$ an injective class, and take a  complex:
\[
\tilde{I}^\bullet\colon\quad\xymatrix@C=15pt{
\cdots\ar[r]& 0\ar[r]& A\ar[r]^-{d^{-1}}& I^0\ar[r]^-{d^{0}}&I^1\ar[r]^-{d^{1}}&\cdots\ar[r]&I^{n-1}\ar[r]^-{d^{n-1}}&I^n\ar[r]^-{d^{n}}&\cdots\ \ \in \Ch^{\geq -1}(\A),
}
\] 
with $I^j\in\mathcal{I}$, for all $j\geq 0$. Denote by $I^\bullet:=(\xymatrix@C=10pt{
\cdots\ar[r]& 0\ar[r]& I^0\ar[r]^-{d^{0}}&\cdots\ar[r]&I^{n-1}\ar[r]^-{d^{n-1}}&I^n\ar[r]^-{d^{n}}&\cdots})$ the  naive truncation  above $0$, and by $d^{-1}\colon S^{0}(A)\to I^{\bullet}$ the obvious map. Then, the following are equivalent:
\begin{enumerate}
\item $(I^\bullet,d^{-1}\colon S^{0}(A)\to I^{\bullet})$ is a relative $\I$-injective resolution of $A$;
\item  $\hom(\tilde{I}^\bullet, I)$ is  exact, for all $I\in\mathcal{I}$;
\item  $\tilde I^k/B^k(\tilde{I}^\bullet)\rightarrow I^{k+1}$  is an $\mathcal{I}$-monomorphism, for all $k\geq -1$.
\end{enumerate}
\end{lem}
\begin{proof} 
The equivalence ``(1)$\Leftrightarrow$(2)'' is trivial; let us verify that (2) is also equivalent to (3). Indeed, for each $I\in\mathcal{I}$ and $k\geq -1$, consider the following sequence:
\begin{equation}\label{lem_relative_injective_resolution_eq}
\xymatrix{
\hom_\A(\tilde I^{k+1},I)\ar[r]^-{(d^k)^*}&\hom_\A(\tilde  I^{k},I)\ar[r]^-{(d^{k-1})^*}&\hom_\A(\tilde I^{k-1},I),
}
\end{equation} 
Observe that, being $\hom_\A(-,I)\colon \A^{\op}\to \Ab$ a left-exact functor, there is a canonical isomorphism $\hom_\A(\tilde I^k/B^k(\tilde{I}^\bullet),I)=\hom_\A(\coker(d^{k-1}),I)\cong \ker((d^{k-1})^*)$. In particular, the sequence \eqref{lem_relative_injective_resolution_eq} is exact precisely when the induced map $\tilde I^k/B^k(\tilde{I}^\bullet)\rightarrow I^{k+1}$ is an $\mathcal{I}$-monomorphism.
\end{proof}

Given an Abelian category $\A$ and an injective class $\I\subseteq \A$, since $\A$ has enough $\I$-injectives, it is easy to construct a relative $\mathcal{I}$-injective resolution for any object $A\in\A$ by induction, just using again and again the equivalence ``(1)$\Leftrightarrow$(3)'' in the above lemma.

In this relative context for resolutions, the classical statement about extension and unicity up to homotopy of maps between resolutions (see~\cite[Chap. III, Par. 6 and 7]{MR1344215}) still holds true

\begin{lem} \label{lem:htopy-unique-mapsfromI-injtoIacyclic}
Let $\A$ be an Abelian category, let $\I$ be an injective class in $\A$, let $(C^\bullet,\partial^\bullet)$ and $(D^\bullet,d^\bullet)\in\Ch(\A)$ be two cochain complexes, and let $0\leq r$ be a natural number such that, 
\begin{enumerate}
\item $D^i\in \I$, for all $i>r$;
\item $H^{-i}(\hom(C^\bullet,I))=0$, for all $I\in \I$ and all $i\geq r$.
\end{enumerate}
Then, any family of morphisms $(f^k\colon C^k\to D^k)_{k\leq r}$ such that $f^{k}\circ \delta^{k-1}=d^{k-1}\circ f^{k-1}$, for all $k\leq r$, extends to a morphism of complexes $f^\bullet\colon C^\bullet\to D^\bullet$. Moreover, any two such extensions are homotopic via a homotopy $(h^n\colon C^n\to D^{n-1})_{n\in \Z}$ such that $h^k=0$, for all $k\leq r$.
\end{lem}
\begin{proof}
Our hypothesis (2) means that the following sequence is exact in $\Ab$
\begin{equation}\label{mapsfromI-injtoIacyclic_eq}
\xymatrix{
\hom_\A(C^{j+1},I)\ar[r]^-{(\delta^{j})^*}&\hom_\A(C^{j},I)\ar[r]^-{(\delta^{j-1})^*}&\hom_\A(C^{j-1},I),
}
\end{equation}
for all $I\in \I$ and all $j\geq r$. Then, by (1), the sequence \eqref{mapsfromI-injtoIacyclic_eq} is exact provided $I=D^{j+1}$ and, for $j>r$, also for $I=D^j$.

We can now proceed to construct $f^\bullet\colon C^\bullet\to D^\bullet$ by induction. Indeed, suppose that we have morphisms $(f^n\colon C^n  \rightarrow D^n)_{n\leq j}$, for some $j \geq r$, such that $f^{n}\circ \delta^{n-1}=d^{n-1}\circ f^{n-1}$, for all $n\leq j$. In particular, taking $I=D^{j+1}$ in \eqref{mapsfromI-injtoIacyclic_eq}, we deduce that  
\[
(\delta^{j-1})^*(d^j\circ f^j)=d^j\circ f^j\circ \delta^{j-1}=d^j\circ d^{j-1}\circ f^{j-1}=0, 
\] 
that is, $d^j\circ f^j\in \ker((\delta^{j-1})^*)=\Im((\delta^j)^*)$ and, therefore, there is some $f^{j+1}\in \hom_\A(C^{j+1},D^{j+1})$ such that $(\delta^j)^*(f^{j+1})=f^{j+1}\circ \delta^j=d^j\circ f^j$. The family $(f^n\colon C^n  \rightarrow D^n)_{n\leq j+1}$ now satisfies $f^{n}\circ \delta^{n-1}=d^{n-1}\circ f^{n-1}$, for all $n\leq j+1$, and the induction can continue.

Assume now that $f^\bullet,\, g^\bullet\in \hom_{\Ch(\A)}(C^\bullet, D^\bullet)$ are both extensions of the family $(f^k)_{k\leq r}$. We proceed by induction to construct $(h^n\colon C^{n+1}\to D^{n})_{n\in \Z}$ such that $f^{n}-g^n=d^{n-1}\circ h_{n-1}+h^{n}\circ \delta^n$, for all $n\in \Z$. Indeed, let $h^i=0$, for all $i\leq r$ and suppose that, for some $j>r$, we have constructed a family $(h^n\colon C^{n+1}\to D^{n})_{n\leq j}$ such that $f^{n}-g^n=d^{n-1}\circ h^{n-1}+h^{n}\circ \delta^n$, for all $n\leq j$. In particular, using that $(f^\bullet-g^\bullet)$ is a map of complexes, we deduce that
\[
(f^{j+1}-g^{j+1})\circ \delta^j=d^j\circ (f^j-g^j)=d^j\circ d^{j-1}\circ h^{j-1}+d^j\circ h^{j}\circ \delta^j=d^j\circ h^{j}\circ \delta^j.
\]
Hence, taking $I=D^j$  in \eqref{mapsfromI-injtoIacyclic_eq}, and using the above computation, we get:
\[
(\delta^{j})^*(f^{j+1}-g^{j+1}-d^j\circ h^{j})=(\delta^{j})^*(f^{j+1}-g^{j+1})-d^j\circ h^{j}\circ \delta^j=0,
\] 
that is, $f^{j+1}-g^{j+1}-d^j\circ h^{j}\in \ker((\delta^{j})^*)=\Im((\delta^{j+1})^*)$ and so there is some $h^{j+1}\in \hom_\A(C^{j},D^{j+1})$ such that $(\delta^{j+1})^*(h^{j+1})=f^{j+1}-g^{j+1}-d^j\circ h^{j}$, that is, $f^{n}-g^n=d^{n-1}\circ h^{n-1}+h^{n}\circ \delta^n$ holds for all $n\leq j+1$, and the induction can continue.
\end{proof}

As a consequence, we obtain that relative $\I$-injective resolutions are unique up to homotopy:

\begin{cor} \label{cor.relative-I-injresol-homotopy equivalent}
Let $\A$ be an Abelian category and let $\I\subseteq \A$ be an injective class. If $u\colon S^0(A)\rightarrow I^\bullet$ and $v\colon S^0(A)\rightarrow J^\bullet$  are both relative $\mathcal{I}$-injective resolutions of an object $A\in \A$, then there is a homotopy equivalence $f^\bullet\colon I^\bullet\rightarrow J^\bullet$ such that $f^\bullet\circ u=v$, and $f^\bullet$ is unique up to homotopy. 
\end{cor}
\begin{proof}
Use Lemma~\ref{lem:htopy-unique-mapsfromI-injtoIacyclic} to extend $\id_{A}\colon A\to A$ to two morphisms $f^\bullet\colon I^\bullet\to J^\bullet$ and $g^\bullet\colon J^\bullet\to I^\bullet$, such that $f^\bullet\circ u=v$ and $g^\bullet\circ v=u$. The uniqueness of liftings up to homotopy forces $g^\bullet\circ f^\bullet$ and $f^\bullet\circ g^\bullet$ to be homotopic to $\id_{I^\bullet}$ and  $\id_{J^\bullet}$, respectively.
\end{proof}

\subsection{Generalities about the $\I$-injective model structure}

Given an injective class $\I$ in an Abelian category $\A$, one can introduce the following classes of morphisms in $\Ch(\A)$ that, under suitable hypotheses, make $\Ch(\A)$ into a model category:

\begin{defn}\label{I_mod_st_deff}
Let $\A$ be an Abelian category and $\I\subseteq \A$ an injective class. We say that a morphism $\phi^\bullet\colon X^\bullet\to Y^\bullet$ is
\begin{itemize}
\item an {\bf $\I$-cofibration} if $\phi^n$ is an $\I$-monomorphism, for all $n\in\Z$;
\item  an {\bf $\I$-weak equivalence}  if $\hom(\phi^\bullet,I)$ is a quasi-isomorphism in $\Ch(\Ab)$, for all $I\in \I$;
\item an {\bf $\I$-fibration} if it is right weakly orthogonal to the trivial $\I$-cofibrations. 
\end{itemize}
We denote these classes of maps by $\C_\I$, $\W_\I$, and $\F_\I$, respectively. Furthermore, we say that a complex $X^\bullet\in \Ch(\A)$ is $\I${\bf -acyclic} provided $\hom(X^\bullet,I)$ is acyclic in $\Ch(\Ab)$, for all $I\in \I$.
\end{defn}

\begin{rmk}
We take advantage of this paper to correct an annoying misprint in the statement of \cite[Theorem~2.3]{zbMATH06915995}: the $\I$-cofibrations in the relative model structure on $\Ch_{\leq n}$ should be $\I$-monomorphisms in degrees $i <n$, and not $i \leq n$.
\end{rmk}

In the following lemma, which is a consequence of \cite[Proposition~2.5 and Lemma~2.7(b)]{CH}, we collect some properties of $\I$-fibrations and of $\I$-fibrant objects:

\begin{lem}\label{rel_fibrations_lem}
Let $\A$ be a bicomplete Abelian category and let $\I$ be an injective class in $\A$. Then:
\begin{enumerate}
\item the $\I$-fibrations are precisely the degree-wise split epimorphisms with $\I$-fibrant kernel;
\item a bounded below complex $X^\bullet\in \Ch^+(\A)$ is $\I$-fibrant, provided $X^i\in \I$, for all $i\in\Z$.
\end{enumerate}
\end{lem}

More generally, let us recall the following important result by Christensen and Hovey (they actually state the dual result for projective classes, but it is an easy exercise to show that their statement is equivalent to the following one):

\begin{thm}[{\rm \cite[Theorem~2.2]{CH}}]\label{CH_main_thm}
Let $\A$ be a bicomplete Abelian category and $\I\subseteq \A$ an  injective class such that, for each $X^\bullet\in\Ch(\A)$, there is an $\I$-fibrant replacement $\lambda^\bullet\colon X^\bullet\to F^\bullet$ (i.e., $\lambda^\bullet\in \W_\I$ and $F^\bullet$ is $\I$-fibrant). Then, $(\Ch(\A), \W_\I,\C_\I,\F_\I)$ is a model category. 
\end{thm}

In the following proposition (adapted from \cite{zbMATH06915995}) we establish the existence of ``Spaltenstein towers of partial $\I$-fibrant replacements'', in complete analogy to our discussion in Section~\ref{DG_inj_sec} for the case $\I=\Inj(\A)$. This construction produces a standard candidate to $\I$-fibrant replacement for any given $X^\bullet\in \Ch(\A)$.  In Subsection~\ref{subs_I_inj_res}, we will discuss a sufficient condition to ensure that the construction below always produces an $\I$-fibrant replacement. 

\begin{prop}\label{prop_CNPS}
Let $\A$ be a bicomplete Abelian category and let $\I$ be an injective class in $\A$. Then, for each $X^\bullet\in \Ch(\A)$, there is an inverse system (a {\bf tower}) of complexes
\begin{equation}\label{fib_tower_eq}
\xymatrix{
\cdots\ar[r]^-{t_2^\bullet}&E_2^\bullet\ar[r]^-{t_1^\bullet}& E_1^\bullet \ar[r]^-{t_0^\bullet}& E_0^\bullet
}
\end{equation}
that satisfies the following conditions, for all $n\in\N$:
\begin{enumerate}
\item $E_n^\bullet\in \Ch^{\geq-n}(\A)$ and $E_n^i\in \I$, for all $i\in\Z$ (so $E_n^\bullet$ is $\I$-fibrant);
\item the morphism $t_n^\bullet$ is an $\I$-fibration;
\item there is an $\I$-weak equivalence $\lambda_n^\bullet \colon \tau^{\geq -n}(X^\bullet)\to E_n^\bullet$ such that 
\[
\lambda_n^\bullet \circ \pi_n^\bullet=t_n^\bullet\circ\lambda_{n+1}^\bullet, \quad\text{for all $n\in\N$,}
\] 
where $\pi_n^\bullet\colon \tau^{\geq -n-1}(X^\bullet)\to \tau^{\geq -n}(X^\bullet)$ is the canonical projection.
\end{enumerate}
\end{prop}
\begin{proof}
By the results in Sections~4-5 in \cite{zbMATH06915995}, there is a model category $\mathrm{Tow}(\A,\I)$ of towers of complexes. Given $X^\bullet\in \Ch(\A)$, one can consider $\mathrm{tow}(X^\bullet)\in \mathrm{Tow}(\A,\I)$, the tower of successive truncations of $X^\bullet$. A sequence \eqref{fib_tower_eq}, with the properties (1--3) in the statement, is precisely a fibrant replacement of  $\mathrm{tow}(X^\bullet)$ in the model category $\mathrm{Tow}(\A,\I)$.
\end{proof}

\subsection{The (Ab.$4^*$)-$k$ condition and the injective model structure relative to $\I$}\label{subs_I_inj_res}
Let $\A$ be a bicomplete Abelian category, $X^\bullet\in \Ch(\A)$, $\I\subseteq \A$ an injective class, and define $E^\bullet:=\varprojlim_\N E_n^\bullet$ as the inverse limit of the tower \eqref{fib_tower_eq} described by Proposition~\ref{prop_CNPS}. 
Then, $E^\bullet:=\varprojlim_\N E_n^\bullet$ is  $\I$-fibrant (by Remark~\ref{rem_lim_of_fib}) and, by condition (3) in the proposition, there is a canonical morphism $\lambda^\bullet\colon X^\bullet\to E^\bullet$. To determine whether  $\lambda^\bullet$ is an $\I$-weak equivalence (i.e., if $E^\bullet$ is an $\I$-fibrant replacement for $X^\bullet$), it is useful to introduce a relative version of Roos' (Ab.$4^*$)-$k$ condition (see \cite[Definition~6.1]{zbMATH06915995}). We do it after the following technical lemma:

\begin{lem} \label{lem.AB4-I-k}
Let $\mathcal{A}$ be a complete Abelian category, let $\I\subseteq \A$ be an injective class, and consider a family  $(u_\lambda\colon A_\lambda\rightarrow I_\lambda^\bullet)_{\Lambda}$   of relative $\mathcal{I}$-injective resolutions in $\mathcal{A}$. The following  are equivalent for any $k\geq 0$:
\begin{enumerate}
\item $H^{-n}(\hom(\prod_{\Lambda}I_\lambda^\bullet,I))=0$, for all  $I\in\mathcal{I}$ and all $n>k$;
\item the induced map $\coker(\prod d_\lambda^{n-1})\rightarrow\prod_{\Lambda}I_\lambda^{n+1}$ is an $\mathcal{I}$-monomorphsm, for all $n> k$.
\end{enumerate}
\end{lem}
\begin{proof}
We set by convention $I_\lambda^{-1}:=A_\lambda$, for all $\lambda\in\Lambda$. Condition (1) means that the following sequences are exact in $\Ab$ for all  $I\in\mathcal{I}$ and all $n>k$:  
\begin{equation}\label{def_ab4*_lemma}
\xymatrix{\hom_\mathcal{A}\left(\prod_\Lambda I_\lambda^{n+1},I\right)\stackrel{(\prod d_\lambda^{n})^*}{\longrightarrow}\hom_\mathcal{A}\left(\prod_\Lambda I_\lambda^{n},I\right)\stackrel{(\prod d_\lambda^{n-1})^*}{\longrightarrow}\hom_\mathcal{A}\left(\prod_\Lambda I_\lambda^{n-1},I\right)},
\end{equation}
that is, we want the morphism $(\prod d_\lambda^{n})^*\colon  \hom_\mathcal{A}\left(\prod_\Lambda I_\lambda^{n+1},I\right)\to \Ker((\prod d_\lambda^{n-1})^*)$ to be surjective. As $\hom_\A(-,I)$ is left exact, there is an isomorphism  $\Ker((\prod d_\lambda^{n-1})^*)\cong\hom_\mathcal{A}(\coker(\prod d_\lambda^{n-1}),I)$. In particular, this shows that the sequences in \eqref{def_ab4*_lemma} are exact for all $I\in\I$ if, and only if, $\hom_\mathcal{A}\left(\prod_\Lambda I_\lambda^{n+1},I\right)\to \hom_\mathcal{A}(\coker(\prod d_\lambda^{n-1}),I)$ is surjective for all $I\in \mathcal{I}$, that is, if the morphism $\coker(\prod d_\lambda^{n-1})\to \prod_\Lambda I_\lambda^{n+1}$ is an $\I$-monomorphism, which is condition (2).
\end{proof}

\begin{defn} \label{def.AB4-I-k}
Let $\mathcal{A}$ be a complete Abelian category, let $\mathcal{I}\subseteq \A$ be an injective class, and let $k$ be a natural number. Then, $\mathcal{A}$ is said to be  {\bf (Ab.4$^*$)-$\mathcal{I}$-$k$} when, for each family $(A_\lambda )_{\Lambda}$ of objects of $\mathcal{A}$, and for some (or, equivalently, any) choice of relative $\mathcal{I}$-injective resolutions $(I_\lambda^\bullet,u_\lambda\colon S^0(A_\lambda)\rightarrow I_\lambda^\bullet)_{ \Lambda}$, the equivalent conditions (1) and (2) in Lemma~\ref{lem.AB4-I-k} hold true.
\end{defn}

Observe that varying the chosen relative $\I$-injective resolutions does not affect the above definition. This is a consequence of Corollary~\ref{cor.relative-I-injresol-homotopy equivalent}, which implies that, if $(I_\lambda^\bullet,u_\lambda)_{\Lambda}$ and $(J_\lambda^\bullet,v_\lambda)_{\Lambda}$ are two families of relative $\mathcal{I}$-injective resolutions corresponding to $(A_\lambda)_{\Lambda}$, then there are homotopy equivalences $(f_\lambda\colon I_\lambda^\bullet\to J_\lambda^\bullet)_{ \Lambda}$, whose product gives a homotopy equivalence $\prod_{\Lambda}f_\lambda\colon \prod_{\Lambda}I^\bullet_\lambda\to \prod_{\Lambda}J^\bullet_\lambda$, which then induces an homotopy equivalence between $\hom(\prod_{\Lambda}I^\bullet_\lambda,I)$ and $\hom(\prod_{\Lambda}J^\bullet_\lambda,I)$, for all $I\in\mathcal{I}$. In particular, these two complexes of Abelian groups have the same cohomology groups. Hence, condition (1) in Lemma~\ref{lem.AB4-I-k} can be checked on either complex with the same result. 

\begin{cor} \label{cor.AB4-I-0}
Let $\mathcal{A}$ be a complete Abelian category, let $\mathcal{I}\subseteq \A$ be an injective class, and consider the following assertions:
\begin{enumerate}
\item each product of $\mathcal{I}$-preenvelopes is an $\mathcal{I}$-preenvelope;
\item each product of $\mathcal{I}$-monomorphisms is an $\mathcal{I}$-monomorphism;
\item $\mathcal{A}$ is (Ab.4$^*$)-$\mathcal{I}:=$(Ab.4$^*$)-$\mathcal{I}$-$0$.
\end{enumerate}
Then, the equivalence ``(1)$\Leftrightarrow$(2)'' holds in general, while ``(1,2)$\Rightarrow$(3)''  holds if $\mathcal{A}$ is (Ab.4$^*$). 
\end{cor}
\begin{proof}
As ``(2)$\Rightarrow$(1)'' is trivial, we just verify that ``(1)$\Rightarrow$(2)'':  let $(\iota_\lambda \colon X_\lambda\rightarrow Y_\lambda)_{\Lambda}$ be a family of $\mathcal{I}$-monomorphisms in $\A$. Choose, for each $\lambda\in\Lambda$, an $\mathcal{I}$-preenvelope $u_\lambda\colon X_\lambda\rightarrow I_\lambda$. As each $\iota_\lambda$ is an $\mathcal{I}$-monomorphism, there are morphisms $v_\lambda\colon Y_\lambda\rightarrow I_\lambda$ such that $v_\lambda\circ\iota_\lambda =u_\lambda$, for all $\lambda\in\Lambda$. By (1), the product map $\prod u_\lambda=(\prod v_\lambda)\circ (\prod \iota_\lambda)$ is an $\mathcal{I}$-preenvelope, and so also an $\mathcal{I}$-monomorphism. As a consequence, also $\prod \iota_\lambda$ is an $\mathcal{I}$-monomorphism, as desired. 

\smallskip
For the implication ``(2)$\Rightarrow$(3)'' recall that, by Lemma~\ref{lem.AB4-I-k}, for $\A$ to be (Ab.4$^*$)-$\mathcal{I}$ it is enough that, for any family $(X_\lambda)_{ \Lambda}$ and a corresponding choice of relative $\I$-injective resolutions $(I_\lambda^\bullet,u_\lambda)_\Lambda$, the induced map $(\prod_\Lambda I_\lambda^n)/B^n(\prod_\Lambda I_\lambda^\bullet)\to \prod_\Lambda I_\lambda^{n+1}$ is an $\I$-monomorphism for all $n\geq-1$. Assuming that $\mathcal{A}$ is (Ab.4$^*$), there is an isomorphism $(\prod_\Lambda I_\lambda^n)/B^n(\prod_\Lambda I_\lambda^\bullet)\cong \prod_\Lambda (I_\lambda^n/ B^n(I_\lambda^\bullet))$, so the condition becomes equivalent to the fact that $\prod_\Lambda (I_\lambda^n/ B^n(I_\lambda^\bullet))\to  \prod_\Lambda I_\lambda^{n+1}$ is an $\I$-monomorphism, for all $n\geq-1$. This is then a consequence of (2), as the map in question is the product of the family of $\I$-monomorphisms $(I_\lambda^n/ B^n(I_\lambda^\bullet)\to I_\lambda^{n+1})_{\Lambda}$. 
\end{proof}

The proof of the following theorem follows the same steps  used in the proof of Theorem~\ref{main_unbounded_dg_thm}:
 
\begin{thm}[{\rm \cite[Theorem~6.4]{zbMATH06915995}}]\label{thm_lim_fib_is_we}
Let $\A$ be a bicomplete Abelian category and $\I\subseteq \A$ an injective class such that $\A$ is (Ab.$4^*$)-$\I$-$k$ for some $k\in \N$. Given $X^\bullet\in\Ch(\A)$ and a tower
\[
\xymatrix{
\cdots\ar[r]^-{t_2^\bullet}&E_2^\bullet\ar[r]^-{t_1^\bullet}& E_1^\bullet \ar[r]^-{t_0^\bullet}& E_0^\bullet
}
\]
satisfying (1--3) in Proposition~\ref{prop_CNPS}, the map $\lambda^\bullet\colon X^\bullet\to E^\bullet:= \varprojlim_\N E_n^\bullet$ is an $\I$-weak equivalence. 
\end{thm}

In particular, in the setting of the above theorem (also using Proposition~\ref{prop_CNPS} and Remark~\ref{rem_lim_of_fib}), any complex $X^\bullet\in \Ch(\A)$ has an $\I$-fibrant replacement $\lambda^\bullet\colon X^\bullet\to E^\bullet$. Combining this result with Theorem~\ref{CH_main_thm}, one immediately deduces that:

\begin{cor}\label{cor_rel_inj_model_structure}
Let $\A$ be a bicomplete Abelian category and let $\I$ be an injective class in $\A$. If $\A$ is (Ab.$4^*$)-$\I$-$k$ for some $k\in \N$, then $(\Ch(\A), \W_\I,\C_\I,\F_\I)$ is a model category.
\end{cor}

Observe that an advantage of Corollary~\ref{cor_rel_inj_model_structure} is that it applies to Abelian categories that are not necessarily Grothendieck, and where the small object argument may not be applicable. 

\section{Examples and applications}\label{ex_and_app_sec}

\subsection{On injective classes that are cogenerating}\label{subs_cogen_inj}

Let $\A$ be an Abelian category. A full subcategory $\mathcal{I}\subseteq \A$ is called {\bf cogenerating} if, for each $A\in\mathcal{A}$, there is a monomorphism $A\to I$, for some $I\in\mathcal{I}$.

\begin{prop}\label{char_cogen_prop}
The following are equivalent for an injective class $\mathcal{I}$ in an Abelian category $\mathcal{A}$:
\begin{enumerate}
\item $\mathcal{I}$ is a cogenerating class in $\A$;
\item $\mathcal{I}$-monomorphisms in $\A$ are, in particular, also monomorphisms;
\item $\mathcal{I}$-acyclic complexes in $\Ch(\A)$ are, in particular, also acyclic;
\item $\mathcal{I}$-weak equivalences  in $\Ch(\A)$ are, in particular, also quasi-isomorphisms.
\end{enumerate}
\end{prop}
\begin{proof}
(1)$\Leftrightarrow$(2). Given  an $\mathcal{I}$-monomorphism $\phi\colon A\to B$, fix a monomorphism $u\colon A\to I\in \I$. The surjectivity of the map $\hom_\A(\phi,I)\colon \hom_\A(B,I)\to \hom_\A(A,I)$ implies that $u=v\circ \phi$, for some $v\in\hom_\mathcal{A}(B,I)$ and, therefore, $\phi$ is a monomorphism. The converse is trivial.

(2)$\Leftrightarrow$(3). Given $X^\bullet=(\cdots\overset{}{\to}X^n\overset{d^n}{\longrightarrow}X^{n+1}\to \cdots)\in \Ch(\A)$, let $\bar d^n\colon X^n/B^n(X^\bullet)\to X^{n+1}$ be the morphism induced by the differential, so that $H^n(X^\bullet)=\Ker(\bar d^n)$, for all $n\in \Z$. If $X^\bullet$ is $\mathcal{I}$-acyclic then $\bar d^n$ is an $\I$-monomorphism (by \cite[Proposition 1.15(7)]{zbMATH06915995}) for all $n\in\Z$, and so, by (2), $H^n(X^\bullet)=\Ker(\bar d^n)=0$, for all $n\in \Z$. Conversely, a map $\psi\colon A\to B$ in $\A$ is a(n) \mbox{($\mathcal{I}$-)}monomorphism if and only if the complex $(\cdots\to 0\rightarrow A\to B\to \coker(\psi)\rightarrow 0\to \cdots)$ in  $\Ch^{\geq-1}(\A)\cap \Ch^{\leq 1}(\A)$ is ($\mathcal{I}$-)acyclic. Hence, also the implication ``(3)$\Rightarrow$(2)'' holds.

\smallskip
(3)$\Leftrightarrow$(4).  It is well-known that a morphism $\phi^\bullet \colon A^\bullet\to B^\bullet$ in $\Ch(\mathcal{A})$ is a quasi-isomorphism if and only if $\cone(\phi^\bullet)$ is  acyclic, while $\phi^\bullet $ is an $\I$-weak equivalence if and only if $\cone(\phi^\bullet)$ is $\I$-acyclic (by \cite[Proposition 1.15(2)]{zbMATH06915995}). These two characterizations imply the desired equivalence. 
\end{proof}

\begin{rmk}
Let $\mathcal{A}$ be an (Ab.4$^*$) Abelian category, and $(X_\lambda^\bullet)_{\Lambda}$ a family in $\Ch(\A)$, then there are isomorphisms $B^n(\prod_\Lambda X^\bullet_\lambda)\cong \prod_\Lambda B^n(X^\bullet_\lambda)$, for all $n\in \Z$. Take  a cogenerating injective class $\I\subseteq \A$, and suppose that $X_\lambda^\bullet$ is $\I$-acyclic, for all $\lambda\in \Lambda$. By Proposition~\ref{char_cogen_prop}, each $X_\lambda^\bullet$ is also acyclic, and so $\coker(d_\lambda^{n-1})=X_\lambda^n/B^n(X_\lambda^\bullet)=X_\lambda^n/Z^n(X^\bullet_\lambda)\cong B^{n+1}(X^\bullet_\lambda)$, for all $n\in \Z$.  Combining these two observations we deduce that $\prod_\Lambda X^\bullet_\lambda$ is $\I$-acyclic if, and only if, the map $\prod_{\Lambda}B^{n}(X^\bullet_\lambda)\to \prod_{\Lambda}X^{n}_\lambda$ is an $\I$-monomorphism, for all $n\in\Z$. 

Using an analogous idea, one proves that an (Ab.4$^*$) Abelian category $\A$ is (Ab.4$^*$)-$\I$-$k$, with $\I$ cogenerating, if and only if the following variation of the condition (2) in Lemma \ref{lem.AB4-I-k} holds for any family of relative $\mathcal{I}$-injective resolutions $(u_\lambda\colon A_\lambda\rightarrow I_\lambda^\bullet)_{\Lambda}$ of objects of $\mathcal{A}$:
\begin{enumerate}
\item[(2')]  $\prod_\Lambda B^{n+1}(I_\lambda^\bullet)\to \prod_\Lambda I_\lambda^{n+1}$ is an $\I$-monomorphism, for all $n> k$;
\end{enumerate}
that is, we want the product functor to preserve the $\mathcal{I}$-monomorphisms $(B^{n+1}(I_\lambda^\bullet)\to I_\lambda^{n+1})_\Lambda$.
\end{rmk}

Let  $\I\subseteq \A$ be a cogenerating injective class. Then, a short  exact sequence in $\A$:
\[
0\longrightarrow X\overset{\iota}\longrightarrow Y\longrightarrow Z\longrightarrow 0
\] 
is called {\bf $\I$-exact} if $\iota$ is an $\I$-monomorphism or, equivalently, if the following sequences are exact: 
\[
0\longrightarrow \hom_\A(Z,I)\longrightarrow \hom_\A(Y,I)\overset{\iota^*}\longrightarrow \hom_\A(X,I)\longrightarrow0, \quad\text{for all $I\in \I$.}
\] 
We have the following variation of Corollary~\ref{cor.AB4-I-0} for cogenerating injective classes:

\begin{lem}\label{criterion_ab4*_I_cogen_lem}
Let $\mathcal{A}$ be a complete Abelian category, and $\I\subseteq \A$  a cogenerating injective class. Suppose that any family of $\I$-exact sequences $(0\to X_\lambda\to Y_\lambda\to Z_\lambda\to 0)_\Lambda$, gives a short exact sequence $0\to \prod_\Lambda X_\lambda\to \prod_\Lambda Y_\lambda\to \prod_\Lambda Z_\lambda\to 0$, which is moreover $\I$-exact. Then, $\A$ is (Ab.4$^*$)-$\I$.
\end{lem}
\begin{proof}
Let $(X_\lambda)_{ \Lambda}\subseteq \A$ and, for each $\lambda\in \Lambda$, consider the following relative $\I$-injective resolution: 
\begin{equation}\label{rel_res_lambda_eq}
\xymatrix{
0\ar[r]&X_\lambda\ar[r]^-{u_\lambda}& E^0_\lambda\ar[r]^{d^0_\lambda}&E^1_\lambda\ar[r]^{d^1_\lambda}&\cdots\ar[r]^{d^{n-1}_\lambda}&E^n_\lambda\ar[r]^{d^n_\lambda}&\cdots
}
\end{equation}
with $E^i_\lambda\in \I$, for all $i\geq 0$. As $\I$ is cogenerating, \eqref{rel_res_lambda_eq} is also an exact sequence. In particular, letting $B^0_\lambda:=X_\lambda$ and $B^n_\lambda:=\Im(d^{n-1}_\lambda)$, for all $n\geq 1$, we have the following short exact sequences, for all $n\geq 0$ and $\lambda\in \Lambda$, that are also $\I$-exact:
\[
0\longrightarrow B^n_\lambda\longrightarrow E_\lambda^n\longrightarrow B_\lambda^{n+1}\longrightarrow 0.
\]
By hypothesis, for each $n\geq0$, the following short exact sequence is $\I$-exact:
\[
\xymatrix{
0\longrightarrow \prod_{\Lambda}B^n_\lambda\longrightarrow \prod_{\Lambda}E_\lambda^n\longrightarrow \prod_{\Lambda}B_\lambda^{n+1}\longrightarrow 0.
}
\]
These conditions tell us that the following is a relative $\I$-injective resolution:
\[
\xymatrix{
0\ar[r]&\prod_{\Lambda}X_\lambda\ar[r]^-{\prod u_\lambda}& \prod_{\Lambda}E^0_\lambda\ar[r]^{\prod d^0_\lambda}&\prod_{\Lambda}E^1_\lambda\ar[r]^{\prod d^1_\lambda}&\cdots\ar[r]^{\prod d^{n-1}_\lambda}&\prod_{\Lambda}E^n_\lambda\ar[r]^{\prod d^n_\lambda}&\cdots,
}
\]
showing that $\A$ is (Ab.4$^*$)-$\I$, as desired.
\end{proof}

Let $\A$ be an Abelian category with an injective class $\I$. As in \cite{zbMATH06915995}, consider the localization of $\Ch(\mathcal{A})$ with respect to $\mathcal{W}_\I$, the class of $\I$-weak equivalences, $\mathcal{D}(\mathcal{A};\mathcal{I}):=\Ch(\mathcal{A})[\mathcal{W}_\I^{-1}]$. In general, this category could fail to be {\bf locally small} (that is, the morphisms between two objects may form a proper class), while it is certainly locally small whenever $(\Ch(\A), \W_\I,\C_\I,\F_\I)$ is a model category (e.g., in the setting of Corollary~\ref{cor_rel_inj_model_structure}) as, in that case, $\mathcal{D}(\mathcal{A};\mathcal{I})$ is equivalent to the corresponding homotopy category. In what follows we denote by $\mathcal{K}(\mathcal{A})$ the homotopy category of $\A$, and by $\mathrm{Ac}(\mathcal{A})\subseteq \mathcal{K}(\mathcal{A})$ (resp., $\mathrm{Ac}_\I(\mathcal{A})\subseteq \mathcal{K}(\mathcal{A})$) its full subcategory of ($\I$-)acyclic complexes.

\begin{cor}\label{coro_verdier_quot}
Let $\A$ be an Abelian category and let $\I\subseteq \A$ be  an injective class. Then, 
\begin{enumerate}
\item if $\mathcal{D}(\mathcal{A};\mathcal{I})$ is locally small, it is a Verdier quotient $\mathcal{D}(\mathcal{A};\mathcal{I})\cong\mathcal{K}(\mathcal{A})/\mathrm{Ac}_\I(\mathcal{A})$;
\item if both $\mathcal{D}(\mathcal{A};\mathcal{I})$ and  $\mathcal{D}(\mathcal{A})$ are locally small and $\I$ is cogenerating, then there is a canonical Verdier quotient functor $\mathcal{D}(\mathcal{A};\mathcal{I})\to\mathcal{D}(\mathcal{A})$.
\end{enumerate}
\end{cor}
\begin{proof}
(1) follows by \cite[Proposition 2.3 and Corollary 2.4]{zbMATH06915995}.

\smallskip
(2) By Proposition~\ref{char_cogen_prop}, there are inclusions  $\text{Ac}_\mathcal{I}(\mathcal{A})\subseteq\text{Ac}(\mathcal{A})\subseteq\mathcal{K}(\mathcal{A})$ of full triangulated subcategories. By \cite[Proposition 2.3.1 in Chap.\,II]{tesis_Verdier}, we know that $\text{Ac}(\mathcal{A})/\text{Ac}_\mathcal{I}(\mathcal{A})$ is a full triangulated subcategory of $\mathcal{K}(\mathcal{A})/\text{Ac}_\mathcal{I}(\mathcal{A})$ and that there is a canonical triangulated functor 
\[
\pi_\I\colon \mathcal{D}(\mathcal{A})\cong\mathcal{K}(\mathcal{A})/\text{Ac}(\mathcal{A})\longrightarrow\frac{\mathcal{K}(\mathcal{A})/\text{Ac}_\mathcal{I}(\mathcal{A})}{\text{Ac}(\mathcal{A})/\text{Ac}_\mathcal{I}(\mathcal{A})}\cong\frac{\mathcal{D}(\mathcal{A};\mathcal{I})}{\text{Ac}(\mathcal{A})/\text{Ac}_\mathcal{I}(\mathcal{A})}.
\]
By Proposition~\ref{char_cogen_prop}, $\mathcal{W}_\I\subseteq \W$, so the canonical functor $p\colon \Ch(\mathcal{A})\to\mathcal{D}(\mathcal{A})$ sends $\I$-weak equivalences to isomorphisms. Therefore, $p$ induces a unique functor $\mathcal{D}(\mathcal{A};\mathcal{I})=\Ch(\mathcal{A})[\mathcal{W}_I^{-1}]\to\mathcal{D}(\mathcal{A})$, which is a triangulated quasi-inverse to $\pi_\I$. 
\end{proof}

We now give an application to the derived category relative to a generator introduced in \cite{Gillespie:2016aa}:

\begin{eg} \label{ex.Gillespie}
Given a complete Abelian category $\A$ with a cogenerator $E$, we can consider the cogenerating injective class $\mathcal{I}=\Prod(E)$ of all the summands of products of copies of $E$. 
In particular, take $\mathcal{A}=\mathcal{G}^{\op}$ for a Grothendieck category $\mathcal{G}$, that is,  $\A$ is the category of strict complete topologically coherent left $R$-modules over some strict complete topologically left coherent and linearly compact ring $R$ (see \cite{OBERST1970473}). Then, $E$ is just a generator of $\mathcal{G}$, and the localization $\mathcal{D}(\mathcal{A};\mathcal{I})\cong \mathcal{D}_E(\mathcal{G})^\op$ is the dual of the $E$-derived category defined in \cite{Gillespie:2016aa}. Since $\mathcal{D}(\mathcal{G})^{op}\cong\mathcal{D}(\mathcal{A})$, Corollary~\ref{coro_verdier_quot} gives a Verdier quotient functor $\mathcal{D}_E(\mathcal{G})\to\mathcal{D}(\mathcal{G})$, for any generator $E$ of $\mathcal{G}$.
\end{eg}

\subsection{Applications to Gillespie's question}\label{other_hyp_fixing_lemma21}
Let $\A$ be a bicomplete Abelian category, and let $(\X,\Y)$ be a complete and hereditary cotorsion pair in $\A$. In this context, Gillespie asked whether $(\dg\X,\tilde \Y)$ and $(\tilde \X,\dg\Y)$ are always complete cotorsion pairs in $\Ch(\A)$ (see Question~\ref{Gillespie_question}). Let us focus our discussion on $(\tilde \X,\dg\Y)$, bearing in mind that dual considerations hold for $(\dg\X,\tilde \Y)$. Now, to show that $(\tilde \X,\dg\Y)$ is complete we need to find, for any given $A^\bullet\in \Ch(\A)$, two short exact sequences:
\begin{enumerate}
    \item $0\to Y_1^\bullet\to X_1^\bullet\to  A^\bullet\to 0$, with $X_1^\bullet\in \tilde \X$ and $Y_1^\bullet\in\dg\Y$;
    \item $0\to  A^\bullet\to Y_2^\bullet\to X_2^\bullet\to 0$, with $X_2^\bullet\in \tilde \X$ and $Y_2^\bullet\in\dg\Y$.
\end{enumerate}
A first simplification is the following: given $A^\bullet\in \Ch(\A)$ take, for each $n\in \Z$, an embedding $\iota^n\colon A^n\hookrightarrow Y^n$, with $Y^n\in \Y$ (they exist as $(\X,\Y)$ is complete), and consider the morphism of complexes $A^\bullet\to D^{n-1}(Y^n)$ whose non-trivial components are $\iota^n\colon A^n\to Y^n$, and $\iota^n\circ d^{n-1}_{A^\bullet}\colon A^{n-1}\to Y^n$. Then, the diagonal map $A^\bullet\to \prod_{\Z}D^{n-1}(Y^n)$ is a monomorphism, and $\prod_{\Z}D^{n-1}(Y^n)\in \dg\Y$. Hence, the class $\dg\Y$ is always cogenerating and so, by Remark~\ref{rem_on_completeness}, if the sequences in (1) do exist, those in (2) do too.

A second important reduction follows from the proof of \cite[Theorem~2.4]{yang2015question}: the sequences in (1) exist, provided there is, for all $A^\bullet\in \Ch(\A)$, a short exact sequence as follows:
\begin{enumerate}
    \item[(0)] $0\to Y_0^\bullet\to E^\bullet\to A^\bullet\to 0$, with $E^\bullet\in \E$ an exact complex, and $Y_0^\bullet\in\dg\Y$.
\end{enumerate}
In fact, to build (1), start with the sequence in (0) and consider the pull-back $Y_1^\bullet$ of the map $Y_0^\bullet\to E^\bullet$ along a special $\tilde \X$-precover $X_1^\bullet\twoheadrightarrow E^\bullet$ (which exists by \cite[Lemma~2.3]{yang2015question}). This gives a short exact sequence $0\to Y_1^\bullet\to X_1^\bullet\to  A^\bullet\to 0$, where $X_1^\bullet\in \tilde \X$ by construction, while $Y_1^\bullet\in \dg\Y$ since it is an extension of $Y_0^\bullet$ by $\Ker(X_1^\bullet\twoheadrightarrow E^\bullet)$, and both are objects in $\dg\Y$.

Now, as we have discussed in Section~\ref{sec_six}, the construction given in \cite[Lemma~2.1(1)]{yang2015question} to produce the sequences in (0) in full generality, fails unless we add the hypothesis that $\A$ is (Ab.$4^*$). On the other hand, whenever $(\X,\Y)$ is a right-complete cotorsion pair in a complete Abelian category  $\A$, the class $\Y$ is cogenerating. Moreover, $\Y$ is closed under summands in $\A$, and it is also pre-enveloping since any right $(\X,\Y)$-approximation sequence is clearly $\Y$-exact, so $\Y$ is an injective class by Lemma~\ref{pre-enveloping+summands->products}. In particular, whenever $\A$ is (Ab.$4^*$)-$\Y$-$k$ for some $k\in \N$, for any given complex $A^\bullet\in \Ch(\A)$ there is a $\Y$-fibrant replacement $\lambda^\bullet\colon A^\bullet\to Y^\bullet$ (see Corollary~\ref{cor_rel_inj_model_structure}); in fact, the explicit construction given in  Section~\ref{sec_seven} shows that $Y^\bullet$ can be built as a limit of a suitable tower of (left bounded) partial resolutions and, therefore, we can assume that $Y^\bullet  \in \mathrm{dg}\Y$. Furthermore, as $\Y$ is cogenerating, $\lambda^\bullet$ is a quasi-isomorphism (by Proposition~\ref{char_cogen_prop}) and, therefore, $E^\bullet:=\cone(\lambda^\bullet)$ (the usual mapping cone of a homomorphism of cochain complexes) is an exact complex. In particular, we obtain a  short exact sequence:
\[
0\longrightarrow \Sigma^{-1}Y^\bullet\longrightarrow \Sigma^{-1}E^\bullet \longrightarrow A^\bullet\longrightarrow 0,\qquad \text{with $\Sigma^{-1}Y^\bullet \in \mathrm{dg}\Y$ and $\Sigma^{-1}E^\bullet\in \E$.}
\]
Summarizing, we have the following result:
\begin{prop}\label{Proposition_8.6}
Let $\A$ be a bicomplete Abelian category, let $(\X,\Y)$ be a complete hereditary cotorsion pair in $\A$, and assume that the following two hypotheses are verified:
\begin{itemize}
    \item[\sc (h.1)] $\A$ is either (Ab.$4^*$), or (Ab.$4^*$)-$\Y$-$k$, for some $k\in \N$;
    \item[\sc (h.2)] $\A$ is either (Ab.$4$), or (Ab.$4$)-$\X$-$h$ (i.e., $\A^\op$ is (Ab.$4^*$)-$\X$-$h$), for some $h\in \N$.
\end{itemize}
Then, both $(\tilde \X,\dg\Y)$ and $(\dg\X,\tilde \Y)$ are complete cotorsion pairs in $\Ch(\A)$. Moreover, they are compatible, so that $(\E,\dg\X,\dg\Y)$ is a Hovey triple that corresponds to an Abelian model structure in $\Ch(\A)$ whose homotopy category is equivalent to the derived category $\D(\A)$.
\end{prop}
\begin{proof}
As we have discussed above, the completeness of $(\tilde \X,\dg\Y)$ follows if we can build, for each $A^\bullet\in \Ch(\A)$, a short exact sequence $0\to Y^\bullet\to E^\bullet\to A^\bullet\to 0$, with $Y^\bullet\in \dg\Y$, and $E^\bullet\in \E$.  Consider the two cases allowed by {\sc (h.1)}: if $\A$ is (Ab.$4^*$), one can use the Ding-Yang construction (see the proof of \cite[Lemma~2.1(1)]{yang2015question}) while, if $\A$ is (Ab.$4^*$)-$\Y$-$k$, then this follows from the existence of $\Y$-fibrant replacements, that can be built as a limit of a suitable tower of partial resolutions, as explained in Section~\ref{sec_seven}. The completeness of $(\dg\X,\tilde \Y)$ follows by duality, and the compatibility of the two cotorsion pairs is proved in \cite[Theorem~2.5]{yang2015question}.
\end{proof}
As a consequence of Proposition~\ref{Proposition_8.6}, we can answer Question~\ref{Gillespie_question} in the affirmative for any complete, hereditary cotorsion pair $(\X,\Y)$ in a bicomplete Abelian category $\A$ that is both (Ab.$4$) and (Ab.$4^*$) (e.g., when $\A$ is a category of modules). Moreover, a positive answer also follows, for example, when $\A$ is (Ab.$4$) and (Ab.$4^*$)-$\Y$-$k$ (for some $k\in \N$): we will see in Section~\ref{n_tilt_subs} that this is always the case when $\A$ is an (Ab.$4$) Abelian category such that $\D(\A)$ is locally small (e.g., a Grothendieck category), and $(\X,\Y)$ is a $k$-tilting cotorsion pair; hence, as $k$-tilting cotorsion pairs are all complete and hereditary, they induce Abelian model structures on $\Ch(\A)$ that enhance the derived category $\D(\A)$. The case of cotilting cotorsion pairs also follows by duality.

\subsection{Injective classes of injectives}
Let $\mathcal{A}$ be a complete Abelian category and $\I\subseteq \A$ an injective class. We have seen in Subsection~\ref{subs_cogen_inj} that,  whenever $\I$ is cogenerating, the classical derived category $\D(\A)$ is a Verdier quotient of $\mathcal{D}(\mathcal{A};\mathcal{I})$. In this subsection we concentrate on the injective classes $\I$ of injectives in a Grothendieck category $\G$, that is, $\I\subseteq \Inj(\G)$. These injective classes have the property that $\D(\G;\I)$ is always a Verdier quotient of $\mathcal{D}(\mathcal{G})$, so they are in some sense at the extreme opposite from the cogenerating ones; actually, the unique cogenerating injective class of injectives is $\Inj(\G)$, for which clearly $\mathcal{D}(\mathcal{G};\Inj(\G))\cong\D(\G)$.

\smallskip
The key to understand the injective classes of injectives in a Grothendieck category $\G$ is the observation that they correspond bijectively to the hereditary torsion pairs in $\G$ (see \cite[Theorem~4.8]{virili2017exactness}). More precisely,  an injective class  $\I\subseteq \Inj(\G)$ corresponds to the hereditary torsion pair $\tau_\I=(\T:={}^\perp\I,\F:=({}^\perp\I)^\perp)$, where $\F$ can also be described as $\mathrm{Cogen}(\I)$,  and $\I=\Inj(\G)\cap \F$.

\begin{rmk}
A consequence of the above connection with hereditary torsion pairs is that, in a Grothendieck category $\G$, a subclass of $\Inj(\G)$ is an injective class if and only if it is closed under taking products and summands. In fact, a similar characterization holds in any complete and well-powered Abelian category $\A$. Indeed, given $\I\subseteq \Inj(\A)$ closed under products and summands consider, for each $A\in \A$, the ``reject'':
\[
\mathrm{Rej}_\I(A):=\bigcap\{\ker(f):f\in \hom_\A(A,I),\, I\in \I\}\leq A,
\]
which is well-defined since $\A$ is well-powered. Moreover, being $\I$ closed under products, one can even find some $I\in \I$ and $f\colon A\to I$ such that $\mathrm{Rej}_\I(A)=\ker(f)$. In fact, such an $f$ is easily seen to be an $\I$-monomorphism, showing that $\I$ is an injective class.
\end{rmk}

\begin{lem}[{\cite[Lemmas~4.4 and 5.6]{virili2017exactness}}]
Let $\G$ be a Grothendieck category, $\I\subseteq \Inj(\G)$ an injective class of injectives, and let $\tau_\I=(\T,\F)$ be the associated hereditary torsion pair. Then
\begin{enumerate}
\item the following are equivalent for a morphism $\phi\colon X\to Y$ in $\G$:
\begin{enumerate}
\item[({1.}1)] $\phi$ is an $\I$-monomorphism;
\item[({1.}2)] the kernel of $\phi$ is $\tau_\I$-torsion, that is, $\Ker(\phi)\in\T$;
\end{enumerate} 
\item the following are equivalent for a morphism $\phi^\bullet\colon X^\bullet\to Y^\bullet$ in $\Ch(\G)$:
\begin{enumerate}
\item[({2.}1)] $\phi^\bullet$ is an $\I$-weak equivalence;
\item[({2.}2)] the cone $\cone(\phi^\bullet)$ is $\I$-acyclic;
\item[({2.}3)] the cone has $\tau_\I$-torsion cohomologies, that is, $H^n(\cone(\phi^\bullet))\in\T$, for all $n\in\Z$.
\end{enumerate} 
\end{enumerate}
\end{lem} 

Consider  the Gabriel quotient $\G/\T$, that comes with an adjunction $\Q\dashv \S\colon \G\leftrightarrows \G/\T$ such that:
\begin{itemize}
\item the right adjoint $\S$ is fully faithful, so denoting by $\eta \colon\id_\G\Rightarrow \S\circ \Q$ and $\varepsilon \colon \Q\circ \S\Rightarrow \id_{\G/\T}$ the unit and the counit, both $\varepsilon$ and $\Q(\eta)$ are natural isomorphisms. Abusing notation, we will often write $\Q(\S(Y))=Y$ and $\Q(X)=\Q(\S(\Q(X)))$, if $Y\in \G/\T$ and $X\in \G$;
\item the left adjoint $\Q$ is an exact functor such that $\Ker(\Q)=\T$, that is, $\Q(X)=0$ if and only if $X\in \T$. In particular, given a morphism $\phi\colon X\to Y$ in $\G$, $\Q(\phi)$ is a monomorphism (resp., epimorphism) if, and only if, $\Ker(\phi)\in \T$ (resp., $\coker(\phi)\in \T$). In particular, $\Ker(\eta_X),\, \coker(\eta_X)\in \T$, for all $X\in \G$;
\item as $\Q$ is exact, $\S$ sends injectives in $\G/\T$ to injectives in $\G$. In fact, the adjunction $\Q\dashv \S$ restricts to an equivalence of categories $\Inj(\G/\T)\cong \F\cap \Inj(\G)=\I$.
\end{itemize}
With a slight abuse of notation,  denote by $\Q\dashv \S\colon \Ch(\G)\leftrightarrows \Ch(\G/\T)$ also the adjunction induced pointwise on complexes. It is easy to reformulate the above lemma as follows:
\begin{cor}[{\cite[Lemmas~4.4 and 5.6]{virili2017exactness}}]\label{cor_rel_inj_of_inj}
Let $\G$ be a Grothendieck category, $\I\subseteq \Inj(\G)$ an injective class of injectives, and let $\tau_\I=(\T,\F)$ be the associated hereditary torsion pair. Then
\begin{enumerate}
\item a morphism $\phi$ in $\G$ is an $\I$-monomorphism if, and only, $\Q(\phi)$ is a monomorphism in $\G/\T$; 
\item a morphism $\phi^\bullet$ in $\Ch(\G)$ is an $\I$-weak equivalence if, and only if, $\Q(\phi^\bullet)$ is a quasi-isomorphism or, equivalently, if $\Q(\cone(\phi^\bullet))$ is an acyclic complex. 
\end{enumerate}
In particular, a morphism $\phi^\bullet$ in $\Ch(\G)$ is a (trivial) $\I$-cofibration if and only if $\Q(\phi^\bullet)$ is a (trivial) $\Inj(\G/\T)$-cofibration in $\Ch(\G/\T)$.
\end{cor} 

The following is a more precise version of \cite[Theorem~5.7]{virili2017exactness}:

\begin{thm}
Let $\G$ be a Grothendieck category, $\I\subseteq \Inj(\G)$ an injective class of injectives, and let $\tau_\I=(\T,\F)$ be the associated hereditary torsion pair. Then:
\begin{enumerate}
\item $\psi^\bullet$ is an $\Inj(\G/\T)$-fibration in $\Ch(\G/\T)$, if and only if $\S(\psi^\bullet)$ is an $\I$-fibration;
\item $E^\bullet\in \Ch(\I)\subseteq \Ch(\G)$ is $\I$-fibrant if, and only if, $\Q(E^\bullet)$ is \DG-injective in $\Ch(\G/\T)$;
\item any complex $X^\bullet\in \Ch(\G)$ admits an $\I$-fibrant replacement $\lambda^\bullet\colon X^\bullet\to E^\bullet\in \Ch(\I)$.
\end{enumerate}
In particular, $(\Ch(\G), \W_\I,\C_\I,\F_\I)$ is a model category, and the following is a Quillen equivalence
\[
\Q\dashv \S\colon \Ch(\G)\leftrightarrows \Ch(\G/\T),
\]
 when $\Ch(\G/\T)$ is endowed with the usual injective model structure. As a consequence, we can describe the homotopy category $\D(\G;\I):=\Ch(\G)[\W_\I^{-1}]$ as a Verdier quotient of the usual derived category as follows: $\D(\G;\I)\cong \D(\G/\T)\cong \D(\G)/\D_\T(\G)$ (where $\D_\T(\G)$ is the full, localizing, subcategory of those $X^\bullet\in \D(\G)$ such that $H^n(X^\bullet)\in \T$, for all $n\in \Z$). 
\end{thm}
\begin{proof}
Fix a morphism $\psi^\bullet\colon Y_1^\bullet\to Y_2^\bullet$ in $\Ch(\G/\T)$ and let $\phi^\bullet\colon X_1^\bullet\to X_2^\bullet$ be a trivial $\I$-cofibration in $\Ch(\G)$. By Corollary~\ref{cor_rel_inj_of_inj} $\Q(\phi^\bullet)$ is a trivial $\Inj(\G/\T)$-cofibration, and all the trivial $\Inj(\G/\T)$-cofibrations in $\Ch(\G/\T)$ are of this form. Consider now the following commutative diagrams:
\[
\xymatrix@C=20pt@C=35pt{
X_1^\bullet\ar[d]_-{\phi^\bullet}\ar[r]|-{\ \alpha_1^\bullet\ }&\S(Y_1^\bullet)\ar[d]^-{\S(\psi^\bullet)}&&\Q(X_1)^\bullet\ar[d]_-{\Q(\phi^\bullet)}\ar[r]|-{\ \beta_1^\bullet\ }&Y_1^\bullet\ar[d]^-{\psi^\bullet}\\
X_2^\bullet\ar@{.>}[ur]\ar[r]|-{\ \alpha_2^\bullet\ }&\S(Y_2^\bullet)&&\Q(X_2)^\bullet\ar@{.>}[ur]\ar[r]|-{\ \beta_2^\bullet\ }&Y_2^\bullet,
}
\] 
where $\alpha_i^\bullet\in\hom_{\Ch(\G)}(X_i^\bullet, \S(Y_i^\bullet))$ and $\beta_i^\bullet\in \hom_{\Ch(\G/\T)}(\Q(X_i^\bullet), Y_i^\bullet)$ (for $i=1,2$) correspond to each other via the adjunction $\Q\dashv \S$. By adjunction, $\phi^\bullet \boxslash \S(\psi^\bullet)$ if, and only if, $\Q(\phi^\bullet) \boxslash \psi^\bullet$, proving (1). In particular, $E^\bullet\in \Ch(\I)\subseteq \Ch(\G)$ is $\I$-fibrant if, and only if, $E^\bullet\to 0$ is an $\I$-fibration, if and only if $\Q(E^\bullet)\to 0$ is an $\Inj(\G/\T)$-fibration, i.e., $\Q(E^\bullet)\in \Ch(\Inj(\G/\T))$ is a \DG-injective complex, verifying (2). For (3), let $X^\bullet\in \Ch(\G)$ and consider a \DG-injective resolution $\lambda^\bullet\colon \Q(X^\bullet)\to E^\bullet$ in $\Ch(\G/\T)$. We claim that  the following composition  is an $\I$-fibrant replacement
\[
\S(\lambda^\bullet)\circ \eta_{X^\bullet}\colon X^\bullet\longrightarrow \S(\Q(X^\bullet))\longrightarrow \S(E^\bullet).
\] 
Indeed, by Corollary~\ref{cor_rel_inj_of_inj}, both the unit $\eta_{X^\bullet}$  (which is sent to an isomorphism by $\Q$) and the map $\S(\lambda^\bullet)$ (that is sent by $\Q$ to $\lambda^\bullet=\Q(\S(\lambda^\bullet))$, which is a quasi-isomorphism by construction) are $\I$-weak equivalences. Hence, the composition $\S(\lambda^\bullet)\circ \eta_{X^\bullet}$ is an $\I$-weak equivalence, while the complex $\S(E^\bullet)\in \Ch(\I)$ is $\I$-fibrant by part (2). By Theorem~\ref{CH_main_thm}, this also shows that $(\Ch(\G), \W_\I,\C_\I,\F_\I)$ is a model category. By Corollary~\ref{cor_rel_inj_of_inj}, $\Q$ preserves both cofibrations and trivial cofibrations, so $\Q\dashv \S$ is a Quillen adjunction. Finally, to see that this is a Quillen equivalence, take $\phi^\bullet\colon X^\bullet\to \S(Y^\bullet)$ in $\Ch(\G)$, whose adjoint map is $\psi^\bullet:=\Q(\phi^\bullet)\colon \Q(X^\bullet)\to \Q(\S(X^\bullet))=X^\bullet$; by Corollary~\ref{cor_rel_inj_of_inj}, $\phi^\bullet\in \W_\I$ if and only if $\psi^\bullet\in \W_{\Inj(\G/\T)}$, concluding the proof.
\end{proof}

Observe that, for any object $X\in \G$, a morphism $\lambda\colon S^0(X)\to I^\bullet\in \Ch^{\geq0}(\I)$ is a relative $\I$-injective resolution if, and only if, $\Q(\lambda)\colon  S^0(\Q(X))\to \Q(I^\bullet)$ is a \DG-injective resolution of $\Q(X)\in \G/\T$. In particular, one can construct a resolution of a given $X\in \G$ as follows: first build a \DG-injective resolution $S^0(\Q(X))\to E^\bullet\in \Ch^{\geq0}(\Inj(\G/\T))$ of $\Q(X)$ in $\G/\T$, and then note that the composition $S^0(X)\to S^0(\S(\Q(X)))\to \S(E^\bullet)$ is the desired $\I$-injective resolution. Recall also that, being $\S$ a right adjoint, it commutes with all limits (and, in particular, with products). Thus, given $(Y^\bullet_\lambda)_\Lambda\subseteq \Ch(\G/\T)$, $\S(\prod_\Lambda Y^\bullet_\lambda)=\prod_\Lambda\S(Y^\bullet_\lambda)$. Exploiting these two ideas, we can show that the (Ab.4$^*$)-$\I$-$k$ condition in $\G$ is equivalent to the ``absolute'' (Ab.4$^*$)-$k$ condition in $\G/\T$:

\begin{prop}\label{exactness_of_products_on_quotients}
Let $\G$ be a Grothendieck category, $\I\subseteq \Inj(\G)$ an injective class of injectives, and let $\tau_\I=(\T,\F)$ be the associated hereditary torsion pair. Then, the following are equivalent:
\begin{enumerate}
\item $\mathcal{G}$ is (Ab.4$^*$)-$\mathcal{I}$-$k$;
\item $\mathcal{G}/\mathcal{T}$ is (Ab.4$^*$)-$k$ (in the sense of Roos).
\end{enumerate}
\end{prop}
\begin{proof}
Let $(X_\lambda)_\Lambda\subseteq \G$ and, for each $\lambda\in \Lambda$, take a \DG-injective resolution $\lambda_\lambda\colon S^0(\Q(X_\lambda))\to E_\lambda^\bullet$ in $\Ch^{\geq0}(\Inj(\G/\T))$, so that $\mu_\lambda\colon S^0(X_\lambda)\to \S(E_\lambda^\bullet)\in \Ch^{\geq0}(\I)$ is the corresponding relative $\I$-injective resolution, for all $\lambda\in \Lambda$, as discussed above. By Lemma~\ref{lem.AB4-I-k}(1), $\mathcal{G}$ is (Ab.4$^*$)-$\mathcal{I}$-$k$ if, and only if:
\begin{equation}\label{lem_AB4_I_k_part1_eq}
\xymatrix{H^{-n}(\hom(\prod_{\Lambda}\S(E_\lambda^\bullet) ,I))=0, \qquad\text{for all $n>k$, and all $I\in \I$.}}
\end{equation}
Since $\I\subseteq \Inj(\G)$ by hypothesis, the functor $\hom_\G(-,I)$ is exact and, therefore, it commutes with cohomologies, so \eqref{lem_AB4_I_k_part1_eq} becomes equivalent to:
\[
\xymatrix{\hom_\G(H^{n}(\prod_{\Lambda}(\S(E_\lambda^\bullet)) ,I)=0, \qquad\text{for all $n>k$, and all $I\in \I$.}}
\]
As, by definition, $\T:={}^{\perp}\I$, the above condition means exactly that $H^{n}(\prod_{\Lambda}(\S(E_\lambda^\bullet))\in \T$, for all $n>k$, that is, $\Q(H^{n}(\prod_{\Lambda}(\S(E_\lambda^\bullet)))=0$, for all $n>k$. Using that $\Q$ commutes with cohomologies, $\S$  commutes with products, and $\Q\circ \S=\id_{\G/\T}$, condition \eqref{lem_AB4_I_k_part1_eq} becomes equivalent to:
\[
\xymatrix{
\prod^{(n)}_\Lambda \Q(X_\lambda):=H^{n}(\prod_{\Lambda}E_\lambda^\bullet)=H^{n}(\Q\circ \S(\prod_{\Lambda}E_\lambda^\bullet))=0, \qquad\text{for all $n>k$.}
}
\] 
This last equation is precisely the (Ab.4$^*$)-$k$ condition for the quotient category $\mathcal{G}/\mathcal{T}$.
\end{proof}

\subsection{The pure derived category of a locally finitely presented Grothendieck category}\label{subs_purity}

Let $\G$ be a Grothendieck category and recall that an object $X$ in $\G$ is said to be {\bf finitely presented} if $\hom_\G(X,-)$ commutes with directed colimits. We denote by $\mathrm{fp}(\G)$ the full subcategory of finitely presented objects in $\G$, and we say that $\G$ is {\bf locally finitely presented} provided $\mathrm{fp}(\G)$ is essentially small, and it generates $\G$. Observe that, whenever $\G$ is locally finitely presented, the so-called {\bf restricted Yoneda embedding} gives a fully faithful functor 
\[
\y\colon \G\longrightarrow [\mathrm{fp}(\G)^{\mathrm{op}},\Ab], \text{ such that $\y(X)=\y_X:=\hom_\G(-,X)_{\restriction \mathrm{fp}(\G)}$.}
\] 
Observe that $\y$ is left exact, and that it commutes with products and directed colimits. The category $[\mathrm{fp}(\G)^{\mathrm{op}},\Ab]$ is generated by $\y(\mathrm{fp}(\G)):=\{\y_P:P\in \mathrm{fp}(\G)\}$, the essentially small subcategory of the representable (=finitely presented projective) functors, and so $[\mathrm{fp}(\G)^{\mathrm{op}},\Ab]$ is an (Ab.4$^*$) Grothendieck category (it is equivalent to a category of modules over a ring with enough idempotents). Let us also recall that:
\begin{itemize}
\item $F\colon \mathrm{fp}(\G)^{\mathrm{op}}\to \Ab$ is said to be {\bf flat} if it belongs to the class $\F:=\varinjlim \y(\mathrm{fp}(\G))$. The restricted Yoneda embedding induces an equivalence $\G\cong \y(\G)=\F$;
\item the functors in the class $\C:=\F^{\perp_1}=(\y(\mathrm{fp}(\G)))^{\perp_1}$ are called {\bf cotorsion}, and $(\F,\C)$ is a complete cotorsion pair in $[\mathrm{fp}(\G)^{\mathrm{op}},\Ab]$ (see \cite[Theorem~2.6]{doi:10.1081/AGB-120028794});
\item a short exact sequence $0\to X\to Y\to Z\to 0$ in $\G$ is said to be {\bf pure exact} if, and only if, $0\to \y_X\to \y_Y\to \y_Z\to 0$ is exact in $[\mathrm{fp}(\G)^{\mathrm{op}},\Ab]$. We say that $X\to Y$ is a {\bf pure monomorphism} if it can be completed to a pure exact sequence  $0\to X\to Y\to Z\to 0$;
\item an object $E\in \G$ is said to be {\bf pure injective} if it is injective with respect to all the pure exact sequences in $\G$, i.e., if $0\to \hom_\G(Z,E)\to \hom_\G(Y,E)\to \hom_\G(X,E)\to 0$ is exact in $\Ab$, for any pure exact sequence $0\to X\to Y\to Z\to 0$ in $\G$;  
\item if we denote by $\mathrm{P.Inj}(\G)$ the class of pure injectives in $\G$, the restricted Yoneda embedding induces an equivalence $\mathrm{P.Inj}(\G)\cong \y(\mathrm{P.Inj}(\G))=\F\cap \C$ (see \cite[Lemma~3]{doi:10.1142/S0219498803000581}).
\end{itemize}

\begin{prop}\label{p_der_cat_prop}
Let $\G$ be a locally finitely presented Grothendieck category, and  $\I:=\mathrm{P.Inj}(\G)$:
\begin{enumerate}
\item $\I$ is an injective class in $\G$, and the $\I$-monomorphisms are  the pure monomorphisms;
\item $\G$ is (Ab.4$^*$)-$\I$, and so $(\Ch(\G),\W_\I,\F_\I,\C_\I)$ is a model category.
\end{enumerate}
In particular, the homotopy category $\D(\G;\I):=\Ch(\G)[\W_\I^{-1}]$ is locally small, and it is equivalent to the so-called  {\bf pure derived category} $\mathcal{D}_{\mathrm{pure}}(\mathcal{G})$ introduced in \cite[Section~5.3]{CH} when $\mathcal{G}$ is a module category, and studied in \cite{Krause:2012aa}.
\end{prop}
\begin{proof}
Recall from \cite[Theorem~6]{doi:10.1142/S0219498803000581} that, for any $X\in \G$, there exists a pure exact sequence $0\to X\to PE(X)\to PE(X)/X\to 0$ in $\G$, with $PE(X)\in \I$. In particular, $\I$ is cogenerating.

\smallskip
(1) It is clear that, if $X\to Y$ is a pure monomorphism, then it is also an $\I$-monomorphism (by our definition of pure injective object). On the other hand, if $f\colon X\to Y$ is an $\I$-monomorphism, take a pure monomorphism $g\colon X\to PE(X)$, with $PE(X)\in \I$, and observe that (by definition of $\I$-monomorphism) $f^*\colon \hom_\G(Y,PE(X))\to \hom_\G(X,PE(X))$ is surjective. In particular, there exists $\tilde g\colon Y\to PE(X)$, such that $\tilde g\circ f=:f^*(\tilde g)=g$. In other words, $\y_g=\y_{\tilde g}\circ \y_{f}$ is a monomorphism in $[\mathrm{fp}(\G)^{\mathrm{op}}, \Ab]$, and so $\y_f$ is a monomorphism; i.e., $f$ is a pure monomorphism. 

\smallskip
(2) Consider a family of pure exact sequences $\{0\to X_\lambda\to Y_\lambda\to Z_\lambda\to 0\}_{\Lambda}$ in $\G$, which gives a family of short exact sequences of functors $\{0\to \y_{X_\lambda}\to  \y_{Y_\lambda}\to  \y_{Z_\lambda}\to 0\}_{\Lambda}$. Since $[\mathrm{fp}(\G)^{\mathrm{op}}, \Ab]$ is (Ab.4$^*$), we obtain a short exact sequence $0\to \prod_{\Lambda}  \y_{X_\lambda}\to   \prod_{\Lambda}\y_{Y_\lambda}\to  \prod_{\Lambda} \y_{Z_\lambda}\to 0$. Moreover, since the restricted Yoneda embedding commutes with products, we obtain the following short exact sequence:  $0\to  \y_{\prod_{\Lambda} \y_\lambda}\to   \y_{\prod_{\Lambda} Y_\lambda}\to   \y_{\prod_{\Lambda} Z_\lambda}\to 0$. Hence, $0\to  \prod_{\Lambda} X_\lambda\to   \prod_{\Lambda} Y_\lambda\to  \prod_{\Lambda} Z_\lambda\to 0$ is pure exact (i.e., $\I$-exact) in $\G$. One can now conclude by Lemma~\ref{criterion_ab4*_I_cogen_lem} and Corollary~\ref{cor_rel_inj_model_structure}.

\smallskip
For the last statements,  it is enough to check that the $\mathcal{I}$-acyclic complexes are precisely the pure-acyclic complexes. As both the $\mathcal{I}$-acyclic and the pure-acyclic complexes are acyclic, this reduces to prove that a short exact sequence $0\rightarrow A\to B\to C\rightarrow 0$ in $\mathcal{G}$ is pure if, and only if, it ``is" an $\mathcal{I}$-acyclic complex which, in turn, is equivalent to prove that a monomorphism $u\colon A\to B$ in $\mathcal{G}$ is pure if, and only if, it is an $\mathcal{I}$-monomorphism, that follows by (1).
\end{proof}

\subsection{Categories of finite $\I$-global dimension} A rich source of examples of  (Ab.4$^*$)-$\mathcal{I}$-$k$ Abelian categories is derived from the following concepts (see Corollary~\ref{cor.finite I-global dimension} below):

\begin{defn}
Let $\mathcal{A}$ be a complete Abelian category and let $\mathcal{I}\subseteq \A$ be an injective class. For each object $A$ in $\A$, define the following set:
\[
\mathbb{N}(A,\mathcal{I}):=\{n\in \N:\text{there is a relative $\I$-injective resolution $A\to I^\bullet\in \Ch^{\geq0}(\A)\cap\Ch^{\leq n}(\A)$}\}.
\]
Then, we define the {\bf $\mathcal{I}$-codimension} of $A$ as follows:
\[
\I\codim A:=\inf \mathbb{N}(A,\mathcal{I}),
\]
with the convention that $\inf \emptyset:=\infty$. Finally, we define the {\bf $\mathcal{I}$-global dimension of $\mathcal{A}$} as:
\[
\I\dim \A:=\sup \{\I\codim A:A\in \A\}\in \N\cup\{\infty\}.
\]
\end{defn}

We have the following immediate consequence.

\begin{cor} \label{cor.finite I-global dimension}
Let $\mathcal{A}$ be a complete Abelian category and let $\mathcal{I}\subseteq \A$ be an injective class such that $\mathcal{I}\dim \A=d<\infty$. Then, $\mathcal{A}$ is (Ab.4$^*$)-$\mathcal{I}$-$d$.
\end{cor}
\begin{proof}
Let $(A_\lambda)_{\lambda\in\Lambda}$ be a family of objects in $\mathcal{A}$ and choose, for each $\lambda$, a relative $\I$-injective resolution $u_\lambda\colon A_\lambda\to I^\bullet_\lambda\in \Ch^{\geq0}(\A)\cap\Ch^{\leq d}(\A)$. Observe then that $\prod_\Lambda I^\bullet_\lambda\in \Ch^{\geq0}(\A)\cap\Ch^{\leq d}(\A)$ and, therefore, $\hom(\prod_\Lambda I^\bullet_\lambda,I)\in \Ch^{\leq0}(\Ab)\cap\Ch^{\geq -d}(\Ab)$. In particular,  $H^k(\hom(\prod_\Lambda I^\bullet_\lambda,I))=0$, for all $k<-d$, that is, $\mathcal{A}$ is (Ab.4$^*$)-$\mathcal{I}$-$d$.
\end{proof}

We are now going to describe a few concrete examples of bicomplete Abelian categories $\A$ that admit a suitable injective class $\I\subseteq \A$ for which $\I\dim \A<\infty$. By the above corollary, these are also examples where $\A$ is (Ab.4$^*$)-$\I$-$k$, for some $k\in \N$. In particular, by Corollary~\ref{cor_rel_inj_model_structure}, in each of these cases $(\Ch(\A), \W_\I,\C_\I,\F_\I)$ is a model category, and so the corresponding homotopy category $\D(\A;\I):=\Ch(\A)[\W_\I^{-1}]$ is locally small. The first two examples below have also appeared in \cite[Section~9]{zbMATH06915995}, with different terminology.

\begin{eg}
Let $R$ be a $d$-Gorenstein ring, i.e., a (left and right) Noetherian ring  that has injective dimension $=d$ on both sides, for a given $d\in \N$, and consider the following class: 
\[
\mathcal{I}=\mathcal{GI}(R):=\{M\in\Mod\text{-}R:\text{ $M$ is Gorenstein injective}\}\subseteq \Mod\text{-}R=:\A.
\]
By \cite[Theorem~2.3]{RSMUP_2002__107__67_0}, $\mathcal{I}$ is an injective class in $\Mod\text{-}R$, and it follows by  \cite[Chapter~12]{enochs2011relative} that $\mathcal{I}\text{-}\gdim(\mathcal{A})=d$.
\end{eg}

\begin{eg}
Let $R$ be a ring of pure global dimension $=d$, for some $d\in \N$  (e.g., if $d>0$ and $|R|\leq\aleph_{d-1}$, then $R$ has pure global dimension $\leq d$). Let $\mathcal{A}:=\Mod\text{-}R$, and consider the injective class $\mathcal{I}:=\mathrm{P.Inj}(\A)$ of the pure-injective right $R$-modules. By definition (see \cite[Section~1]{430dc0673e7b4b03ac859c068dbe10e1}), the pure global dimension of $R$ is the supremum of the pure-injective dimensions of its modules. Hence,  $\mathcal{I}\text{-}\gdim(\mathcal{A})=d$ coincides with the pure global dimension of $R$.
\end{eg}

In a bicomplete Abelian category $\A$, a given $V\in \A$ is {\bf $1$-tilting} (\cite[Definition~6.1]{Parra:2023aa}) if:
\begin{itemize}
\item $V^{\perp_1}:=\Ker(\Ext^1_\A(V,-))=\mathrm{Gen}(V)=:\I$;
\item $\I$ is cogenerating in $\A$.
\end{itemize}
We refer to \cite[Section~6]{Parra:2023aa} for several characterizations of $1$-tilting objects and their connection with the traditional concept of a $1$-tilting module. In the following example we show that, if $V$ is $1$-tilting, then $\mathrm{Gen}(V)\text{-}\gdim(\A)=1$; for an extension to the $n$-tilting case, see Subsection~\ref{n_tilt_subs}.

\begin{eg}
Let $\mathcal{A}$ be bicomplete Abelian category with a generator, $V\in\A$ a $1$-tilting object, and let $\I:=\mathrm{Gen}(V)=V^{\perp_1}$. By \cite[Proposition~6.7]{Parra:2023aa}, there is a generator $G$ of $\mathcal{A}$ together with a short exact sequence $0\rightarrow G\stackrel{u}{\rightarrow} V_0\rightarrow V_1\rightarrow 0$ (with $V_0,\, V_1\in\Add(V)$), such that any coproduct of copies of the sequence remains exact. Now, given $A\in\mathcal{A}$, choose an epimorphism $\pi\colon G^{(S)}\rightarrow A$ and observe that $u^{(S)}\colon G^{(S)}\rightarrow V_0^{(S)}$ is an $\mathcal{I}$-preenvelope, since $\Ext_\mathcal{A}^1(V_1^{(S)},I)=0$, for all $I\in\mathcal{I}$. Taking the push-out of $u^{(S)}$ along $\pi$, one gets a short exact sequence $0\rightarrow A\stackrel{v}{\rightarrow} I_A\rightarrow V_1^{(S)}\rightarrow 0$, with $I_A\in\mathcal{I}$. By construction, $v$ is clearly an $\mathcal{I}$-preenvelope, so  $\mathcal{I}\codim A\leq 1$. Thus, $\mathcal{I}\text{-}\gdim(\mathcal{A})=1$.
\end{eg}

\begin{eg}
 Let $V$ be a projective variety over a field $K$, i.e., a closed set for the Zariski topology of $\mathbf{P}^n(K)$, for some integer $n>0$, let $\Gamma (V)$ be its (graded) ring of coordinates, and let $\A:=\mathrm{Gr}\text{-}\Gamma (V)$ be the category of graded $\Gamma (V)$-modules. If $V$ is regular of dimension $d$, let $\mathcal{I}$ be the class of injective objects of $\mathrm{Gr}\text{-}\Gamma (V)$ that have zero graded socle. Now, if $\mathcal{G}:=\mathrm{Qcoh}(V)$ is the Grothendieck category of quasi-coherent sheaves on $V$, when we view $V$ as a scheme in the usual way, the regularity of $V$ implies that its dimension coincides with the global dimension of $\mathcal{G}$, i.e., $\Inj(\mathcal{G})\text{-}\gdim(\mathcal{G})$ with our terminology. A well-known result of Serre (see Proposition~7.8 in Section~59, p.~252 in \cite{Serre_Faisceaux}) says that there is an equivalence of categories $\mathcal{G}\cong(\mathrm{Gr}\text{-}\Gamma(V))/\mathcal{T}$, where $\mathcal{T}\subseteq\mathrm{Gr}\text{-}\Gamma(V)$ is the hereditary torsion class  of the locally finite graded $\Gamma (V)$-modules. This is exactly the torsion class generated by the class $\mathcal{S}$ of graded-simple modules, and hence the corresponding torsion pair is cogenerated by $\mathcal{S}^\perp\cap\Inj(\Gamma (V))=\mathcal{I}$. By the discussion right before Proposition~\ref{exactness_of_products_on_quotients}, it is now easy to see that $\mathcal{I}\text{-}\gdim(\mathcal{A})=d$.
\end{eg}

\begin{prop} \label{propXX_adelanto}
Let $\mathcal{A}$ be an Abelian category such that $\D(\A)$ is locally small, and let $(\mathcal{X},\mathcal{Y})$ be a right complete hereditary cotorsion pair. Then, the following  are equivalent for any  $n\geq 0$:
\begin{enumerate}
\item $\Ext_\mathcal{A}^{k}(X',X)=0$, for all $X,X'\in\mathcal{X}$ and all $k>n$;
\item $\Ext_\mathcal{A}^{n+1}(X',X)=0$, for all $X,X'\in\mathcal{X}$;
\item for any $X\in\mathcal{X}$, there is an exact sequence of the form: $0\rightarrow X\rightarrow W^0\rightarrow\cdots\rightarrow W^n\rightarrow 0$, with $W^i\in \W:=\mathcal{X}\cap\mathcal{Y}$, for all $i=0,\dots,n$.
\end{enumerate}
Moreover, if $\mathcal{Y}$-$\gdim(\mathcal{A})\leq n$, then the above equivalent conditions are all verified.
\end{prop}
\begin{proof}
(1)$\Rightarrow$(2) is clear.

\smallskip
(2)$\Rightarrow$(3). Let $X\in \X$ and consider a right $(\X,\Y)$-approximation $0\to X\to W^0\to X^1\to 0$, where  $W^0\in \Y$, $X^1\in \X$ and, since $\X={}^{\perp_1}\Y$ is closed under extensions, also $W^0\in \X$, so that $W^0\in \W$. Continue by taking a right $(\X,\Y)$-approximation $0\to X^1\to W^1\to X^2\to 0$, with $W^1\in \W$ and $X^2\in \X$, and proceed inductively in this way to construct a $\W$-coresolution 
\begin{equation}\label{eq_W_cores}
\xymatrix{
0\ar[r]& X\ar[r]^-{u}& W^0\ar[r]^-{d^0}& \cdots\ar[r]^-{d^{n-1}}& W^n\ar[r]^-{d^n}& \cdots,
}
\end{equation}
with $\Im(d^i)\in \X$, for all $i\geq0$. Observe that, letting $X':=\Im(d^n)\in \X$, any object $W\in \W$ is $\hom_\A(X',-)$-acyclic (i.e., $\Ext^i_\A(X',W)=0$, for all $i>0$), and so the $\W$-coresolution in \eqref{eq_W_cores} gives the following formula: $\Ext^i_\A(X',X)\cong H^i(\hom_\A(X',W^\bullet))$, for all $i>0$. By (2), we deduce that $H^{n+1}(\hom_\A(X',W^\bullet))=0$ or, equivalently, that the map $\hom_\A(X',W^n)\to \hom_\A(X',X')$ is surjective, that is, $W^n\cong \Im(d^{n-1})\oplus X'$, showing that $\Im(d^{n-1})\in \W$, as this class is closed under summands.  In particular,  we obtain a $\W$-coresolution of the desired length:
\[
\xymatrix{
0\ar[r]& X\ar[r]^-{u}& W^0\ar[r]^-{d^0}& \cdots\ar[r]^-{d^{n-2}}& W^{n-1}\ar[r]^-{d^{n-1}}&\Im(d^{n-1})\ar[r]&0\ar[r]& \cdots.
}
\]

(3)$\Rightarrow$(1). Given $X,\, X'\in\mathcal{X}$, take a $\mathcal{W}$-coresolution $0\rightarrow X\rightarrow W^0\rightarrow\cdots\rightarrow W^n\rightarrow 0$ and observe that, as in the previous implication, $\Ext^i_\A(X',X)\cong H^i(\hom_\A(X',W^\bullet))$, for all $i>0$. Given $k>n$, $\hom_\A(X',W^k)=\hom_\A(X',0)=0$, and so $\Ext^k_\A(X',X)\cong H^k(\hom_\A(X',W^\bullet))=0$. 

\smallskip
For the final statement, let $X,\, X'\in\mathcal{X}$. If $\mathcal{Y}$-$\gdim(\mathcal{A})\leq n$, take a relative $\mathcal{Y}$-injective resolution 
\begin{equation}\label{eq_Y_cores_eq}
\xymatrix{
0\ar[r]& X\ar[r]^-{u}& Y^0\ar[r]^-{d^0}& \cdots\ar[r]^-{d^{n-1}}& Y^n\ar[r]&0\ar[r]& \cdots.
}
\end{equation}
Since $\Y$ is cogenerating, the sequence \eqref{eq_Y_cores_eq} is exact while, as any $Y\in \Y$ is $\hom_\A(X',-)$-acyclic, $\Ext^i_\A(X',X)\cong H^i(\hom_\A(X',Y^\bullet))$, for all $i>0$. This clearly implies (1). 
\end{proof}

As a direct consequence of Proposition~\ref{propXX_adelanto}, we get:

\begin{cor} \label{cor.final1}
Let $\mathcal{A}$ be an Abelian category such that $\D(\A)$ is locally small, let $(\mathcal{X},\mathcal{Y})$ be a right complete hereditary cotorsion pair in $\mathcal{A}$, and fix $n\in\mathbb{N}$  such that the equivalent conditions of Proposition~\ref{propXX_adelanto} are verified. Then, $\mathcal{Y}\text{-}\gdim(\mathcal{A})\leq n+1$ and, therefore, $\mathcal{A}$ is (Ab.4$^*$)-$\mathcal{Y}$-$(n+1)$.
\end{cor}
\begin{proof}
Given $A\in\mathcal{A}$, take a right $(\X,\Y)$-approximation $0\rightarrow A\rightarrow Y\rightarrow X\rightarrow 0$, with $Y\in \Y$ and $X\in \X$. Take now a $\W$-coresolution of $X$ as in  Proposition~\ref{propXX_adelanto}(3), that is, an exact sequence $0\rightarrow X\rightarrow W^0\rightarrow \cdots\rightarrow W^n\rightarrow 0$, with $W^i\in\mathcal{W}$ for all $i=0,\dots,n$. We obtain an exact sequence $0\rightarrow A\rightarrow Y\rightarrow W^0\rightarrow \cdots\rightarrow W^n\rightarrow 0$, that is clearly a relative $\mathcal{Y}$-injective resolution of $A$.
\end{proof}

\subsection{$n$-Tilting cotorsion pairs}\label{n_tilt_subs} 
In this subsection we will work on an Abelian category $\mathcal{A}$ that is (Ab.4), and such that $\mathcal{D}(\mathcal{A})$ is locally small. In particular, $\mathcal{D}(\mathcal{A})$ has arbitrary coproducts, which are computed degree-wise. These conditions are always verified in each of the following cases:
\begin{itemize}
\item if $\A$ is a Grothendieck category;
\item if $\A$ is a bicomplete (Ab.4) Abelian category with enough projectives; this follows by \cite[Theorem~6.4]{zbMATH06915995}, or even Corollary~\ref{cor_rel_inj_model_structure}, applied to $\A^{\mathrm{op}}$ with the injective class $\Inj(\A^{\mathrm{op}})=(\mathrm{Proj(\A)})^{\mathrm{op}}$.
\end{itemize} 
Recall that the {\bf projective dimension} $\pdim_\A(X)$ of an object $X\in \A$ is defined as follows:
\[
\pdim_\A(X):=\inf\{n\in \N:\Ext_\mathcal{A}^k(X,A)=0,\,\forall k>n,\, \forall A\in\mathcal{A}\},
\]
with the convention that $\inf\emptyset=\infty$. The following definition comes from \cite[Definition~6.8]{NICOLAS20192273}:
\begin{defn} \label{def.n-tilting}
Let $\mathcal{A}$ be an (Ab.4) Abelian category such that $\D(\A)$ is locally small, and let $n>0$ be an integer. An object $T$ of $\mathcal{A}$ is called {\bf $n$-tilting} when the following conditions hold:
\begin{enumerate}
\item[(T.1)] $\Ext_\mathcal{A}^k(T,T^{(I)})=0$, for any set $I$ and all $k>0$;
\item[(T.2)] $\pdim_\A(T)=n$;
\item[(T.3)] there is a generating class $\mathcal{P}\subseteq\mathcal{A}$ such that, for each $P\in\mathcal{P}$, there is an exact sequence
\begin{equation*}\label{eq_Y_cores}
\xymatrix{
0\ar[r]& P\ar[r]^-{u}& T^0\ar[r]^-{d^0}& \cdots\ar[r]^-{d^{n-1}}& T^n\ar[r]&0,
}
\end{equation*} 
with $T^k\in \Add(T)$, for all $k=0,\dots, n$.
\end{enumerate}
\end{defn}

Any $n$-tilting object, has an associated cotorsion pair, that we describe in the following lemma:

\begin{lemdef} \label{lem.n-tilting cotorsion pair}
Let $\mathcal{A}$ be an (Ab.4) Abelian category such that $\D(\A)$ is locally small, let $T\in\mathcal{A}$ be $n$-tilting (with $n\geq 1$), and let $(\mathcal{X},\mathcal{Y}):=(^{\perp_{>0}} (T^{\perp_{>0}}),T^{\perp_{>0}})$. Then, the following statements hold true:
\begin{enumerate}
\item $(\X,\Y)$ is a complete and hereditary cotorsion pair in $\A$;
\item $\W:=\mathcal{X}\cap\mathcal{Y}=\Add(T)$;
\item $\Y\text{-}\gdim(\A)\leq n$ and, therefore, $\mathcal{A}$ satisfies the (Ab.4$^*$)-$\mathcal{Y}$-$n$ condition.
\end{enumerate}
The pair $(\mathcal{X},\mathcal{Y}):=(^{\perp_{>0}} (T^{\perp_{>0}}),T^{\perp_{>0}})$ is called the {\bf $n$-tilting cotorsion pair} associated to $T$.
\end{lemdef}
\begin{proof}
(1) Observe that the assignments $\mathcal S\mapsto{}^{\perp_{>0}}\mathcal S$ and $\mathcal S\mapsto\mathcal S^{\perp_{>0}}$ define a Galois connection of the (large) poset of subclasses of $\Ob(\A)$ with itself,  so $({}^{\perp_{>0}}(\mathcal S^{\perp_{>0}}))^{\perp_{>0}}=\mathcal S^{\perp_{>0}}$ and ${}^{\perp_{>0}}(({}^{\perp_{>0}}\mathcal S)^{\perp_{>0}})={}^{\perp_{>0}}\mathcal S$, for any  $\mathcal S\subseteq \Ob(\A)$. In particular, ${}^{\perp_{>0}}\Y=\X$ (by construction), and $\X^{\perp_{>0}}=\Y$ (by the previous discussion). Moreover, by the first part of Remark~\ref{rem_on_hereditary}, $\X$ is closed under kernels of epimorphisms so, in the notation of Definition~\ref{def.n-tilting}, the generating class $\mathcal P$ is contained in $ \X$ and, therefore, $\X$ is generating. On the other hand, for each $A\in \A$, \cite[Lemma~6.13]{NICOLAS20192273} gives us an exact sequence of the form:
\begin{equation}\label{eq_Y_cores_pf_lem}
\xymatrix{
0\ar[r]& A\ar[r]^-{u}& Y\ar[r]^-{d^0}& T^1\ar[r]^-{d^1}& \cdots\ar[r]^-{d^{n-1}}& T^n\ar[r]&0,
}
\end{equation} 
such that $Y\in\mathcal{Y}$ and $T^k\in\Add(T)$, for all $k=1,\dots,n$. Since $\mathcal{X}$ is closed under kernels of epimorphisms, $X^{i-1}:=\Ker(d^{i-1})\in\mathcal{X}$, for all $i=2,\dots,n$, so  that $0\rightarrow A\rightarrow Y\rightarrow X^1\rightarrow 0$ is a right $(\X,\Y)$-approximation of $A$; in particular, $\Y$ is cogenerating. It is now easy to conclude by Lemma~\ref{her_com_cot_pair_lem}.

\smallskip
(2) The inclusion ``$\Add(T)\subseteq\mathcal{W}$'' is clear. On the other hand, given $W\in\mathcal{W}\subseteq \Y\subseteq \mathrm{Gen}(T)$ (see \cite[Theorem~6.3 and Remark~6.2]{NICOLAS20192273} and their proofs), if we denote by $\varepsilon_f\colon T\to T^{(\hom_\mathcal{A}(T,W))}$ the inclusion in the coproduct of the copy of $T$ corresponding to the morphism $f\in \hom_\mathcal{A}(T,W)$, we get a canonical morphism $p_W\colon T^{(\hom_\mathcal{A}(T,W))}\rightarrow W$ such that $p_W\circ \varepsilon_f=f$, for all $f\in \hom_\mathcal{A}(T,W)$. The condition $W\in\mathrm{Gen}(T)$ can be expressed equivalently by saying that $p_W$ is an epimorphism; let $W':=\ker(p_W)$. The condition $f=p_W\circ \varepsilon_f=:p_W^*(\varepsilon_f)$ for all $f\colon T\to W$, shows the surjectivity of $p_W^*\colon \hom_\A(T,T^{(\hom_\mathcal{A}(T,W))})\rightarrow \hom_\A(T,W)$, so $\Ext^1_\A(T,W')=0$ by the exact sequence below:
\[
\xymatrix@C=10pt{
\hom_\A(T,T^{(\hom_\mathcal{A}(T,W))})\ar@{->>}[rr]^-{p_W^*}&& \hom_\A(T,W)\ar[r]&\Ext^1_\A(T,W')\ar[r]&\Ext^1_\A(T,T^{(\hom_\mathcal{A}(T,W))})=0,
}
\] 
where the last term on the right is $=0$ by condition (T.1). Furthermore, $\Ext_\A^k(T,W')=0$, for all $k\geq 2$, as shown by the following exact sequence:
\[
\xymatrix@C=15pt{
0=\Ext^{k-1}_\A(T,W)\ar[r]&\Ext^k_\A(T,W')\ar[r]&\Ext^k_\A(T,T^{(\hom_\mathcal{A}(T,W))})=0.
}
\] 
Thus, we have shown that $W'\in T^{\perp_{>0}}=\Y$. As a consequence, $\Ext_\mathcal{A}^1(W,W')=0$, and so the exact sequence $0\to W'\to T^{(\hom_\mathcal{A}(T,W))}\rightarrow W\to 0$ has to split. In particular, $W\in\text{Add}(T)$ as desired.

\smallskip
(3) Observe that, as in the proof of part (1), for each $A\in \A$ there is an exact sequence of the form \eqref{eq_Y_cores_pf_lem} such that $Y\in\mathcal{Y}$, $T^k\in\Add(T)$ (for all $k=1,\dots,n$), and  $X^{i-1}:=\Ker(d^{i-1})\in\mathcal{X}$ (for all $i=2,\dots,n$).
Moreover, $0\rightarrow X^{i-1}\rightarrow T^{i-1}\rightarrow X^{i}\rightarrow 0$ is a right $(\X,\Y)$-approximation (of $X^{i-1}$), for all $i=2,\dots,n$ and, therefore, the exact sequence in \eqref{eq_Y_cores_pf_lem} is a relative $\mathcal{Y}$-injective resolution of $A$, showing that $\Y\codim A\leq n$, as desired.
\end{proof}

In the following theorem we give a  characterization of the tilting cotorsion pairs:

\begin{thm} \label{propXX}
Let $\mathcal{A}$ be an (Ab.4) Abelian category such that $\D(\A)$ is locally small, and let $(\mathcal{X},\mathcal{Y})$ be a complete,  hereditary cotorsion pair in $\mathcal{A}$. If there is $W\in \A$ such that $\Add(W)=\mathcal{X}\cap\mathcal{Y}=:\mathcal{W}$, and $\pdim_\A(W)=n\leq m\in \N$, then the following are equivalent:
\begin{enumerate}
\item $\mathcal{Y}\text{-}\gdim(\mathcal{A})\leq m$;
\item $\Ext_\mathcal{A}^{m+1}(X',X)=0$, for all $X,X'\in\mathcal{X}$;
\item there is a subcategory $\mathcal{P}\subseteq \X$ which is generating in $\mathcal{A}$ and such that, for each $P\in\mathcal{P}$, there is an exact sequence $0\rightarrow P\rightarrow W^0\rightarrow \cdots \rightarrow W^{m}\rightarrow 0$, with $W^i\in \W$, for all $i=0,\dots,m$.
\end{enumerate}
If these conditions hold, then $W$ is $n$-tilting, and $(\X,\Y)$ is its associated $n$-tilting cotorsion pair. 
\end{thm}
\begin{proof}
The implications ``(1)$\Rightarrow$(2)'' and ``(2)$\Rightarrow$(3)''  follow by Proposition~\ref{propXX_adelanto}, while the implication ``(3)$\Rightarrow$(1)'' will follow, by Lemma~\ref{lem.n-tilting cotorsion pair}(3), as a consequence of the last part of the statement. Thus, to conclude our proof, it is enough to show that  (3) implies that $W$ is $n$-tilting. In fact, $W$ satisfies the condition (T.1) (see Definition~\ref{def.n-tilting}), since $(\X,\Y)$ is hereditary and $W\in \X\cap \Y$, while (T.2) is given by hypothesis, so we just need to take care of the axiom (T.3). We proceed by induction on $m-n\geq 0$. Indeed, if $m=n$, then clearly (3) implies (T.3), so $W$ is $n$-tilting. On the other hand, if $m>n$
 and given $P\in\mathcal{P}$, fix the following exact sequence as in (3): 
 \[
\xymatrix{
0\ar[r]& P\ar[r]^-{u}& W^0\ar[r]^-{d^0}& W^1\ar[r]^-{d^1}& \cdots\ar[r]^-{d^{m-1}}& W^m\ar[r]&0,
}
 \] 
Using again and again the fact that $\mathcal{X}$ is closed under taking kernels of epimorphisms, one gets that $X^k:=\Im(d^k)\in\mathcal{X}$, for $k=-1,0,\ldots,m$, where $X^{-1}:=P$, and $d^{-1}:=u$. 
Hence, the induced exact sequence  
$
0\rightarrow X^{m-n-2}\rightarrow W^{m-n-1}\rightarrow\cdots\rightarrow W^m\rightarrow 0
$ 
yields an element of $\Ext_\mathcal{A}^{n+1}(W^m,X^{m-n-2})=0$. Since $\mathcal{W}$ consists of $\hom_\mathcal{A}(W^m,-)$-acyclic objects, an argument similar to the one in the proof of implication ``(2)$\Rightarrow$(3)'' of Proposition~\ref{propXX_adelanto} shows that $d^{m-1}\colon W^{m-1}\rightarrow W^m$ is a retraction, so that $X^{m-2}\in\mathcal{W}$. Letting $\overline{W}^k:=W^k$ for $k=0,1,...,m-2$ and $\overline{W}^{m-1}:=X^{m-2}$, we get an exact sequence $0\rightarrow P\rightarrow\overline{W}^0\rightarrow \cdots\rightarrow\overline{W}^{m-1}\rightarrow 0$, so we can conclude by the inductive hypothesis.
\end{proof}

\begin{eg}\label{sup_noeth_eg}
Let $\mathcal{A}$ be a locally Noetherian Grothendieck category, and suppose that 
\[
\sup\{\pdim_\mathcal{A}(I):I\text{ is an indecomposable injective}\}=n>0,
\] 
and that there is a set $\mathcal{S}$ of generators,  each of which has injective dimension $\leq n$.  A coproduct $W$ of all the indecomposable injectives, one for each isoclass, is an $n$-tilting object. Indeed, as $\Inj(\mathcal{A})$ is closed under coproducts,  conditions (T.1--3) of Definition~\ref{def.n-tilting} hold.  In particular, if $\mathcal{Y}:=W^{\perp_{>0}}=(\Inj(\mathcal{A}))^{\perp_{>0}}$, Corollary~\ref{cor.finite I-global dimension}  and Lemma~\ref{lem.n-tilting cotorsion pair} tell us  that $\mathcal{A}$ is  (Ab.4$^*$)-$\mathcal{Y}$-$n$.
\end{eg}

Observe that any locally Noetherian Grothendieck category  $\mathcal{A}$ whose homological dimension is $n$ (i.e., $\Ext_\mathcal{A}^n(-,-)\neq 0=\Ext_\mathcal{A}^{n+1}(-,-)$) satisfies the conditions of the above example. In particular, we  get the following case, where we can completely describe the classes $\X$ and $\Y$:

\begin{eg}
Let $R$ be a $n$-Gorenstein ring, $\mathcal{A}:=\Mod\text{-}R$ and define $(\X,\Y)$ as in Example~\ref{sup_noeth_eg}. Then, the subcategory $\mathcal{Y}$ is precisely that of the Gorenstein injectives, and $\mathcal{X}:={}^{\perp_1}\mathcal{Y}$ is the subcategory of modules of finite injective (=finite projective) dimension (see \cite[Example~8.13]{GT}). 
\end{eg}

\begin{rmk}
If $\mathcal{A}$ is an (Ab.4$^*$) Abelian category such that $\mathcal{D}(\mathcal{A})$ is locally small, then $\A^{\mathrm{op}}$ satisfies the  hypotheses listed at the beginning of this subsection. This allows one to define an (Ab.4)-$\mathcal{X}$-$n$ condition and to relate it to $n$-cotilting objects and $n$-cotilting cotorsion pairs.
\end{rmk}

\subsection{Examples and applications for categories of quasi-coherent sheaves}\label{geometry_subs} 

In this final subsection we show that the Grothendieck category $\Qcoh(X)$, where $X$ is a scheme, is (Ab.4$^*$)-$n$ (for some $n$), whenever $X$ is quasi-compact and semi-separated (see Theorem~\ref{thm_semi_quasi_prods}).

The idea of the following technical lemma comes from \cite[Lemma~3.2 and Remark~3.3]{HX}. Here we prefer to give a more elementary inductive proof, avoiding the use of spectral sequences. We refer to \cite{HX} for other examples of (Ab.4$^*$)-$n$ Grothendieck categories in geometric contexts.  

\begin{lem} \label{resolution_implies_ab4*}
Let $\A$ be a complete Abelian category with enough injectives, and take a finite coresolution of the identity $\id_\A\colon\A\to\A$, i.e., an exact sequence of endofunctors of $\A$ of the form:
\begin{equation}\label{exact_sequence_of_functors_eq}
0\overset{}{\Longrightarrow} \id_\A\overset{\alpha_{0}}{\Longrightarrow} F_1\overset{\alpha_{1}}{\Longrightarrow} \cdots \overset{\alpha_{n}}{\Longrightarrow} F_{n+1}\overset{}{\Longrightarrow}0, \qquad \text{for some $n\geq 0$.}
\end{equation}
If $\prod^{(k)}_{J}\circ F_i=0$, for any set $J$,  $i=1,\ldots,n+1$, and all  $k\geq1$,  then $\A$ is (Ab.4$^*$)-$n$.
\end{lem}
\begin{proof} 
For each $i=1,\dots,n$, let $K_i:=\ker(\alpha_i)\colon\A\to\A$, i.e., $K_i(A)=\ker((\alpha_i)_A)$, for all $A\in \A$, and let $K_{n+1}:=F_{n+1}$. The exactness of \eqref{exact_sequence_of_functors_eq} allows us to build the following  short exact sequences:
\begin{equation}
\tag{$\dag_i$}    0 \Longrightarrow K_{i} \Longrightarrow F_{i} \Longrightarrow K_{i+1} \Longrightarrow 0,
\end{equation}
for all $i=1,\dots,n$ (of course, in ($\dag_1$), we have $K_1=\id_\A$). Given a set of objects $\{A_j\}_J\subseteq \A$, consider, for each $j\in J$, the short exact sequence $0\to K_i(A_j)\to F_i(A_j)\to K_{i+1}(A_j)\to 0$, obtained evaluating ($\dag_i$) at $A_j$. Taking products, we get the following long exact sequences:
\begin{align}
\tag{$\ddag_i$}{\resizebox{0.75\hsize}{!}{$0\to \prod_J K_{i}(A_j) \to\prod_J F_{i}(A_j) \to \prod_J K_{i+1}(A_j) \to \prod^{(1)}_J K_{i}(A_j) \to \prod^{(1)}_J F_{i}(A_j)-$}} \\
\notag{\resizebox{0.988\hsize}{!}{$\to \prod^{(1)}_JK_{i+1} (A_j) \to \prod^{(2)}_J K_{i}(A_j) \to \prod^{(2)}_J F_{i}(A_j) - \cdots
\to \prod^{(k-1)}_J K_{i+1}(A_j) \to \prod^{(k)}_J K_{i}(A_j) \to \prod^{(k)}_J F_{i}(A_j)   \cdots$}}
\end{align}
for all $i=1,\dots,n$. By hypothesis,  $\prod^{(k)}_{J} F_i(A_j)=0$ (for all $k\geq 1$ and $i=1,\ldots,n+1$). Hence, by the exactness of ($\ddag_n$), also $\prod^{(k)}_{J} K_n(A_j)=0$, for all $k\geq 2$, as it fits in between of the two vanishing objects $\prod^{(k-1)}_{J} F_{n}(A_j)=0=\prod^{(k)}_{J} F_{n+1}(A_j)=\prod^{(k)}_{J} K_{n+1}(A_j)$. Similarly, for $1\leq i<n$, the exactness of ($\ddag_i$), plus the conditions $\prod^{(k)}_{J} F_i(A_j)=0=\prod^{(k-1)}_{J} K_{i+1}(A_j)$, for all $k\geq n-i+2$, imply that  $\prod^{(k)}_{J} K_{i}(A_j)=0$, for all $k\geq n-i+2$. In particular, taking $i=1$, this gives $\prod^{(k)}_{J} A_j=\prod^{(k)}_{J} K_{1}(A_j)=0$ (as $K_1=\id_\A$), for all $k\geq n+1$. Thus, $\A$ is (Ab.4$^*$)-$n$, as desired. 
\end{proof}

\begin{thm}\label{thm_semi_quasi_prods}
Let $X$ be a semi-separated, quasi-compact scheme that admits an affine open cover of the form $\{V_1,\dots,V_{n+1}\}$ (for some $n\geq0$). Then,  $\Qcoh(X)$ is (Ab.4$^*$)-$n$.
\end{thm}
\begin{proof} 
$X$ being semi-separated is equivalent to its family of open affines being closed under finite intersections (see \cite[Section~2]{Leo_fsh}). Hence, the inclusion of an open affine $\iota_V\colon V\hookrightarrow X$ is an affine morphism of schemes and, therefore, both  $\iota_V^*=(-)_{\restriction V}\colon \Qcoh(X)\to \Qcoh(V)\cong \Mod \text{-}{\O_V}$ and its right adjoint $(\iota_V)_*\colon \Mod \text{-}{\O_V}\to \Qcoh(X)$ are exact: for the former just use that an open immersion is always a flat morphism of schemes, while for the latter we refer to \cite[\url{https://stacks.math.columbia.edu/tag/01XC}{Lemma~01XC}]{stacks-project}. Then $(\iota_V)_*$  preserves products (being a right adjoint), it sends injective objects to injective objects (as its left adjoint is exact), and it also commutes with cohomology (since it is exact): this shows that $(\iota_V)_*$ commutes with $\prod_J^{(k)}$ (for all $k\geq 0$, and any set $J$). Of course, the same holds if we take instead the coproduct (i.e., the ``disjoint union'') of a finite number of affine opens in $X$.

For $i=1,\ldots,n+1$, let $U_i := \bigsqcup \{V_{j_1,\dots,j_i}:=V_{j_1}\cap \ldots \cap V_{j_i}: 1\leq j_1<\dots<j_i\leq n+1\}$, denote by $\varphi_i:=\bigsqcup \iota_{V_{j_1,\dots,j_i}}\colon U_i\to X$ the natural map induced by the inclusions, and define an endofunctor $F_i:= (\varphi_i)_* \circ \varphi^*_i \colon \Qcoh(X)\to \Qcoh(X)$. Consider the following exact sequence of endofunctors of $\Qcoh(X)$, which gives the usual (ordered) \v Cech coresolution with respect to the cover  $\{V_1,\dots,V_{n+1}\}$ (see \cite[\url{https://stacks.math.columbia.edu/tag/01FG}{Tag 01FG}]{stacks-project}):
$
0\overset{}{\Longrightarrow} \id_\A\overset{}{\Longrightarrow} F_1\overset{}{\Longrightarrow} \cdots \overset{}{\Longrightarrow} F_{n+1}\overset{}{\Longrightarrow}0,
$
and observe that, for any set $J$, any $k\geq 1$, and any $i=1,\dots, n+1$, the following holds:
\[
\textstyle\prod_J^{(k)}\circ F_i=\prod_J^{(k)}\circ (\varphi_i)_*\circ \varphi_i^* = (\varphi_i)_*\circ \left(\prod_J^{(k)}\circ \varphi_i^*\right) = (\varphi_i)_*\circ 0 = 0,
\]
where the first equality comes from the definition of $F_i$, the second one is discussed in the first half of the proof, and third one holds since $\Qcoh(U_i)\cong \Mod \text{-}{\O_{U_i}}\cong \prod\Mod \text{-}{\O_{V_{j_1,\dots,j_i}}}$ is (Ab.4$^*$), being equivalent to a category of modules. One can now conclude using Lemma~\ref{resolution_implies_ab4*}.
\end{proof}

Under the hypotheses of the above theorem,  Positselski \cite{positselski2024roos} proved that $\Qcoh(X)$ has a generator of projective dimension $\leq n+1$, deducing that it is  (Ab.4$^*$)-$n+1$. He then asked in \cite[Question~3.14]{positselski2024roos} if this bound is sharp: his question is completely answered by the above theorem.

\bibliographystyle{abbrv.bst}
\bibliography{refs}

\vspace{0.2cm}
---------------------------------------

\end{document}